\documentclass[11pt]{amsart}

\usepackage{todonotes}

\usepackage{amsfonts, amssymb, amscd}
\numberwithin{equation}{section}

\usepackage[symbol]{footmisc}

\usepackage{bm}
\usepackage{verbatim}
\usepackage{amssymb}
\usepackage{mathrsfs}
\usepackage{graphicx}
\usepackage{tikz-cd}
\usepackage{subcaption}
\usepackage{listings}
\usepackage{subfiles}
\usepackage[toc,page]{appendix}
\usepackage{mathtools}
\usepackage{comment}
\usepackage{enumerate}
\usepackage{enumitem}

\usepackage{graphicx}
\graphicspath{{images/}}

\usepackage{appendix}
\usepackage{hyperref}
\newcommand{\bQ}{\mathbb{Q}}
\newcommand{\bP}{\mathbb{P}}

\newcommand{\cO}{\mathcal{O}}
\newcommand{\oE}{\overline{E}}
\newcommand{\cF}{\mathcal{F}}
\newcommand{\bZ}{\mathbb{Z}}
\newcommand{\bb}{\bm{b}}
\newcommand{\oY}{\overline{Y}}
\newcommand{\oL}{\overline{L}}
\newcommand{\cI}{\mathcal{I}}
\newcommand{\ind}{\mathrm{ind}}
\newcommand{\Spec}{\mathrm{Spec}}
\newcommand{\id}{\mathrm{id}}

\newcommand{\Qq}{\mathbb{Q}}

\newcommand{\Rr}{\mathbb{R}}

\newcommand{\Zz}{\mathbb{Z}}

\newcommand{\Nn}{\mathbb{N}}

\newcommand{\Span}{\operatorname{Span}}

\newcommand{\vol}{\operatorname{vol}}
\newcommand{\Center}{\operatorname{center}}

\newcommand{\mld}{{\rm{mld}}}

\newcommand{\Weil}{\operatorname{Weil}}
\newcommand{\lct}{\operatorname{lct}}
\newcommand{\LCT}{\operatorname{LCT}}
\newcommand{\CR}{\operatorname{CR}}

\newcommand{\Supp}{\operatorname{Supp}}
\newcommand{\Diff}{\operatorname{Diff}}
\newcommand{\mult}{\operatorname{mult}}
\newcommand{\Rct}{\operatorname{Rct}}
\newcommand{\RCT}{\operatorname{RCT}}

\newcommand{\LCP}{\mathcal{LCP}}
\newcommand{\Oo}{\mathcal{O}}
\newcommand{\Ii}{{\Gamma}}

\newcommand{\reg}{\mathrm{reg}}
\newcommand{\creg}{\mathrm{creg}}

\newcommand\FT{{\rm{FT}}}
\newcommand{\crt}{{\rm{crt}}}
\newcommand{\CRT}{{\rm{CRT}}}
\newcommand{\Coeff}{{\rm{Coeff}}}

\newtheorem{thm}{Theorem}[section]
\newtheorem{conj}[thm]{Conjecture}
\newtheorem{cor}[thm]{Corollary}
\newtheorem{lem}[thm]{Lemma}
\newtheorem{prop}[thm]{Proposition}
\newtheorem{exprop}[thm]{Example-Proposition}
\newtheorem{claim}[thm]{Claim}

\theoremstyle{definition}
\newtheorem{defn}[thm]{Definition}
\newtheorem{ques}[thm]{Question}
\theoremstyle{definition}
\newtheorem{rem}[thm]{Remark}
\newtheorem{remdef}[thm]{Remark-Definition}
\newtheorem{ex}[thm]{Example}

\newtheorem*{notation}{Notation ($\star$)}

\theoremstyle{definition}

\begin{document}

\title{ACC for minimal log discrepancies of exceptional singularities}

\author{Jingjun Han, Jihao Liu, and V.V. Shokurov\\	(W\MakeLowercase{ith an} A\MakeLowercase{ppendix by} Y\MakeLowercase{uchen} L\MakeLowercase{iu})}

\address{Department of Mathematics, Johns Hopkins University, Baltimore, MD 21218, USA}
\email{jhan@math.jhu.edu}

\address{Department of Mathematics, The University of Uath, Salt Lake City, UT 84112, USA}
\email{jliu@math.utah.edu}

\address{Department of Mathematics, Johns Hopkins University, Baltimore, MD 21218, USA}
\email{shokurov@math.jhu.edu}
\address{Department of Algebraic Geometry, Steklov Mathematical Institute of Russian Academy of Sciences, Moscow, 119991, Russia}
\email{shokurov@mi-ras.ru}

\begin{abstract}
We prove the existence of $n$-complements for pairs with DCC coefficients and the ACC for minimal log discrepancies of exceptional singularities. In order to prove these results, we develop the theory of complements for real coefficients. We introduce $(n,\Gamma_0)$-decomposable $\Rr$-complements, and show its existence for pairs with DCC coefficients. 

\end{abstract}
\date{\today}

\maketitle
\pagestyle{myheadings}\markboth{\hfill  J. Han, J. Liu and V.V. Shokurov \hfill}{\hfill ACC for minimal log discrepancies of exceptional singularities\hfill}

\tableofcontents


\section{Introduction} 
We work over the field of complex numbers $\mathbb C$.

The minimal log discrepancy (mld for short), which was introduced by the third author, plays a fundamental role in birational geometry. It not only characterizes the singularities of varieties but is also closely related to the minimal model program. In \cite{Sho04}, the third author proved that the conjecture on termination of flips follows from two conjectures on mlds: the ascending chain condition (ACC) conjecture for mlds and the lower-semicontinuity (LSC) conjecture for mlds. In this paper, we focus on the ACC conjecture for mlds.

\begin{conj}[{\cite[Problem 5]{Sho88}}, ACC for mlds]\label{conj: ACC for mlds}Let $d$ be a positive integer and $\Ii\subset[0,1]$ a set which satisfies the descending chain condition (DCC). Then the set
	$$\mld(d,\Ii):=\{\mld(X\ni x,B)\mid (X\ni x,B) \text{ is lc}, \dim X=d, B\in \Ii\}$$
	satisfies the ACC. Here $B\in\Ii$ means that the coefficients of $B$ belong to the set $\Ii$.
\end{conj}

 Conjecture \ref{conj: ACC for mlds} is known for surfaces by Alexeev \cite{Ale93} and the third author \cite{Sho91}, and for toric pairs by Borisov \cite{Bor97} and Ambro \cite{Amb06}. When $\Ii$ is a finite set, Conjecture \ref{conj: ACC for mlds} is known for a fixed germ by Kawakita \cite{Kaw14} and for three-dimensional canonical pairs by Nakamura \cite{Nak16}. For other related results, we refer the readers to \cite{Sho96,Sho04,Kaw11,Kaw18,Jia19}. 
 
In this paper, we show that Conjecture \ref{conj: ACC for mlds} holds for exceptional singularities. This is the first result for non-toric varieties regarding Conjecture \ref{conj: ACC for mlds} for any DCC set $\Ii$ in arbitrary dimensions. Exceptional singularities were introduced by the third author \cite{Sho92,Sho00} in the study of flips and complements, and have been further studied by many people. We refer to \cite{MP99,Pro00,PS01,IP01,Kud02,CPS10,CS11,Sak12,CS14,Sak14} for related references. 

In his proof of of the BBAB (Birkar--Borisov--Alexeev--Borisov) theorem, Birkar divided Fano varieties into two separated classes: exceptional varieties and non-exceptional varieties \cite{Bir16,Bir19}. Roughly speaking, for non-exceptional varieties, he proved the theorem by creating non-klt centers and using induction. Thus the main difficulty comes from the exceptional case. As klt singularities correspond to Fano varieties (cf. Lemma \ref{lem: existence plt blow up}) and exceptional singularities correspond to exceptional varieties (see  
Theorem \ref{thm: local global exceptional correspondence}), the study of exceptional singularities is expected to be important in the study of klt singularities. 

\medskip 

\noindent\textbf{ACC for mlds of exceptional singularities.} We say that $(X\ni x,B)$ is an \emph{exceptional singularity} if $(X\ni x,B)$ is an lc germ, and for any $\mathbb R$-divisor $G\geq 0$ on $X$ such that $(X\ni x,B+G)$ is lc, there exists at most one prime $\bb$-divisor $E$ over $X\ni x$ such that $a(E,X,B+G)=0$ (see Definition \ref{defn: exc sing}). 

\medskip

Our first main result is the ACC for mlds of exceptional singularities:
\begin{thm}\label{thm: ACC mld exc} Let $d$ be a positive integer and $\Ii\subset [0,1]$ a DCC (resp. finite) set. Then there exists an ACC set (resp. an ACC set whose only possible accumulation point is $0$) $\Ii'$ and a positive real number $\epsilon$ depending only on $d$ and $\Ii$ satisfying the following. Assume that 
	\begin{itemize}
		\item $(X\ni x,B)$ is an exceptional singularity of dimension $d$, and
		\item $ B\in\Ii$,
	\end{itemize}
	then 
	\begin{enumerate}
	\item $\mld(X\ni x,B)\in\Ii'$, and
	\item if $\dim X\geq 2$ and $0<\mld(X\ni x,B)\le\epsilon$, then $\mld(X\ni x,B)=a(E,X,B)$, where $E$ is the unique reduced component  (see Definition \ref{defn: reduced component}) of $(X\ni x,B)$.
	\end{enumerate}
\end{thm}
It would be interesting to ask about the LSC conjecture for mlds of exceptional singularities, and whether the exceptional property of a singularity is a closed condition or not.

\medskip
In order to study exceptional singularities, we introduce a new category of singularities in this paper: the \emph{singularities admitting an $\epsilon$-plt blow-up}. Assume that $(X\ni x,B)$ is a klt germ. We recall that a plt blow-up $f: Y\rightarrow X$ of $(X\ni x,B)$ is an extraction of a unique prime divisor $E$ over $X\ni x$, such that $-E$ is ample over $X$ and $(Y,B_Y+E)$ is plt near $E$, where $B_Y$ is the strict transform of $B$ on $Y$ (see Definition \ref{defn: reduced component}). In this paper, $E$ is called \emph{the reduced component} of $f$. For any non-negative real number $\epsilon$, if $(E,\Diff_E(B_Y))$ is $\epsilon$-klt, then $f$ is called an \emph{$\epsilon$-plt blow-up} of $(X\ni x,B)$. If such $f$ exists, we say that $(X\ni x,B)$ \emph{admits an $\epsilon$-plt blow-up}. We also say that $(X\ni x,B)$ is a singularity  \emph{admitting an $\epsilon$-plt blow-up}.

When $\dim X\ge 2$, there are many singularities admitting an $\epsilon$-plt blow-up. For example, any fixed klt germ admits an $\epsilon$-plt blow-up for some positive real number $\epsilon$, and any smooth point admits an $\epsilon$-plt blow-up for every $0\leq\epsilon<1$. We will show that any klt exceptional singularity with DCC coefficients always admits an $\epsilon$-plt blow-up for some positive real number $\epsilon$ depending only on the DCC set and the dimension (see Lemma \ref{lem: exp sing epsilon plt}). In particular, Theorem \ref{thm: ACC mld exc} is a special case of Theorem \ref{thm: acc mld dcc e0lc eplt}.

\begin{thm}\label{thm: acc mld dcc e0lc eplt}
	Let $d$ be a positive integer, $\epsilon$ a positive real number, and $\Ii\subset [0,1]$ a DCC (resp. finite) set. Assume that $(X\ni x,B)$ is a klt germ of dimension $d$, such that
	\begin{itemize}
				\item $(X\ni x,B)$ admits an $\epsilon$-plt blow-up $f:Y\rightarrow X$, and
						\item $B\in \Ii$,
	\end{itemize}
	then 
	\begin{enumerate}
	\item $\mld(X\ni x,B)$ belongs to an ACC set (resp. an ACC set whose only possible accumulation point is $0$) depending only on $d,\epsilon$ and $\Ii$, and
	\item if $\mld(X\ni x,B)\leq\epsilon$, then $\mld(X\ni x,B)$ is attained at the reduced component of $f$.
	\end{enumerate}
\end{thm}


 We will show that the log discrepancy of the reduced component, $a(E,X,B)$, belongs to an ACC set (see Proposition \ref{prop: acc ld kc}). In particular, $\mld(X\ni x,B)$ is bounded from above. Thus Theorem \ref{thm: acc mld dcc e0lc eplt} follows immediately from Theorem \ref{thm: acc ld dcc coefficient bdd}. 

\begin{thm}\label{thm: acc ld dcc coefficient bdd}
	Let $d$ be a positive integer, $\epsilon_0,\epsilon$ two positive real numbers, and $\Ii\subset [0,1]$ a DCC set. Let $\mathcal{LD}(d,\epsilon_0,\epsilon,\Ii)$ be the set of all $a(F,X,B)$, where
	\begin{enumerate}
		\item $(X\ni x,B)$ is an $\epsilon_0$-lc germ,
		\item $B\in \Ii$,
		\item $(X\ni x,B)$ admits an $\epsilon$-plt blow-up, and
		\item $F$ is a prime divisor over $X\ni x$.
	\end{enumerate}
	Then $\mathcal{LD}(d,\epsilon_0,\epsilon,\Ii)$ does not have any accumulation point from below, i.e. for any positive real number $M$, the set
	$$\mathcal{LD}(d,\epsilon_0,\epsilon,\Ii)\cap[0,M]$$
	satisfies the ACC. 
\end{thm}

\begin{rem}
	Actually, there exists a strictly increasing sequence $a_i\in \mathcal{LD}(d,\epsilon_0,\epsilon,\Ii)$, such that $\lim_{i\to +\infty}a_i=+\infty$. 
\end{rem}

\noindent\textbf{Boundedness of local algebraic fundamental groups.} As one of the key steps in the proof of Theorem \ref{thm: acc ld dcc coefficient bdd}, we show that if an $\epsilon_0$-lc germ admits an $\epsilon$-plt blow-up, then the order of the local algebraic fundamental
group is bounded. We remark that such kind of singularities are not bounded even in the analytic sense \cite[Example 3.4]{HLM19}.


Recall that the local algebraic fundamental group $\hat{\pi}_1^{loc}(X,x)$ is the pro-finite completion of
 $\pi_1^{loc}(X,x):=\pi_1(L(X\ni x))$, where $L(X\ni x)$ is the link of the singularity $X\ni x$.
\begin{thm}\label{thm: bdd local index exceptional}
	Let $d$ be a positive integer and $\epsilon_0,\epsilon$ two positive real numbers. Then there exists a positive integer $N$ depending only on $d, \epsilon_0$ and $\epsilon$ satisfying the following. Assume that
	\begin{enumerate}
	    \item $(X\ni x,B)$ is an $\epsilon_0$-lc germ of dimension $d$, and
	    \item $(X\ni x,B)$ admits an $\epsilon$-plt blow-up.
	\end{enumerate} 
	Then the order of the local algebraic fundamental group $\hat{\pi}_1^{loc}(X,x)$ is bounded from above by $N$. In particular, $N!D$ is Cartier near $x$ for any $\mathbb Q$-Cartier Weil divisor $D$ on $X$.
\end{thm}
Xu proved Theorem \ref{thm: bdd local index exceptional} when $(X\ni x,B)$ is a fixed klt germ \cite[Theorem 1]{Xu14}. The third author \cite{Sho00} proved Theorem \ref{thm: bdd local index exceptional} and Theorem \ref{thm: ACC mld exc} when $\dim X=3$, and the coefficients of $B$ belong to the standard set $\{1-\frac{1}{n}\mid n\in\mathbb{N}^{+}\}\cup\{0\}$. His proof also works for the standard set in higher dimensions by assuming the BBAB theorem, see \cite{Mor18} for some special cases. We remark that in Theorem \ref{thm: bdd local index exceptional}, there is no assumption on the coefficients of $B$.

\begin{rem}\label{exp: mld exc sing goes to 0}
Both assumptions (1) and (2) of Theorem \ref{thm: bdd local index exceptional} are necessary even for surfaces: for example, 
\begin{itemize}
    \item Du Val singularities $A_n$ are $1$-lc, and
    \item the blow-up of the vertex of the cone over a rational curve of degree $n$ is a $\frac{1}{2}$-plt blow-up,
\end{itemize}
but the Cartier indices of all Weil divisors are not bounded from above when $n\to+\infty$ in either case.
\end{rem}

\medskip
\noindent\textbf{Existence of complements.} The third author introduced the theory of complements when he investigated log flips of threefolds \cite{Sho92}. It is related to the boundedness of varieties and singularities of linear systems. The theory of complements is further developed in \cite{Sho00,PS01,PS09,Bir19}. These researches were mainly focused on pairs with finite rational coefficients. 

In this paper, we develop the theory of complements for arbitrary DCC coefficients. This is not only crucial to the proof of Theorem \ref{thm: acc ld dcc coefficient bdd}, but is also important to topics related to the minimal model program (cf. \cite{Liu18}). Hacon informed us that it is expected that the theory of complements in this case may play an important role in future applications, e.g. towards the termination of flips both in characteristic zero and positive characteristics. 


\medskip

We show the existence of $n$-complements for pairs with DCC coefficients, which is a generalization of  \cite[Theorem 1.1, Theorem 1.7, Theorem 1.8]{Bir19} from finite rational coefficients to arbitrary DCC coefficients.
\begin{thm}\label{thm: dcc existence n complement}
	Let $d,p$ be two positive integers and $\Ii\subset [0,1]$ a DCC set. Then there exists a positive integer n divisible by $p$, depending only on $d,p$ and $\Ii$ satisfying the following. Assume that $(X,B)$ is an lc pair of dimension $d$ and $X\to Z$ is a contraction, such that
	\begin{itemize}
		\item $X$ is of Fano type over $Z$, 
		\item $B\in\Ii$, and
		\item $(X/Z,B)$ is $\Rr$-complementary.
	\end{itemize}
	Then there exists an $n$-complement $(X/Z\ni z,B^+)$ of $(X/Z\ni z,B)$ for any point $z\in Z$. Moreover, if
	$\Span_{\Qq_{\geq 0}}(\bar\Ii\backslash\Qq)\cap\mathbb Q\backslash\{0\}=\emptyset,$ then we may pick $B^{+}\geq B$, that is, $(X/Z\ni z,B^+)$ is a monotonic $n$-complement of $(X/Z\ni z,B)$. Here $\bar\Ii$ stands for the closure of $\Ii$ in $\Rr$.

\end{thm}


It is surprising that in the above theorem, we can choose a unique $n:=n(d,p,\Ii)$ as opposed to a finite set of possible $n$'s as originally conjectured by the third author. 

When $-(K_X+B)$ is nef over $Z$, Theorem \ref{thm: dcc existence n complement} was proved by the third author \cite{Sho00} in dimension 2 when $\Ii\subset \Qq$ is the standard set, by Prokhorov and the third author \cite{PS09} in dimension 3 when $\Ii\subset\Qq$ is a hyperstandard set, and finally by Birkar \cite[Theorem 1.7, Theorem 1.8]{Bir19} when $\Ii\subset\Qq$ is a hyperstandard set in any given dimension.

We remark that in this paper, instead of adopting the assumption ``$-(K_X+B)$ is nef over $Z$'' as in the previous works, we use the weaker assumption ``$(X/Z,B)$ is $\Rr$-complementary'' introduced by the third author \cite[1.3]{Sho00}. We believe the latter assumption is the proper setting in the study of boundedness of complements for both Fano type and non-Fano type varieties, see \cite{Sho20} for more details. We also remark that in Theorem \ref{thm: dcc existence n complement} we cannot expect that $B^{+}\ge B$ even in dimension 1 (cf. Example \ref{ex: irr complement}). 



If we replace $\Ii$ with $[0,1]$ in Theorem \ref{thm: dcc existence n complement}, then the theorem does not hold (cf. Example \ref{ex: ex acc no complement}), and the theorem does not hold even we weaken the conclusion ``a positive integer $n$ divisible by $p$'' to ``$n$ belonging to a finite set'' (the boundedness of complements) (cf. \cite{Sho20}). By Theorem \ref{thm: dcc existence n complement}, it is not hard to show that if the number of irreducible components of $B$ is bounded, then the boundedness of complements holds. The third author \cite{Sho20} announced a proof of the boundedness of complements under one of the following conditions: either $(X/Z\ni z,B)$ has a klt $\Rr$-complement, or $\Ii\cap\Qq$ is a DCC set. His proof uses a quite different approach.


\medskip






In this paper, we introduce $(n,\Ii_0)$-decomposable $\Rr$-complements. The proofs of Theorem \ref{thm: acc ld dcc coefficient bdd} and Theorem \ref{thm: dcc existence n complement} heavily rely on Theorem \ref{thm: existence ni1i2 complement}, the existence of $(n,\Ii_0)$-decomposable $\Rr$-complements for pairs with DCC coefficients. 

Recall that an $\Rr$-complement $(X/Z\ni z,B^+)$ of $(X/Z\ni z,B)$ is a pair $(X,B^+)$ such that $(X,B^+)$ is lc over a neighborhood of $z$ and $B^+\geq B$.

\begin{defn}\label{defn: nI1I2complmenet}
	Let $(X,B)$ be a pair, $X\to Z$ a contraction, $z\in Z$ a point, and $\Ii_0=\{a_1,\ldots,a_k\}\subset(0,1]$ a finite set, such that $\sum_{i=1}^ka_i=1$. We say that $(X/Z\ni z,B^+)$ is an \emph{$(n,\Ii_0)$-decomposable $\Rr$-complement} of $(X/Z\ni z,B)$ if 
	\begin{enumerate}
	\item $(X/Z\ni z,B^+)$ is an $\Rr$-complement of $(X/Z\ni z,B)$,
	\item $\sum_{i=1}^ka_iB_i^{+}=B^+$ for some boundaries $B_1^{+},\dots,B_k^{+}$, and
	\item each $(X/Z\ni z,B_i^{+})$ is an $n$-complement of itself.
	\end{enumerate}
\end{defn}

\begin{thm}\label{thm: existence ni1i2 complement}
	Let $d$ be a positive integer and $\Ii\subset[0,1]$ a DCC set. Then there exist a positive integer $n$ and a finite set $\Ii_0\subset(0,1]$ depending only on $d$ and $\Ii$ satisfying the following. Assume that $(X,B)$ is an lc pair of dimension $d$ and $X\to Z$ is a contraction such that 
	\begin{enumerate}
		\item $X$ is of Fano type over $Z$,
		\item $B\in\Ii$, and
		\item $(X/Z,B)$ is $\Rr$-complementary.
	\end{enumerate} 
	Then for any point $z\in Z$, there exists an $(n,\Ii_0)$-decomposable $\Rr$-complement $(X/Z\ni z,B^{+})$ of $(X/Z\ni z,B)$. Moreover, if $\bar\Gamma\subset\Qq$, then we may pick $\Ii_0=\{1\}$, and $(X/Z\ni z,B^{+})$ is a monotonic $n$-complement of $(X/Z\ni z,B)$.
\end{thm}

We need to use \cite[Theorem 1.7, Theorem 1.8]{Bir19} to prove Theorem \ref{thm: existence ni1i2 complement}, and thus Theorem \ref{thm: dcc existence n complement}, hence we do not give an independent proof of \cite[Theorem 1.1, Theorem 1.7, Theorem 1.8]{Bir19}. 

We remark that Theorem \ref{thm: existence ni1i2 complement} implies the ACC for $\Rr$-complementary thresholds, and the existence of uniform $\Rr$-complementary rational polytopes for Fano type varieties. See applications below. 

\medskip

\noindent\textbf{Applications}. There are several applications and by-products of our main theorems. We give a short description of them here.

\medskip

\noindent\textbf{Uniform rational polytopes for (relative) Fano type varieties}. We show the existence of uniform lc rational polytopes (see Theorem \ref{thm: Uniform perturbation of lc pairs}), and the existence of uniform $\Rr$-complementary rational polytopes for (relative) Fano type varieties (see Theorem \ref{thm: Uniform perturbation of R complements}). 

\medskip

\noindent\textbf{Log discrepancies and mlds for $\epsilon_0$-lc germs admitting an $\epsilon$-plt blow-up with finite coefficients}. For these pairs, we show that their log discrepancies form a discrete set and their minimal log discrepancies form a finite set (see Theorem \ref{thm: discrete ld plt}), a generalization of \cite[Theorem 1.1, Theorem 1.2]{Kaw11}.
\medskip


\noindent\textbf{Miscellaneous results for exceptional singularities}. We show that the ACC for $a$-lc thresholds holds for exceptional singularities (see Theorem \ref{thm: ACC aLCT exc}), and the ACC for normalized volumes holds for exceptional singularities (see Theorem \ref{thm: ACC NV exc}). We also show the correspondence between (local) exceptional singularities and (global) exceptional pairs (see Theorem \ref{thm: local global exceptional correspondence}).

\medskip

\noindent\textbf{Complete regularities}. We study the invariant \emph{complete regularity thresholds} (cf. Definition \ref{defn: crt}), which generalizes the log canonical thresholds and the $\Rr$-complementary thresholds. We show that the thresholds satisfy the ACC for (relative) Fano type varieties (see Theorem \ref{thm: crt acc}).

\medskip

\noindent\textbf{Complements for non-Fano type varieties}.
For non-Fano type varieties, we reduce the conjecture on the existence of $n$-complements for pairs with DCC coefficients $\Ii$ to the case when $\Ii$ is a finite rational set (see Theorem \ref{thm: existence n complement nft q to r}).

\medskip

\noindent\textbf{Other applications}. We show the existence of monotonic klt $n$-complements for any $\epsilon_0$-lc germ admitting an $\epsilon$-plt blow-up (see Theorem \ref{thm: existence local klt complement}). We also prove a result on accumulation points of log canonical thresholds for pairs with DCC coefficients (see Theorem \ref{thm: accmu of lct}). 


\medskip

\noindent\textit{Structure of the paper.} In Section 2, we give a sketch of the proofs of some of our main theorems. In Section 3, we introduce some notation and tools which will be used in this paper, and prove certain results. In Section 4, we prove Theorem \ref{thm: bdd local index exceptional}. In Section 5, we prove Theorem \ref{thm: existence ni1i2 complement}. In Section 6, we prove Theorem \ref{thm: dcc existence n complement}. In Section 7, we prove Theorem \ref{thm: acc ld dcc coefficient bdd}. In Section 8, we give some applications (Theorem \ref{thm: existence local klt complement}, \ref{thm: ACC aLCT exc}, \ref{thm: ACC NV exc}, \ref{thm: local global exceptional correspondence}, \ref{thm: strong acc mld for exceptional singularities}, \ref{thm: crt acc}, \ref{thm: existence n complement nft q to r}, and \ref{thm: accmu of lct}) of our main theorems. In Appendix A, we prove Theorem \ref{thm:pltind}.
\medskip

\section{Sketch of the proofs}
In this section, we give a brief account of some of the ideas of the proofs of Theorem \ref{thm: bdd local index exceptional}, Theorem \ref{thm: existence ni1i2 complement}, and Theorem \ref{thm: acc ld dcc coefficient bdd}.

\medskip

\noindent\textit{Sketch of the proof of Theorem \ref{thm: bdd local index exceptional}.} Suppose that $(X\ni x,B)$ is $\epsilon_0$-lc and $f:Y\to X$ is an $\epsilon$-plt blow-up of $(X\ni x,B)$ with the reduced component $E$. By the BBAB Theorem, $a(E,X,B)$ is bounded from above. Moreover, the normalized volume of $E$,
$$\widehat{\vol}(E,X,B):=a(E,X,B)\cdot\vol(-(K_Y+f^{-1}_*B+E)|_{E})$$
is bounded from above and from below by two positive real numbers (see Theorem \ref{thm: nv acc 2}). For any finite morphism $g: X'\ni x'\rightarrow X\ni x$ such that $g^*(K_X+B)=K_{X'}+B'$ for some $\Rr$-divisor $B'\geq 0$, we have
$$\deg g\cdot\widehat{\vol}(E,X,B)=\widehat{\vol}(E',X',B'),$$
where $E'$ is the reduced component corresponding to $E$ with respect to $g$. Since $\widehat{\vol}(E',X',B')$ is also bounded from above, $\deg g$ belongs to a finite set.


\medskip

\noindent\textit{Sketch of the proof of Theorem \ref{thm: existence ni1i2 complement}.} By applying the ideas in \cite{PS09,Bir19}, the ACC for log canonical thresholds, and the global ACC, we can reduce Theorem \ref{thm: existence ni1i2 complement} to the case when $\Ii$ is a finite set (see Theorem \ref{thm: dcc limit lc divisor} and Theorem \ref{thm: dcc limit divisor}). Next, we show the existence of uniform $\Rr$-complementary rational polytope (see Theorem \ref{thm: Uniform perturbation of R complements}), and reduce Theorem \ref{thm: existence ni1i2 complement} to the case when $\Ii$ is a finite set of rational numbers. Theorem \ref{thm: existence ni1i2 complement} immediately follows from \cite[Theorem 1.7, Theorem 1.8]{Bir19} in this case.

\medskip


\noindent\textit{Sketch of the proof of Theorem \ref{thm: acc ld dcc coefficient bdd}.} Let $f: Y\rightarrow X$ be an $\epsilon$-plt blow-up of $(X\ni x,B)$ with the reduced component $E$. For simplicity, we assume that $X$ is $\Qq$-Gorenstein and $Y$ is $\Qq$-factorial. Let $B_Y$ be the strict transform of $B$ on $Y$. By Theorem \ref{thm: existence ni1i2 complement}, there exist a positive integer $n$ and a finite set of real numbers $\Ii'\subset(0,1]$ depending only on $d$ and $\Ii$, an $\mathbb R$-divisor $B_Y^+\ge0$ on $Y$ and $\mathbb Q$-divisors $B_i\geq 0$ on $Y$, such that
\begin{itemize}
\item $B_Y^+=\sum_ia_iB_i$, where $\sum_ia_i=1$ and each $a_i\in\Ii'$, 
\item $B_Y^+\geq B_Y$, and
\item $(Y/X\ni x,B_i+E)$ is a monotonic $n$-complement of itself for any $i$.
\end{itemize}
Let $b:=a(F,Y,B_Y+E)$ and $b':=a(F,Y,B_{Y}^++E)$. Since 
$$M\ge a\ge a(F,Y,B_Y^++E)=\sum_{i} a_ia(F,Y,B_i+E)\ge 0,$$
$b'$ belongs to a finite set.  

\medskip

Our next goal is to show that $b$ belongs to an ACC set. Since $b=b'+\mult_{F}(B_{Y}^+-B_Y)$ and the coefficients of $B_{Y}^+-B_Y$ belong to an ACC set, it suffices to show that 1) the Cartier index of each component of $\Supp(B_Y^+)$ is bounded, and 2) $\mult_F(B_Y^+)$ is bounded from above.

To show 1), let $D_Y$ be an irreducible component of $\Supp(B_Y^+)$. By the BBAB Theorem, $(E,\Diff_{E}(B_Y^+))$ belongs to a bounded family and $\mult_{D_Y}B_Y^+$ belongs to a finite set. Thus the Cartier indices of $K_E+\Diff_{E}(D_Y)$ and $K_E+\Diff_{E}(0)$ are bounded from above. Therefore, the Cartier indices of $K_Y+E$ and $K_Y+E+D_Y$ near $E$ are bounded, which implies that the Cartier index of $D_Y$ near $E$ is bounded.

To show 2), since $\mult_F(D_Y)\leq \frac{Mn}{\min\{a_i\}}$ for any irreducible component $D_Y\subset\Supp(B_Y^+)\backslash\Supp(B_Y)$, it is enough to show that $\mult_F(B_Y)$ is bounded from above. Since $(E,B_{E})$ is log bounded, $\lct(E,B_{E};B_{Y}|_{E})\ge t>0$ is bounded from below away from $0$ (\cite[Theorem 1.6]{Bir16}). In particular, $(Y,(1+t)B_Y+E)$ is lc near $E$, which implies that $\mult_{F}(B_Y)$ is bounded from above.

\medskip

Since $a(E,X,B)$ belongs to an ACC set (see Proposition \ref{prop: acc ld kc}) and 
$$M\ge a=a(F,Y,B_Y+(1-a(E,X,B))E)=b+a(E,X,B)\mult_{F}E,$$
we only need to show that $\mult_FE$ belongs to a finite set. Since $f$ is also an $\epsilon$-plt blow-up of $(X\ni x,0)$, by Theorem \ref{thm: bdd local index exceptional}, the Cartier index of $K_X$ near $x$ is bounded. Therefore, the Cartier index of $K_Y+(1-a(E,X,0))E$ near $E$ is bounded, which implies that the Cartier index of $E$ is bounded, $\mult_FE$ is bounded from above, and $\mult_FE$ belongs to a finite set.

\section{Preliminaries}
\subsection{Pairs and singularities}
We adopt the standard notation and definitions in \cite{Sho92} and \cite{KM98}, and will freely use them.

\begin{defn}[Pairs and singularities]\label{defn: positivity}
	A pair $(X,B)$ consists of a normal quasi-projective variety $X$ and an $\Rr$-divisor $B\ge0$ such that $K_X+B$ is $\Rr$-Cartier. Moreover, if the coefficients of $B$ are $\leq 1$, then $B$ is called a boundary of $X$.
	
	Let $E$ be a prime divisor on $X$ and $D$ an $\mathbb R$-divisor on $X$. 	We define $\mult_ED$ to be the multiplicity of $E$ along $D$. Let $\phi:W\to X$
	be any log resolution of $(X,B)$ and let
	$$K_W+B_W:=\phi^{*}(K_X+B).$$
	The \emph{log discrepancy} of a prime divisor $D$ on $W$ with respect to $(X,B)$ is $1-\mult_{D}B_W$ and it is denoted by $a(D,X,B).$
	For any positive real number $\epsilon$, we say that $(X,B)$ is lc (resp. klt, $\epsilon$-lc, $\epsilon$-klt) if $a(D,X,B)\ge0$ (resp. $>0$, $\ge \epsilon$, $>\epsilon$) for every log resolution $\phi:W\to X$ as above and every prime divisor $D$ on $W$. We say that $(X,B)$ is plt (resp. $\epsilon$-plt) if $a(D,X,B)>0$ (resp. $>\epsilon$) for any exceptional prime divisor $D$ over $X$. 
	
	A germ $(X\ni x,B)$ consists of a pair $(X,B)$ and a closed point $x\in X$. $(X\ni x,B)$ is called an lc (resp. a klt, an $\epsilon$-lc) germ if $(X,B)$ is lc (resp. klt, $\epsilon$-lc) near $x$. $(X\ni x,B)$ is called $\epsilon$-lc at $x$ if $a(D,X,B)\geq\epsilon$ for any prime divisor $D$ over $X\ni x$ (i.e., $\Center_{X}D=x$).
\end{defn}

\begin{defn}[Prime $\bb$-divisors] Let $X$ be a normal quasi-projective variety. We call $Y$ a birational model over $X$ if there exists a projective birational morphism $Y\to X$. A prime $\bb$-divisor over $X$ is a prime divisor $E$ on a birational model $Y$ over $X$ up to the following equivalence: another prime divisor $E'$ on a birational model $Y'$ over $X$ defines the same prime $\bb$-divisor if for any common resolution $W\to Y$ and $W\to Y'$, the strict transforms of $E$ and $E'$ on $W$ coincide.
\end{defn}

\begin{defn}
	For an $\Rr$-divisor $B=\sum b_iB_i$, where the $B_i$ are the irreducible components of $B$, we define $||B||:=\max\{|b_i|\}.$
\end{defn}

\begin{defn} Let $X$ and $Z$ be two normal quasi-projective varieties. We say $f:X\to Z$ is a \emph{contraction} if $f$ is a projective morphism, and $f_{*}\Oo_X=\Oo_Z$ ($f$ is not necessarily birational). 
	
	Let $X\to Z$ be a contraction. We say $X$ is of \emph{Fano type} over $Z$ if $(X,B)$ is klt and $-(K_X+B)$ is big and nef over $Z$ for some boundary $B$.
\end{defn}
\begin{rem}
	Assume that $X$ is of Fano type over $Z$. Then we can run the MMP over $Z$ on any $\Rr$-Cartier $\Rr$-divisor $D$ on $X$ which terminates with some model $Y$ (cf. \cite[Corollary 2.9]{PS09}).
\end{rem}

\begin{defn}[Minimal log discrepancies]\label{defn: mld and alct}
	Let $(X\ni x,B)$ be an lc germ. The \emph{minimal log discrepancy} of $(X\ni x,B)$ is defined as
	$$\mld(X\ni x,B):=\min\{a(E,X,B)\mid E \text{ is a prime divisor over } X\ni x\}.$$
	If $E$ is a prime divisor over $X\ni x$ such that $a(E,X,B)=\mld(X\ni x,B)$, then we say that the minimal log discrepancy of $(X\ni x,B)$ is \emph{attained at} $E$.
\end{defn}

\begin{defn}[$a$-lc thresholds] Assume that $(X,B)$ is an lc pair (resp. $(X\ni x,B)$ is an lc germ). Suppose that $(X,B)$ is $a$-lc (resp. $a$-lc at $x$) for some $a\ge0$. The \emph{$a$-lc threshold} (resp. \emph{$a$-lc threshold at} $x$) of an $\Rr$-Cartier $\mathbb R$-divisor $G\geq 0$ with respect to $(X,B)$ (resp. $(X\ni x,B)$) is 
	$$a\text{-}\lct(X,B;G):=\sup\{c\ge 0\mid (X,B+cG) \text{ is } a\text{-lc}\}.$$ 
	$$\text{(resp. }a\text{-}\lct(X\ni x,B;G):=\sup\{c\ge 0\mid (X\ni x,B+cG) \text{ is } a\text{-lc at }x\}. \text{)}$$
	In particular, if $a=0$, we obtain the lc threshold (resp. lc threshold at $x$). For simplicity, we will use $\lct(X,B;G)$ instead of $0$-$\lct(X,B;G)$ and $\lct(X\ni x,B;G)$ instead of $0$-$\lct(X\ni x,B;G)$.
\end{defn}

\begin{defn}\label{defn: DCC and ACC}
	Let $\Ii$ be a set of real numbers. 
We denote $\bar{\Ii}$ the closure of $\Ii$.
	We say that $\Ii$ satisfies the \emph{descending chain condition} (DCC) if any decreasing sequence $a_1\ge a_2 \ge \cdots \ge a_k \ge\cdots$ in $\Ii$ stabilizes. We say that $\Ii$ satisfies the \emph{ascending chain condition} (ACC) if any increasing sequence in $\Ii$ stabilizes. 
\end{defn}

\begin{thm}[ACC for lc thresholds, {\cite[Theorem 1.1]{HMX14}}]\label{thm: acc lct}
	Fix a positive integer $d$, and DCC sets $\Ii\subset [0,1]$ and $\Ii'\subset [0,+\infty)$. Then 
	$$\LCT(\Ii,\Ii',d):=\{\lct(X,B;G)\mid \dim X=d, \,(X,B) \text{ is lc},\,B\in\Ii,G\in\Ii'\}$$
	satisfies the ACC. 
\end{thm}

\begin{thm}\label{thm: adjunction}
	Let $(X,B+S)$ be an lc pair, such that $S$ is a prime normal divisor, $ S\not\subset \Supp B$, and $B=\sum_{i} b_iB_i$, where $B_i$ are the irreducible components of $B$. Then there exists a naturally defined $\Rr$-divisor $B_S$ on $S$, such that
	\begin{enumerate}
		\item $K_S+B_S=(K_X+B+S)|_{S}$,
		\item $(X,B+S)$ is plt (resp. lc, $\epsilon$-plt) if and only if $(S,B_S)$ is klt (resp. lc, $\epsilon$-klt),
		\item if $(X,B+S)$ is plt, then for any codimension 1 point $V$ of $S$, we have
		$$\mult_VB_S=\frac{m-1+\sum_{i} n_ib_i}{m}$$
		for some non-negative integers $n_i$, where $m$ is the order of the cyclic group $\Weil(\mathcal{O}_{X,V})$, 
		\item if $(X,B+S)$ is $\epsilon$-plt, then $|\Weil(\mathcal{O}_{X,V})|\le\lfloor\frac{1}{\epsilon}\rfloor$. In particular, for any $\mathbb Q$-Cartier Weil divisor $D$ on $X$, $rD|_{S}$ is a Weil divisor, where $r=\lfloor\frac{1}{\epsilon}\rfloor!$, and   
		\item if $b_i$ belongs to a DCC set, then the coefficients of $B_S$ belong to a DCC set.
	\end{enumerate}		
	Furthermore, if $(X,B+S)$ is dlt and $K_X+\Delta+S$ is $\Rr$-Cartier for some $\Delta=\sum \delta_iB_i\geq0$ and $S \not\subset \Supp \Delta$, then there exists a naturally defined $\Rr$-divisor $\Delta_S$ on $S$, such that $K_S+\Delta_S=(K_X+\Delta+S)|_{S}$, and for any codimension 1 point $V$ of $S$, we have
	$$\mult_V\Delta_S=\frac{m-1+\sum_{i} n_i'\delta_i}{m}$$
	for some non-negative integers $n_i'$, where $m$ is the order of the cyclic group $\Weil(\mathcal{O}_{X,V})$.
\end{thm}
\begin{proof}
	(1),(2),(3),(4) follow from \cite[\S 3]{Sho92},\cite[\S 16]{Kol92},\cite{Kaw07},\cite[1.4.5]{BCHM10}, and (5) follows from \cite[Lemma 3.4.1]{HMX14}.
	
	Suppose that $(X,B+S)$ is dlt. Then there exists a small $\Qq$-factorialization $f:(X',B'+S')\to(X,B+S)$. In particular, the codimension of the exceptional locus of $f$ in $X'$ is at least 2. Thus $X$ is $\Qq$-factorial in codimension 2. In particular, $(X,S)$ is dlt in codimension 2, and the statement about $K_X+\Delta+S$ follows from \cite[\S 3]{Sho92} and \cite[\S 16]{Kol92}. 
\end{proof}

\subsection{Plt blow-ups}
\begin{defn}[Plt blow-ups]\label{defn: reduced component}
	Let $(X\ni x,B)$ be a klt germ and $\epsilon$ a positive real number.
	
	A \emph{plt blow-up} of $(X\ni x,B)$ (resp. an \emph{$\epsilon$-plt blow-up} of $(X\ni x,B)$) is a blow-up $f: Y\rightarrow X$ with the exceptional divisor $E$ over $X\ni x$, such that
	\begin{itemize}
		\item $(Y,B_Y+E)$ is plt (resp. $\epsilon$-plt) near $E$, where $B_Y$ is the strict transform of $B$ on $Y$, and
		\item $-E$ is ample over $X$.
	\end{itemize}
	The divisor $E$ is called \emph{a reduced component} of $(X\ni x,B)$, and we also call it \emph{the reduced component} of $f$.

	A \emph{$\Qq$-factorial weak plt blow-up} of $(X\ni x,B)$ (resp. $\Qq$-factorial weak \emph{$\epsilon$-plt blow-up} of $(X\ni x,B)$) is a blow-up $f: Y\rightarrow X$ with the exceptional divisor $E$, such that
	\begin{itemize}
		\item $(Y,B_Y+E)$ is $\Qq$-factorial plt (resp. $\Qq$-factorial $\epsilon$-plt) near $E$, where $B_Y$ is the strict transform of $B$ on $Y$,
		\item $-E$ is nef over $X$, 
		\item $-(K_Y+B_Y+E)|_{E}$ is big, and
		\item $f^{-1}(x)=\Supp E$. 
	\end{itemize}
	The divisor $E$ is called \emph{a weak reduced component} of the germ $(X\ni x,B)$, and we also call it \emph{the reduced component} of $f$. 
\end{defn}

\begin{lem}[{\cite[3.1]{Sho96},\cite[Proposition 2.9]{Pro00},\cite[Theorem 1.5]{Kud01},\cite[Lemma 1]{Xu14}}]\label{lem: existence plt blow up}
	Assume that $(X\ni x,B)$ is a klt germ such that $\dim X\geq 2$. Then there exists a plt blow-up of $(X\ni x,B)$.
\end{lem}

\begin{lem}\label{lem: existence qfact weak plt blow up}
	Let $(X\ni x,B)$ be a klt germ. Suppose that $(X\ni x,B)$ has an $\epsilon$-plt blow-up $f:Y\rightarrow X$, then $(X\ni x,B)$ has a $\Qq$-factorial weak $\epsilon$-plt blow-up $f': W\rightarrow X$. In particular, $f'$ can be constructed as the composition of $f$ and a small $\mathbb Q$-factorialization $g: W\rightarrow Y$ near the reduced component of $f$.
\end{lem}
\begin{proof}
	Let $E$ be the reduced component of $f$. Since $(Y,B_Y+E)$ is plt near $E$, by \cite[Corollary 1.4.3]{BCHM10}, there exists a small $\Qq$-factorialization $g:W\to Y$ near $E$. Then
	$$-(K_W+B_W+E_W):=-g^{*}(K_Y+B_Y+E)$$
	is big and nef over $X$, where $B_W$ and $E_W$ are the strict transforms of $B_Y$ and $E$ on $W$ respectively. Since $-E_{W}=-g^{*}E$ is nef over $X$ and $\Center_{X}E_{W}=x$, by the negativity lemma, $\Supp E_{W}$ coincides with the fiber over $x$. 
	
	Since $(W,B_W+E_W)$ and $(Y,B_Y+E)$ are plt near $E_W$ and $E$ respectively, $E_W$ and $E$ are normal. The induced birational morphism $g|_{E_W}:E_W\to E$ is a contraction. Thus
	$$(K_W+B_W+E_W)|_{E_W}=g|_{E_W}^{*}((K_Y+B_Y+E)|_{E}),$$
	and $-(K_W+B_W+E_W)|_{E_W}$ is big and nef.
\end{proof}

\begin{defn}[Exceptional singularities]\label{defn: exc sing}
	Let $(X\ni x,B)$ be an lc germ. We say that $(X\ni x,B)$ is an \emph{exceptional singularity}, if for any $\mathbb R$-divisor $G\ge0$ on $X$ such that $(X\ni x,B+G)$ is lc, there exists at most one exceptional prime $\bb$-divisor $E$ over $X$, such that $x\in\Center_{X}E$ and $a(E,X,B+G)=0$. We also say that \emph{$(X\ni x,B)$ is exceptional}.
\end{defn}

\begin{remdef}[Reduced component of exceptional singularities]\label{remdef: reduced component exc sing}
	By Lemma \ref{lem: reduced component exc sing} below, for any klt exceptional singularity $(X\ni x,B)$, there exists a unique plt blow-up $f:Y\rightarrow X$ of $(X\ni x,B)$. In this situation, the reduced component of $f$ will also be called \emph{the reduced component} of $(X\ni x,B)$.
\end{remdef}

The following lemma is a slight generalization of \cite[Corollary 2.6, Corollary 2.7]{PS01} and \cite[Proposition 2.7]{MP99}. For readers' convenience, we give a full proof here.

\begin{lem}\label{lem: reduced component exc sing} 
Let $(X\ni x,B)$ be an exceptional singularity. Then there exists a unique prime $\bb$-divisor $E$ over $X\ni x$ satisfying the following. 

For any $\Rr$-divisor $G\geq 0$ on $X$, if $(X\ni x,B+G)$ is lc but not klt, then $E$ is the only lc place of $(X,B+G)$ whose center on $X$ contains $x$. Moreover, if $\dim X\geq 2$ and $(X\ni x,B)$ is klt, then there exists a unique plt blow-up of $(X\ni x,B)$ with the reduced component $E$. 
\end{lem}

\begin{proof}

Suppose that there exist two $\Rr$-Cartier $\Rr$-divisors $G_1\ge0,G_2\ge0$ on $X$ and a log resolution $g: W\rightarrow X$ of $(X,B+G_1+G_2)$ near a neighborhood of $x$ with prime exceptional divisors $E_1,\dots,E_n$ for some $n\geq 2$, such that
\begin{itemize}
\item $x\in\Center_XE_i$ for every $1\leq i\leq n$,
    \item $(X\ni x,B+G_1)$ and $(X\ni x,B+G_2)$ are lc, and
    \item $a(E_i,X,B+G_i)=0$ for $i\in\{1,2\}$.
\end{itemize}

 For any real number $0\leq s\leq 1$, we define 
	$$t(s):=\lct(X\ni x,B+sG_1;G_2).$$	
	Then $t(s)$ is monotonically decreasing, $t(0)=1$ and $t(1)=0$. For any $s$, there exists $1\leq k\leq n$, such that $E_k$ is the only lc place of $(X,B+sG_1+t(s)G_2)$ with $x\in\Center_XE_k$.
	Let
	$$s_0:=\inf\{s\geq 0\mid E_2\text{ is not an lc place of }(X\ni x,B+sG_1+t(s)G_2)\}.$$
$s_0$ is well-defined because $E_2$ is not an lc place of $(X\ni x,B+G_1)$. By continuity of log discrepancies, there exist two different lc places of the lc germ $(X\ni x,B+s_0G_1+t(s_0)G_2)$, a contradiction. 
	
	Suppose that $\Center_{X}E\neq x$, and $a(E,X,B+G)=0$ for some $\Rr$-divisor $G\geq 0$. Let $H\ni x$ be a general hyperplane section of $X$. Then $c:=\lct(X\ni x,B+G;H)>0$, and $(X,B+G+cH)$ has at least two lc places near $x$, a contradiction. 
	
	If $(X\ni x,B)$ is klt, then we may let $f:Y\to X$ be a plt blow-up of $(X\ni x,B)$, $\tilde{E}$ the reduced component of $f$, and $B_Y$ the strict transform of $B$ on $Y$. First, we show that $\tilde{E}=E_{Y}$, where $E_Y$ is the center of $E$ on $Y$. Since $-(K_Y+B_Y+\tilde{E})$ is ample over $X$, there exists an $\Rr$-divisor $G_Y\geq 0$ on $Y$ such that $K_Y+B_Y+\tilde{E}+G_Y\sim_{\Rr}0/X$ and $(Y,B_Y+\tilde{E}+G_Y)$ is plt near $\tilde{E}$. Let $K_X+B+G:=f_{*}(K_Y+B_Y+\tilde{E}+G_Y)$. Since $(X\ni x,B+G)$ is lc and $a(\tilde{E},X,B+G)=0$, we have $\tilde{E}=E_Y$. Next, we show the uniqueness of $f$. Let $h: V\rightarrow X$ be a log resolution of $(X,B+G)$. Then $V\dashrightarrow Y$ is the ample model of $-E_V$ over $X$, where $E_V$ is the center of $E$ on $V$. The uniqueness of plt blow-up of $(X\ni x,B)$ follows from the uniqueness of the ample model.
\end{proof}

\subsection{Complements}
\begin{defn}[Complements]\label{defn: complement}
	Let $X\to Z$ be a contraction between normal quasi-projective varieties, $B\ge0$ an $\mathbb R$-divisor on $X$, and $z\in Z$ a (not necessarily closed) point. We say that $(X/Z\ni z,B^+)$ is an \emph{$\Rr$-complement} of $(X/Z\ni z,B)$ if $B^{+}\ge B$, $(X,B^{+})$ is lc and $K_X+B^{+}\sim_{\Rr}0$ over some neighborhood of $z$. We say that $(X/Z\ni z,B)$ is $\Rr$-complementary if $(X/Z\ni z,B)$ has an $\Rr$-complement $(X/Z\ni z,B^+)$.
	
	Let $n$ be a positive integer. We say that $(X/Z\ni z,B^+)$ is an \emph{$n$-complement} of $(X/Z\ni z,B)$ if over some neighborhood of $z$,
	\begin{enumerate}
		\item $(X,B^+)$ is lc,
		\item  $n(K_X+B^+)\sim 0$, and
		\item $B^+\geq \lfloor B\rfloor+\frac{1}{n}\lfloor (n+1)\{B\}\rfloor$.
	\end{enumerate}
	We say that $(X/Z\ni z,B^+)$ is a \emph{monotonic $n$-complement} of $(X/Z\ni z,B)$ if we additionally have $B^+\geq B$.
	
	If $\dim Z=0$, for simplicity, we will omit $z$, and say $(X,B^+)$ is an $\Rr$-complement (resp. $n$-complement, monotonic $n$-complement) of $(X,B)$. 
	
	If for any $z\in Z$, $(X/Z\ni z,B^+)$ is an $\Rr$-complement (resp. $n$-complement, monotonic $n$-complement) of $(X/Z\ni z,B)$, then we may also omit $z$ and say $(X/Z,B^+)$ is an $\Rr$-complement (resp. $n$-complement, monotonic $n$-complement) of $(X/Z,B)$. We say that $(X/Z,B)$ is $\Rr$-complementary if $(X/Z,B)$ has an $\Rr$-complement. 
	
	If $X\to Z$ is the identity map, we may omit $Z$ and say $(X\ni z,B^+)$ is an $\Rr$-complement (resp. $n$-complement, monotonic $n$-complement) of $(X\ni z,B)$, and in this case, we also say $(X,B^+)$ is a local $\Rr$-complement (resp. local $n$-complement, monotonic local $n$-complement) of $(X,B)$ near $z$.
	
\end{defn}
\begin{rem}\label{rem: mmp montonic n keep}
	We will use the following fact many times in our paper. See \cite[6.1]{Bir19} for a proof.
	
	Let $(X,B)$ be a pair and $X\rightarrow Z$ a contraction. Assume that $X\dashrightarrow Y$ is a partial MMP over $Z$ on $-(K_X+B)$ and $B_Y$ is the strict transform of $B$ on $Y$. If $(Y/Z,B_Y)$ has a monotonic $n$-complement (resp. $\Rr$-complement) $(Y/Z,B_{Y}^{+})$, then $(X/Z,B)$ has a monotonic $n$-complement (resp. $\Rr$-complement) $(X/Z,B^{+})$.    	
\end{rem} 

We need the following theorem which is conjectured in \cite{PS09} and proved by Birkar \cite{Bir19}. 
\begin{thm}[{\cite[Theorem 1.7, Theorem 1.8]{Bir19}}]\label{thm:Bir19thm1.8}
	Let $d$ be a positive integer and $\Ii_0\subset[0,1]$ a finite set of rational numbers. Then there exists a positive integer $n$ depending only on $d$ and $\Ii_0$ satisfying the following. 
	
	Assume that $(Y,B_Y)$ is a pair and $Y\to X$ is a contraction such that 
	\begin{enumerate}
		\item $(Y,B_Y)$ is lc of dimension $d$,
		\item $B_Y\in\Ii_0$,
		\item $Y$ is of Fano type over $X$, and
		\item $-(K_Y+B_Y)$ is nef over $X$.
	\end{enumerate}
	Then for any point $x\in X$, there exists a monotonic $n$-complement $(Y/X\ni x,B_{Y}^{+})$ of $(Y/X\ni x,B_Y)$.
\end{thm}
The following corollary is a simple consequence of Theorem \ref{thm:Bir19thm1.8}.
\begin{cor}\label{cor: plt complement}
	Let $d$ be a positive integer and $\Ii_0\subset [0,1]$ a finite set of rational numbers. Then there exists a positive integer $n$ depending only on $d$ and $\Ii_0$ satisfying the following. Assume that $(X\ni x,B)$ is an lc germ of dimension $d$, such that
	\begin{enumerate}
		\item $B\in \Ii_0$,
		\item $(X\ni x,\Delta)$ is a klt germ for some boundary $\Delta$, 
		\item there exists either a plt blow-up or a $\Qq$-factorial weak plt blow-up $f: Y\rightarrow X$ of $(X\ni x,\Delta)$ with the reduced component $E$, and
		\item $(Y,B_Y+E)$ is lc, where $B_Y$ is the strict transform of $B$ on $Y$.
	\end{enumerate}
	Then there exists a $\Qq$-Weil divisor $G_Y\geq 0$ on $Y$, such that $nG_{Y}$ is a Weil divisor and $(Y/X\ni x,B_Y+E+G_Y)$ is an $n$-complement of $(Y/X\ni x,B_Y+E)$.  
\end{cor}

\begin{proof}
	Let $\Delta_Y$ be the strict transform of $\Delta$ on $Y$. Since $(X\ni x,\Delta)$ is a klt germ, $(Y,\Delta_Y+E)$ is plt near $E$, and $-E$ is big and nef over $X$. Hence $(Y,\Delta_Y+(1-\epsilon)E)$ is klt near $E$ and $-(K_Y+\Delta_Y+(1-\epsilon)E)$ is big and nef over $X$ for some $0<\epsilon\ll 1$. Thus $Y$ is of Fano type over a neighborhood $x$. 
	
	Since $(X\ni x,B)$ is an lc germ, $$K_Y+B_Y+E=f^{*}(K_X+B)+\alpha E$$
	for some $\alpha\ge0$. Hence $-(K_Y+B_Y+E)$ is nef over $X$, and the corollary follows from Theorem \ref{thm:Bir19thm1.8}. 
\end{proof}

\begin{lem}\label{lem:Fanotypedlt}
Let $X\to Z$ be a contraction. Suppose that $X$ is of Fano type over $Z$, $(X/Z,B)$ is $\Rr$-complementary for some boundary $B$, and $f:Y\to X$ is a birational contraction from a normal quasi-projective variety $Y$, such that $a(E_i,X,B)<1$ for any prime exceptional divisor $E_i$ of $f$. Then $Y$ is of Fano type over $Z$.
\end{lem}
\begin{proof}
Since $X$ is of Fano type over $Z$, there exists a klt pair $(X,C)$, such that $-(K_X+C)$ is big and nef over $Z$. Let $(X/Z,B^{+})$ be an $\Rr$-complement of $(X/Z,B)$, and $D_t=tB^{+}+(1-t)C$. Then $(X,D_t)$ is klt and $-(K_X+D_t)$ is big and nef over $Z$ for any $t\in[0,1)$. We have
$$K_Y+B_{Y}^{+}+\sum (1-a_i^{+})E_i=f^{*}(K_X+B^{+})$$
and
$$K_Y+C_{Y}+\sum (1-a_i)E_i=f^{*}(K_X+C),$$
where $B_{Y}^{+}$ and $C_Y$ are the strict transforms of $B^{+}$ and $C$ on $Y$ respectively, $E_i$ are the exceptional divisors of $f$, $a_i^{+}=a(E_i,X,B^{+})$, and $a_i=a(E_i,X,C)$. Then
\begin{align*}
&K_Y+D_{t,Y}:=f^{*}(K_X+D_t)\\
=&K_Y+tB_{Y}^{+}+(1-t)C_{Y}+\sum (t(1-a_i^{+})+(1-t)(1-a_i))E_i.
\end{align*}
Pick $0<t_0<1$ such that $t_0(1-a_i^{+})+(1-t_0)(1-a_i)\ge0$ for any $i$. Thus  
$-(K_Y+D_{t_0,Y})$ is big and nef over $Z$, $(Y,D_{t_0,Y})$ is klt, and $Y$ is of Fano type over $Z$.
\end{proof}

The following lemma shows that any klt exceptional singularity $(X\ni x,B)$ admits an $\epsilon$-plt blow-up if the coefficients of $B$ belong to a DCC set. The proof is similar to \cite[Proposition 7.2]{Bir19}.

\begin{lem}\label{lem: exp sing epsilon plt}
	Let $d\ge 2$ be a positive integer and $\Ii\subset[0,1]$ a DCC set. Then there exists a positive real number $\epsilon$ depending only on $d$ and $\Ii$ satisfying the following. 
	
	Let $(X\ni x,B)$ be a klt exceptional singularity of dimension $d$ such that $B\in\Ii$, and $f: Y\rightarrow X$ the unique plt blow-up of $(X\ni x,B)$ with the reduced component $E$. Then for any $\Rr$-complement $(Y/X\ni x,B_Y+E+G_Y)$ of $(Y/X\ni x,B_Y+E)$, $(Y,B_Y+E+G_Y)$ is $\epsilon$-plt near $E$, where $B_Y$ is the strict transform of $B$ on $Y$. In particular, $f$ is an $\epsilon$-plt blow-up of $(X\ni x,B)$.
\end{lem}

\begin{proof}
	By Theorem \ref{thm: acc lct}, we may let 
	$$1-\epsilon:=\max\{t\mid t\in \LCT(\Ii\cup\{1\},\Nn,d),0<t<1\}.$$
	We will show that $\epsilon$ satisfies the required properties.
	
	Suppose that there exists a prime $\bb$-divisor $F\neq E$ over $X\ni x$, such that $a:=a(F,Y,B_Y+E+G_Y)<\epsilon$.
	
	We construct a birational morphism $g:W\rightarrow Y$ in the following way. If $F\subset Y$, then we let $g$ be a small $\Qq$-factorialization. If $F\not\subset Y$, then we let $g$ be a birational contraction which extracts exactly $F$, such that $W$ is $\mathbb Q$-factorial. The existence of $g$ follows from \cite[Corollary 1.4.3]{BCHM10}.
	
	By Lemma \ref{lem:Fanotypedlt}, $W$ is of Fano type over neighborhood of $x$. Let $E_W$ be the strict transform of $E$ on $W$, $K_W+B_W+E_W$ the pullback of $K_Y+B_Y+E$, and $G_W$ the pullback of $G$. Let $F_W$ be the center of $F$ on $W$ and $c$ the coefficient of $F_W$ in $G_W$. The coefficient of $F_W$ in $K_W+B_W+E_W+cF_W$ is $1-a>1-\epsilon$, and $(W/X\ni x,B_W+E_W+cF_W)$ is $\Rr$-complementary.
	
	By the construction of $\epsilon$, $(W,B_W+E_W+(a+c)F_W)$ is lc near $E_W$. Running an MMP on $-(K_W+B_W+E_W+(a+c)F_W)$ over a neighborhood of $x$, we get a model $W'$ on which $-(K_{W'}+B_{W'}+E_{W'}+(a+c)F_{W'})$ is nef, where $B_{W'},E_{W'}$ and $F_{W'}$ are the strict transforms of $B_W,E_W$ and $F_W$ on $W'$ respectively.
	
	We claim that $F_W$ is not contracted in the MMP. Otherwise, 
	\begin{align*} 
	&q^{*}(K_{W'}+B_{W'}+E_{W'})\\
	=&q^{*}(K_{W'}+B_{W'}+E_{W'}+(a+c)F_{W'})\\
	\ge &p^{*}(K_W+B_W+E_W+(a+c)F_W),
	\end{align*}
	and 
	$$a(F,K_{W'}+B_{W'}+E_{W'})>a(F,W,B_W+E_W+(a+c)F_W),$$
	where $p:\tilde{W}\to W$, $q:\tilde{W}\to W'$ is a resolution of $W\dashrightarrow W'$. In particular, $K_{W'}+B_{W'}+E_{W'}$ is not lc over a neighborhood of $x$. This is a contradiction as $K_{W'}+B_{W'}+E_{W'}+G_{W'}\sim_{\Rr,X} 0$ and $(W',B_{W'}+E_{W'}+G_{W'})$ is lc over a neighborhood of $x$, where $G_{W'}$ is the strict transform of $G_W$ on $W'$. The claim is proved.
	
	Since $(W/X\ni x,B_W+E_W+cF_W)$ is $\Rr$-complementary, 
	$(W'/X\ni x,B_{W'}+E_{W'}+cF_{W'})$ is also $\Rr$-complementary. By the construction of $\epsilon$, 
	$({W'}/X\ni x,B_{W'}+E_{W'}+(a+c)F_{W'})$ is lc. Hence $({W'}/X\ni x,B_{W'}+E_{W'}+(a+c)F_{W'})$ is $\Rr$-complementary as $-(K_{W'}+B_{W'}+E_{W'}+(a+c)F_{W'})$ is nef over $X$ and $W'$ is of Fano type over a neighborhood of $x$. Let  $(W'/X\ni x,B_{W'}+E_{W'}+(a+c)F_{W'}+D_{W'})$ be an $\Rr$-complement of $({W'}/X\ni x,B_{W'}+E_{W'}+(a+c)F_{W'})$. We define
	$$K_X+\Delta:=f_*g_{*}p_{*}q^{*}(K_{W'}+B_{W'}+E_{W'}+(a+c)F_{W'}+D_{W'}),$$
	then $(X\ni x,\Delta)$ is an $\Rr$-complement of $(X\ni x,B)$. Since $F_W$ and $E$ are lc places of $(X\ni x,\Delta)$ over $X\ni x$, $(X\ni x,B)$ is not exceptional, a contradiction.
\end{proof}

\subsection{Bounded families}
\begin{defn}\label{defn: bdd}
	A \emph{couple} consists of a normal projective variety $X$ and a divisor $D$ on $X$ such that $D$ is reduced. Two couples $(X,D)$ and $(X',D')$ are \emph{isomorphic} if there exists an isomorphism $X\rightarrow X'$ mapping $D$ onto $D'$.
	
	A set $\mathcal{P}$ of couples is \emph{bounded} if there exist finitely many projective morphisms $V^i\rightarrow T^i$ of varieties and reduced divisors $C^i$ on $V^i$ such that for each $(X,D)\in\mathcal{P}$, there exists $i$, a closed point $t\in T^i$, and two couples $(X,D)$ and $(V^i_t,C^i_t)$ are isomorphic, where $V^i_t$ and $C^i_t$ are the fibers over $t$ of the morphisms $V^i\rightarrow T^i$ and $C^i\rightarrow T^i$ respectively. 
	
	A set $\mathcal{C}$ of projective pairs $(X,B)$ is said to be \emph{log bounded} if the set of the corresponding set of couples $\{(X,\Supp B)\}$ is bounded. A set $\mathcal{D}$ of projective varieties $X$ is said to be \emph{bounded} if the corresponding set of couples $\{(X,0)\}$ is bounded. A log bounded (resp. bounded) set is also called a \emph{log bounded family} (resp. \emph{bounded family}).
	
\end{defn}

The following theorem was known as Borisov-Alexeev-Borisov conjecture, and is proved by Birkar.
\begin{thm}[BBAB Theorem, {\cite[Theorem 1.1]{Bir16}}]\label{thm: BAB}
	Let $d$ be a positive integer and $\epsilon$ a positive real number. Then the projective varieties $X$ such that 
	\begin{enumerate}
		\item $(X,B)$ is $\epsilon$-lc of dimension $d$ for some boundary $B$, and
		\item $-(K_X+B)$ is big and nef
	\end{enumerate}
	form a bounded family.
\end{thm}

\begin{lem}\label{lem: bddvolacc}
	Let $d$ be a positive integer and $\epsilon$ a positive real number. Then there exists a positive real number $v$ depending only on $d$ and $\epsilon$ satisfying the following. Suppose that
	\begin{enumerate}
		\item $(X,\Delta)$ is projective $\epsilon$-lc of dimension $d$, and
		\item $-(K_X+\Delta)$ is big and nef,
	\end{enumerate}
	then $\vol(-(K_X+B))\le v$ for any pair $(X,B)$.
\end{lem}

\begin{proof}
Let $f: (X',\Delta')\rightarrow (X,\Delta)$ be a small $\Qq$-factorialization, and $K_{X'}+B':=f^*(K_X+B)$. Possibly replacing $(X,\Delta)$ with $(X',\Delta')$ and $B$ with $B'$, we may assume that $X$ is $\Qq$-factorial.
	
	Since $X$ is of Fano type, we may run a $(-K_X)$-MMP and reach a model $Y$ on which $-K_Y$ is nef. By \cite[Theorem 1.6]{Bir19}, there exists a positive real number $v$ depending only on $d$ and $\epsilon$, such that
	$$\vol(-(K_X+B))\le\vol(-K_X)=\vol(-K_Y)\le v.$$
\end{proof}

\begin{lem}\label{lem: bddvolfin}
	Let $d$ be a positive integer, $\epsilon$ a positive real number and $\Ii_0\subset[0,1]$ a finite set. Then there exists a finite set $\Ii'$ depending only on $d,\epsilon$ and $\Ii$ satisfying with the following. 
	
	If $(X,B)$ is projective $\epsilon$-lc of dimension of $d$, $B\in\Ii$, and $-(K_X+B)$ is big and nef, then $\vol(-(K_X+B))\in \Ii'$.
\end{lem}

\begin{proof}
	Suppose that there exists a sequence of pairs $(X_i,B_i)$, such that $(X_i,B_i)$ is $\epsilon$-lc, $B_i\in\Ii$, $-(K_{X_i}+B_i)$ is big and nef, and $\vol(-(K_{X_i}+B_i))$ is strictly increasing or strictly decreasing.
	
	Let $(X_i',B_i')\rightarrow (X_i,B_i)$ be a  small $\Qq$-factorialization. Possibly replacing $(X_i,B_i)$ with $(X_i',B_i')$, we may assume that $X_i$ is $\Qq$-factorial. By Theorem \ref{thm: BAB}, $(X_i,B_i)$ belongs to a log bounded family.

	Possibly passing to a subsequence of $(X_i,\Supp B_i)$, we may assume that there exist a projective morphism $V\to T$ of varieties, a reduced divisor $C=\sum C_j$ on $V$, and a dense set of closed points $t_i\in T$ such that $X_i$ is the fiber of $V\to T$ over $t_i$, and each component of $\Supp B_i$ is a fiber of $C_j\to T$ over $t_i$ for some $j$. Since $X_i$ is normal, possibly replacing $V$ with its normalization and replacing $C$ with its inverse image with reduced structure, we may assume that $V$ is normal. Possibly shrinking $T$, using Noetherian induction, passing to a subsequence of $(X_i,\Supp B_i)$, and according to \cite[Proposition 2.4]{HX15}, we may assume that $V\to T$ is flat, $C_j\to T$ is flat, $C_j$ is $\Qq$-Cartier and $K_V$ is klt for any $j$ and $t\in T$.
	
	We may assume that $B_i=\sum_{j}b_i^j C_j|_{t_i}$, where $b_i^j\in\Ii$. Possibly passing to a subsequence, we may assume that $b_i^j$ is a constant for each $j$.
	
	For each $j$, let $a_i^j$ be a sequence of increasing rational numbers, such that $a_i^j\le b_1^j$ and $\lim_{i\to+\infty}a_i^j=b_1^j$. Let $B_1^i:=\sum_{j} a_i^j C_j|_{t_1}$ and $B_2^i:=\sum_{j} a_i^j C_j|_{t_2}$. By the asymptotic Riemann-Roch theorem and the invariance of Euler characteristic in a flat family, we have
	\begin{align*}
	&\vol(-(K_{X_1}+B_1))=(-(K_{X_1}+B_1))^d=\lim_{i\to+\infty}(-(K_{X_1}+B_1^i))^d\\ =&\lim_{i\to+\infty}(-(K_{X_2}+B_2^i))^d
	=(-(K_{X_2}+B_2))^d=\vol(-(K_{X_2}+B_2)),               
	\end{align*}
	a contradiction.
\end{proof}
We need the following lemma, which may be well-known to experts. For a proof, see \cite[Lemma 2.5]{Jia18}, \cite[Lemma 4.2]{DS16}, \cite[Lemma 2.1]{CDHJS18}.
\begin{lem}\label{lem: volineq}
	Let $X$ be a normal projective variety, $D$ an $\Rr$-Cartier
	$\Rr$-divisor, and $S$ a base-point free normal Cartier prime divisor. 
	Then for any positive real number $q$,
	\[
	\vol(X,D+qS)\leq \vol(X, D)+q(\dim X) \vol(S, D|_S+qS|_S).
	\]
\end{lem}

The proof of the following proposition was suggested by Christopher Hacon and Chen Jiang.

\begin{prop}\label{prop: bddvolacc}
	Let $d$ be an integer, $\epsilon$ a positive real number, and $\Ii\subset[0,1]$ a DCC set. Then there exists an ACC set $\Ii'$ depending only on $d,\epsilon$ and $\Ii$ satisfying the following. 
	
	If $(X,B)$ is projective $\epsilon$-lc of dimension $d$, $B\in\Ii$, and $-(K_X+B)$ is big and nef, then $\vol(-(K_X+B))\in \Ii'$. 
\end{prop}
\begin{proof}
	Suppose that there exists a sequence of pairs $(X_i,B_i)$, such that $(X_i,B_i)$ is $\epsilon$-lc, $B_i\in\Ii$, $-(K_{X_i}+B_i)$ is big and nef, and $\vol(-(K_{X_i}+B_i))$ is strictly increasing. 
	
	Let $(\tilde X_i,\tilde B_i)\rightarrow (X_i,B_i)$ be a  small $\Qq$-factorialization. Possibly replacing $(X_i,B_i)$ with $(\tilde X_i,\tilde B_i)$, we may assume that $X_i$ is $\Qq$-factorial. By Theorem \ref{thm: BAB}, $(X_i,B_i)$ belongs to a log bounded family. In particular, the number of irreducible components of $B_i$ is bounded from above. Possibly passing to a subsequence, we may assume that
	$$B_i=\sum_{j=1}^m b_{i,j}B_{i,j},$$
    where $m$ is a non-negative integer, $B_{i,j}$ are the irreducible components of $B_i$, $b_{i,j}\in\Ii$, and $\{b_{i,j}\}_{i=1}^{\infty}$ is an increasing sequence for every $1\leq j\leq m$. Let
	$$\overline{b_j}:=\lim_{i\to +\infty}b_{i,j},\text{ and }\overline{B_i}:=\sum_{j=1}^m \overline{b_j}B_{i,j}.$$

	\begin{claim}\label{claim: effective acc volume} There exists a sequence of positive real numbers $\epsilon_i$, such that $\lim_{i\to+\infty}\epsilon_i=0$, and
		$$\epsilon_i(-(K_{X_i}+B_i))-(\overline{B_i}-B_i)$$
		is effective.
	\end{claim}
	
	Suppose that the claim is true. Possibly passing to a subsequence, we may assume that $\epsilon_i<1$ for every $i$. Then 
	$$-(K_{X_i}+\overline{B_i})-(1-\epsilon_i)(-(K_{X_i}+B_i))=\epsilon_i(-(K_{X_i}+B_i))-(\overline{B_i}-B_i)$$
	is effective
	and $-(K_{X_i}+\overline{B_i})$ is big. Thus
	\begin{align*}
	\vol(-(K_{X_i}+\overline{B_i}))
	\ge&(1-\epsilon_i)^d\vol(-(K_{X_i}+B_i))\\
	\ge&(1-\epsilon_i)^d\vol(-(K_{X_i}+\overline{B_i})).
	\end{align*}	 
	We may run a $-(K_{X_i}+\overline{B_i})$-MMP, and reach a minimal model $X_i\dashrightarrow X_i'$. Since $-(K_{X_i}+B_i)$ is big and nef, there exists $D_i\ge0$, such that $K_{X_i}+B_i+D_i\sim_{\Rr}0$ and $(X_i,B_i+D_i)$ is $\epsilon$-lc. Thus $K_{X_i'}+{B_i}'+D_i'\sim_{\Rr}0$, and
	$(X_i',B_i'+D_i')$ is $\epsilon$-lc, where $B_i'$ and $D_i'$ are the strict transforms of $B_i$ and $D_i$ on $X_i'$. By \cite[Corollary 1.2]{Bir16}, $(X_i',B_i')$ belongs to a log bounded family. By taking a log resolution on the log bounded family, and possibly passing to a subsequence, we may assume that $(X_i',\overline{B_i}')$ is $\epsilon$-lc, where $\overline{B_i}'$ is the strict transform of $\overline{B_i}$ on $X_i'$. By Lemma \ref{lem: bddvolfin}, 	 
	$$\vol(-(K_{X_i}+\overline{B_i}))=\vol(-(K_{X_i}'+\overline{B_i}'))$$ belongs to a finite set. Possibly passing to a subsequence, we may assume that $\vol(-(K_{X_i}+\overline{B_i}))=C$ is a constant. Hence 
	$$
	\frac{C}{(1-\epsilon_i)^d}
	\ge\vol(-(K_{X_i}+B_i))\\
	\ge C.
	$$	
	Since $\epsilon_i\to 0$ and $\vol(-(K_{X_i}+B_i))$ is increasing, we deduce that $\vol(-(K_{X_i}+B_i))$ belongs to a finite set, a contradiction.
\end{proof}

\begin{proof}[Proof of Claim \ref{claim: effective acc volume}] 
	Let $A_i$ be a very ample divisor on $X_i$, such that 
	$$\delta_i A_i-(\overline{B_i}-B_i)$$
	is effective for some positive real number $\delta_i$, $\lim_{i\to +\infty}\delta_i=0$, and $(-(K_{X_i}+B_i))^{d-1}\cdot A_i<a$, where $a$ is a positive real number depending only on $d,\epsilon$ and $\Ii$. There exists a positive real number $b$, such that $$b<\frac{(-(K_{X_i}+B_i))^{d}}{d(-(K_{X_i}+B_i))^{d-1}\cdot A_i}=\frac{\vol(-(K_{X_i}+B_i))}{d(-(K_{X_i}+B_i))^{d-1}\cdot A_i}.$$
	
	By Lemma \ref{lem: volineq}, we have
	\begin{align*}
	&\vol(-(K_{X_i}+B_i)-bA_i)\\
	\ge&     \vol(-(K_{X_i}+B_i))-bd\vol(-(K_{X_i}+B_i)|_{A_i})\\
	>&\vol(-(K_{X_i}+B_i))-\frac{\vol(-(K_{X_i}+B_i))}{(-(K_{X_i}+B_i))^{d-1}\cdot A_i}\cdot ((-(K_{X_i}+B_i))^{d-1}\cdot A_i)\\
	=&0,
	\end{align*}
	which implies that 
	$$-(K_{X_i}+B_i)-bA_i$$
	is effective.
	
	Let $\epsilon_i:=\frac{\delta_i}{b}$, then
	$$\epsilon_i(-(K_{X_i}+B_i))-(\overline{B_i}-B_i)=(\epsilon_i(-(K_{X_i}+B_i))-b\epsilon_iA_i)+(\delta_iA_i-(\overline{B_i}-B_i))$$
	is effective, and the claim is proved.
\end{proof}

\begin{lem}\label{lem: indexofL}
	Let $d$ be a positive integer and $\epsilon,\epsilon_0$ two positive real numbers. Then there exists a positive integer $m$ depending only on $d,\epsilon$ and $\epsilon_0$ satisfying the following. Suppose that
	\begin{enumerate}
		\item $(X,\Delta)$ is projective $\epsilon$-lc of dimension $d$,
		\item $-(K_X+\Delta)$ is big and nef,
		\item $L$ is a nef Weil divisor on $X$, and
		\item $-(K_X+B)-\epsilon_0 L$ is pseudo-effective for some $\Rr$-divisor $B\geq 0$,
	\end{enumerate}
	then $mL$ is Cartier.
\end{lem}
\begin{proof}
	Let $(X',\Delta')\to (X,\Delta)$ be a small $\Qq$-factorialization, $K_{X'}+B'$ the pullback of $K_X+B$, and $L'$ the pullback of $L$. By Theorem \ref{thm: BAB},
	$X'$ is bounded. In particular, there exists a very ample divisor $A'$ on $X'$ and a positive real number $r$ depending only on $d$ and $\epsilon$, such that ${A'}^d\le r$ and $-K_{X'}\cdot {A'}^{d-1}\le r$. We have 
	$$\epsilon_0 L'\cdot {A'}^{d-1}\le -(K_{X'}+B')\cdot {A'}^{d-1}\le -K_{X'}\cdot {A'}^{d-1}\le r.$$
	By \cite[Lemma 2.25]{Bir19}, there exists a positive integer $m$ depending only on $d,r$ and $\epsilon_0$, such that $mL'$ is Cartier. 
	
	Since $X'\to X$ is a blow up, there exists a $\Qq$-divisor $H'\ge0$ on $X'$, such that $-H'$ is ample over $X$. In particular, $K_{X'}+\Delta'+H'$ is antiample over $X$. Rescaling it we can assume that $(X',\Delta'+H')$ is klt. Then $X'\to X$ is a $(K_{X'}+\Delta'+H')$-negative contraction of an extremal face of the Mori-Kleiman cone of $X'$. According to the cone theorem, $mL$ is Cartier. 
\end{proof}
It is expected that the following more general statement holds.
\begin{ques}
	Suppose that projective klt pairs $(X,B)$ belong to a log bounded family. Is there a positive integer $I$ such that the Cartier index of any $\Qq$-Cartier Weil divisor on $X$ is bounded from above by $I$?
\end{ques}
\begin{rem}
	If $X$ is a fixed projective variety, then the above question has a positive answer (cf. \cite[Lemma 7.14]{CH20}). By applying the methods and results in \cite{KM92,dFH11}, we may show the above question has a positive answer for weak log Fano pairs.
\end{rem}

\subsection{Normalized volumes}
\begin{defn}[Normalized volumes]\label{defn: nv}
	Let $(X\ni x,B)$ be a klt germ. Let $E$ be the reduced component of a plt blow-up $f: Y\rightarrow X$ of $(X\ni x,B)$. We define the \emph{normalized volume of $E$} with respect to $(X\ni x,B)$ as
	$$ \widehat{\vol}(E,X,B):=(-(K_Y+E+f^{-1}_*B)|_{E})^{d-1}\cdot a(E,X,B),$$
	where $d:=\dim X$. The \emph{normalized volume of $(X\ni x,B)$} is defined as
	$$\widehat{\vol}(X\ni x,B):=\inf_{E}\widehat{\vol}(E,X,B),$$
	where the above infimum takes over all the possible reduced components of $(X\ni x,B)$. 
\end{defn}

\begin{lem}[{\cite[Proposition 5.20]{KM98}}]\label{lem: ramification cover}
	Let $(X\ni x,B)$ be an lc germ and $h: X'\ni x'\rightarrow X\ni x$ a finite morphism, where $x'\in X'$ is a closed point and $h(x')=x$. Suppose that there exists an $\mathbb R$-divisor $B'\ge0$ such that $K_{X'}+B'=h^*(K_X+B)$. Then
	\begin{enumerate}
		\item for any prime divisor $E$ over $X$, $a(E',X',B')=ra(E,X,B)$, where $E'$ is the corresponding prime divisor of $E$ over $X'$, $0<r\leq\deg h$ is an integer, and
		\item if $(X\ni x,B)$ is a klt (resp, $\epsilon$-lc) germ, then $(X'\ni x',B')$ is a klt (resp. $\epsilon$-lc) germ.
	\end{enumerate}
\end{lem}
\begin{lem}\label{lem: plt blow-up rami}
	Let $(X\ni x,B)$ and $(X'\ni ,B')$ be klt germs. Suppose that there exists a finite morphism $h: (X'\ni x',B')\rightarrow (X\ni x,B)$, such that $h^*(K_X+B)=K_{X'}+B'$. Let $E$ be the reduced component of a plt blow-up $f: Y\rightarrow X$ of $(X\ni x,B)$, $Y'$ the main component of $Y\times_{X}X'$, $f': Y'\rightarrow X'$ the induced morphism, and $E'$ the exceptional divisor of $f'$. Then $E'$ is the reduced component of $(X'\ni x',B')$. In particular, if $f$ is an $\epsilon$-plt blow-up of $(X\ni x,B)$, then $f'$ is an $\epsilon$-plt blow-up of $(X'\ni x',B')$.
	
	Moreover, 	
	$$\deg f\cdot\widehat{\vol}(E,X,B)=\widehat{\vol}(E',X',B').$$
\end{lem}

\begin{proof}
	The lemma follows from \cite[Lemma 2.9]{LX16}, Lemma \ref{lem: ramification cover} and \cite[Lemma 2.10]{LX16}.
\end{proof}

\begin{defn}\label{name rami reduced component} Assumptions and notation as in Lemma \ref{lem: plt blow-up rami}. We say that $E'$ is the \textit{component corresponding} to $E$ with respect to $h$.
\end{defn}

\section{Boundedness of local algebraic fundamental groups}

The goal of this section is to prove Theorem \ref{thm: bdd local index exceptional}.

\subsection{Boundedness of log discrepancies and normalized volumes} 

\begin{lem}\label{lem: bdd lemma}
	Let $d$ be a positive integer and $\epsilon,\delta$ two positive real numbers. Suppose that $(X\ni x,B)$ is a germ of dimension $d$, such that
	\begin{enumerate}
		\item all the (nonzero) coefficients of $B$ are $\geq\delta$,
		\item $(X\ni x,\Delta)$ is a klt germ for some boundary $\Delta$,
		\item there exists either an $\epsilon$-plt blow-up or a $\Qq$-factorial weak $\epsilon$-plt blow-up $f: Y\rightarrow X$ of $(X\ni x,\Delta)$ with the reduced component $E$, and
		\item $-(K_Y+B_Y+E)|_{E}$ is pseudo-effective, where $B_Y$ is the strict transform of $B$ on $Y$.
	\end{enumerate}
	Then $(E, {\Diff}_{E}(B_Y))$ belongs to a log bounded family. Moreover, if all the coefficients of $B$ belong to a DCC (resp. finite) set $\Ii\subset[0,1]$, then all the coefficients of $ {\Diff}_{E}(B_Y)$ belong to a DCC (resp. finite) set.	
\end{lem}
\begin{rem} We remark that $(E, {\Diff}_{E}(B_Y))$ may not be lc.
\end{rem}

\begin{proof} Let $\Delta_Y$ be the strict transform of $\Delta$ on $Y$, and 
	$$K_E+{\rm Diff}_{E}(\Delta_Y):=(K_Y+\Delta_Y+E)|_{E}.$$
	By Theorem \ref{thm: adjunction}, $(E,{\rm Diff}_{E}(\Delta_Y))$ is $\epsilon$-klt. By Theorem \ref{thm: BAB}, $E$ belongs to a bounded family. Hence there exists a very ample divisor $H$ on $E$, such that both $H^{d-1}$ and $(-K_E)\cdot H^{d-2}$ are bounded from above.
	
	Since $-(K_Y+B_Y+E)|_{E}$ is pseudo-effective, 
	$$-(K_{E}+{\Diff}_{E}(B_Y))\cdot H^{d-2}\ge 0.$$
	We may assume that $\delta\le \frac{1}{2}$. Since all the nonzero coefficients of $B$ are $\geq \delta$, by Theorem \ref{thm: adjunction}, we have
	$${\Diff}_{E}(B_Y)\geq\delta\Supp({\Diff}_{E}(B_Y)).$$	
	Thus
	$$\Supp{\Diff}_{E}(B_Y)\cdot H^{d-2}\le {\Diff}_{E}(B_Y)\cdot\frac{1}{\delta}H^{d-2} \le \frac{1}{\delta}(-K_E)\cdot H^{d-2}$$
	is bounded from above, and $(E, {\Diff}_{E}(B_Y))$ belongs to a log bounded family.
	
	Let $V$ be any irreducible component of ${\Diff}_{E}(B_Y)$. By Theorem \ref{thm: adjunction}, 
	$$\mult_V{\Diff}_{E}(B_Y)=\frac{m-1+\sum_i n_ib_i}{m}$$
	and
	$$\mult_V{\Diff}_{E}(\Delta_Y)=\frac{m-1+\sum_i n_i'\delta_i}{m}$$
	for some positive integer $m$, non-negative integers $n_i,n_i'$, and $b_i,\delta_i\in\Ii$, where $b_i$ and $\delta_i$ are coefficients of $B_Y$ and $\Delta_Y$ respectively. Since $(E,{\rm Diff}_{E}(\Delta_{Y}))$ is $\epsilon$-klt, $m\le\frac{1}{\epsilon}$. Hence $m$ belongs to a finite set. In particular, if $\Ii$ is a finite set, then all the coefficients of $ {\Diff}_{E}(B_Y)$ belong to a finite set.	
\end{proof}

\begin{prop}\label{prop: acc ld kc}
	Let $d$ be a positive integer, $\epsilon$ a positive real number, and $\Ii\subset [0,1]$ a DCC set. Assume that
	\begin{itemize}
		\item $(X\ni x,B), (X\ni x,\Delta)$ are two klt germs of dimension $d$, and
		\item there exists either an $\epsilon$-plt blow-up or a $\Qq$-factorial weak $\epsilon$-plt blow-up $f: Y\rightarrow X$ of $(X\ni x,\Delta)$ with the reduced component $E$.
	\end{itemize}      
	Then
	\begin{enumerate}
		\item $a(E,X,B)$ is bounded from above,
		\item if $B\in\Ii$, then $\{a(E,X,B)\}$ is an ACC set depending only on $d$, $\epsilon$ and $\Ii$, and
		\item if $B\in\Ii$ and we additionally assume that $\Ii$ is finite, then the only possible accumulation point of $\{a(E,X,B)\}$ is $0$.
	\end{enumerate}

\end{prop}
\begin{proof} 
	Let $a:=a(E,X,B)$, $B_Y$ the strict transform of $B$ on $Y$, and
	$$K_E+B_E:=(K_Y+B_Y+E)|_{E}.$$	
	Since $f$ is either an $\epsilon$-plt blow up or a $\Qq$-factorial weak $\epsilon$-plt blow-up of $(X\ni x,\Delta)$, by Theorem \ref{thm: BAB}, $E$ belongs to a bounded family, and if $B\in\Ii$, then by Lemma \ref{lem: bdd lemma}, $(E,B_E)$ is log bounded. Thus, we may pick a very ample Cartier divisor $H$ on $E$, such that $-K_E\cdot H^{d-2}$ is bounded from above. In particular, there are only finitely many possibilities of $K_E\cdot H^{d-2}$. Moreover, $B_E\cdot H^{d-2}\le-K_E\cdot H^{d-2}$ is bounded from above. If $B\in\Ii$, then by Lemma \ref{lem: bdd lemma}, all the coefficients of $B_E$ belong to a DCC set (finite set if $\Ii$ is finite), hence $B_E\cdot H^{d-2}$ belongs to a DCC set (finite set if $\Ii$ is finite), and $-(K_E+B_E)\cdot H^{d-2}$ belongs to an ACC set (finite set if $\Ii$ is finite).

	Let $H_1,\ldots,H_{d-2}$ be general elements in $|H|$, $C:=E\cap H_1\cap H_2\cdots \cap H_{d-2}$, and $r:=\lfloor\frac{1}{\epsilon}\rfloor!$. Since $(Y,\Delta_Y+E)$ is $\epsilon$-plt near $E$, by Theorem \ref{thm: adjunction}, $rE|_{E}$ is a $\Qq$-Cartier Weil divisor. In particular, $-E\cdot C\in\frac{1}{r}\mathbb N^+$ belongs to a discrete set. We have
	\begin{align*}
	\frac{a}{r}\le &-aE\cdot C=-(K_Y+B_Y+E)\cdot C\\
	=&-(K_E+B_E)\cdot H^{d-2}\le -K_E\cdot H^{d-2}.
	\end{align*}
	In particular, $a$ is bounded from above. 
	
	Since $(K_Y+B_Y+(1-a)E)\cdot C=0$ and $-E\cdot C>0$, we have 
	$$a=\frac{-(K_Y+B_Y+E)\cdot C}{-E\cdot C}.$$
	Suppose that $B\in\Ii$. Then $-(K_Y+B_Y+E)\cdot C=-(K_E+B_E)\cdot H^{d-2}$ belongs to an ACC set (finite set if $\Ii$ is finite), and $a$ belongs to an ACC set (an ACC set whose only possible accumulation point is $0$ if $\Ii$ is finite). 
\end{proof}

\begin{prop}\label{prop:CartierindexofEE}
	Let $d$ be a positive integer and $\epsilon_0,\epsilon$ two positive real numbers. Then there exists a positive integer $I$ depending only on $d,\epsilon_0$ and $\epsilon$ satisfying the following. For any germ $(X\ni x,B)$ of dimension $d$, if 
	\begin{enumerate}
		\item $a(E,X,B)\ge\epsilon_0$, 
		\item $(X\ni x,\Delta)$ is a klt germ for some boundary $\Delta$, and
		\item there exists either an $\epsilon$-plt blow-up or a $\Qq$-factorial weak $\epsilon$-plt blow-up $f: Y\rightarrow X$ of $(X\ni x,\Delta)$ with the reduced component $E$,
	\end{enumerate}
	then $IE|_{E}$ is Cartier.
\end{prop}
\begin{proof}
	Let $a:=a(E,X,B)$ and $B_Y$ the strict transform of $B$ on $Y$. 
	We have
	$$K_Y+B_Y+(1-a)E=f^*(K_X+B).$$
	Let
	$$K_E+B_E:=(K_Y+B_Y+E)|_{E}\sim_{\Rr} aE|_{E},$$
	and $r:=\lfloor\frac{1}{\epsilon}\rfloor!$. Since $f$ is an $\epsilon$-plt blow up of $(X\ni x,\Delta)$, by Theorem \ref{thm: adjunction}, $L:=-rE|_{E}$ is a Weil divisor on $E$. Since 
	$$-(K_E+B_E)-\frac{\epsilon_0}{r}L\sim_{\mathbb R}-(a-\epsilon_0)E|_{E}$$
	is nef, by Lemma \ref{lem: indexofL}, there exists a positive integer $m$ depending only on $d,\epsilon_0$ and $\epsilon$, such that $mL$ is Cartier. We get the desired $I$ by letting $I:=mr$.
\end{proof}

\begin{thm}\label{thm: nv acc 2} Let $d$ be a positive integer, $\epsilon$ a positive real number, and $\Ii\subset [0,1]$ a DCC (resp. finite) set. Then for any klt germ $(X\ni x,B)$ of dimension $d$, if
	\begin{enumerate}
	\item there exists a plt blow-up $f:Y\rightarrow X$ of $(X\ni x,B)$ with the reduced component $E$, and
	\item $f$ is an $\epsilon$-plt blow-up for some klt germ $(X\ni x,\Delta)$,

	\end{enumerate}
	then $\widehat{\vol}(E,X,B)$ is bounded from above, and $\widehat{\vol}(E,X,B)$ is bounded from below by a positive real number if and only if $a(E,X,B)$ is bounded from below by a positive real number. 
	
	Moreover, if $B\in \Ii$,
	then $\widehat{\vol}(E,X,B)$ belongs to an ACC set (resp. ACC set whose only possible accumulation point is $0$) depending only on $d,\epsilon$ and $\Ii$.
	\end{thm} 

\begin{proof}
	Let $a:=a(E,X,B)$, and $B_Y$, $\Delta_Y$ the strict transforms of $B$, $\Delta$ on $Y$ respectively. We have
	$$K_Y+B_Y+(1-a)E=f^*(K_X+B).$$
	Let
	$$K_E+B_E:=(K_Y+B_Y+E)|_{E}\sim_{\Rr} aE|_{E},$$
	and
	$$K_E+\Delta_E:=(K_Y+\Delta_Y+E)|_{E},$$
	then by Theorem \ref{thm: adjunction}, $(E,\Delta_E)$ is $\epsilon$-klt. By Lemma \ref{lem: bddvolacc}, $\vol(-(K_{E}+B_E))$ is bounded from above. By Proposition \ref{prop: acc ld kc}, $a$ is bounded from above. Thus
	$$\widehat{\vol}(E,X,B)=a\vol(-(K_{E}+B_E))$$
	is bounded from above.
	
	Suppose that $a$ is bounded from below by a positive real number $\epsilon_0$. By Proposition \ref{prop:CartierindexofEE}, there exists a positive integer $I$ depending only on $d,\epsilon_0$ and $\epsilon$, such that $IE|_{E}$ is Cartier. Thus
	\begin{align*}
	\widehat{\vol}(E,X,B)=&a\vol(-(K_{E}+B_E))=a\vol(-aE|_{E})\\
	=&\frac{a^d}{I^{d-1}}\vol(-IE|_{E})\ge\frac{a^d}{I^{d-1}}\geq \frac{\epsilon_0^d}{I^{d-1}}>0
	\end{align*}
	is bounded from below by a positive real number.
	
	Now suppose that $\widehat{\vol}(E,X,B)$ is bounded from below by a positive real number. Since $\vol(-(K_{E}+B_E))$ has an upper bound, 
	$$a=\frac{\widehat{\vol}(E,X,B)}{\vol(-(K_{E}+B_E))}$$
	is bounded from below by a positive real number.
	
	If $B\in \Ii$, then by Proposition \ref{prop: acc ld kc}, $a$ belongs to an ACC set (resp. an ACC set whose only possible accumulation point is $0$). By Lemma \ref{lem: bddvolfin} and Proposition \ref{prop: bddvolacc}, $\vol(-(K_{E}+B_E))$ belongs to an ACC set (resp. finite set). Thus
	$$ \widehat{\vol}(E,X,B)=a\vol(-(K_{E}+B_E))$$
	belongs to an ACC set (resp. ACC set whose only possible accumulation point is $0$).
\end{proof}

\subsection{Proof of Theorem \ref{thm: bdd local index exceptional}}

In this subsection, we prove Theorem \ref{thm: bdd local index exceptional}. First we show the following proposition.

\begin{prop}\label{prop: finite degree rami}
	Let $d$ be a positive integer and $\epsilon_0,\epsilon$ two positive real numbers. Assume that
	\begin{enumerate}
	    \item $(X\ni x,B)$ is an $\epsilon_0$-lc germ of dimension $d$, and
	    \item $(X\ni x,B)$ admits an $\epsilon$-plt blow-up.
	\end{enumerate} 
	Then for any finite morphism $g: X'\ni x'\rightarrow X\ni x$ such that $K_{X'}+B'=g^*(K_X+B)$ for some $\Rr$-divisor $B'\geq 0$, $\deg g$ belongs to a finite set depending only on $d,\epsilon_0$ and $\epsilon$.
\end{prop}
\begin{proof} 
	By Lemma \ref{lem: ramification cover}, $(X'\ni x',B')$ is an $\epsilon_0$-lc germ. Let $f: Y\rightarrow X$ be an $\epsilon$-plt blow-up of $(X\ni x,B)$ with the reduced component $E$. By Lemma \ref{lem: plt blow-up rami},
	$$\deg g=\frac{\widehat{\vol}(E',X',B')}{\widehat{\vol}(E,X,B)},$$
	where $E'$ is the component corresponding to $E$ with respect to $g$, and the induced blow-up $f': Y'\rightarrow X'$ is an $\epsilon$-plt blow-up.
	
	By Theorem \ref{thm: nv acc 2}, $\widehat{\vol}(E,X,B)$ is bounded from below by a positive real number and $\widehat{\vol}(E',X',B')$ is bounded from above. Thus $\deg g$ belongs to a finite set. 
\end{proof}
\begin{proof}[Proof of Theorem \ref{thm: bdd local index exceptional}] This follows from Proposition \ref{prop: finite degree rami}.  
\end{proof}

\section{$(n,\Ii_0)$-decomposable $\Rr$-complements}

The goal of this section is to show Theorem \ref{thm: existence ni1i2 complement}, i.e., the existence of $(n,\Ii_0)$-decomposable $\Rr$-complements (Definition \ref{defn: nI1I2complmenet}).

We use an example to illustrate our motivation.
\begin{ex}\label{ex: irr complement}
	The existence of monotonic $n$-complements is useful in many applications. However, monotonic $n$-complements may not exist even for $\mathbb P^1$.
	
	Consider the pair
	$$\left(X:=\mathbb P^1,B:=\frac{2-\sqrt{2}}{2}(p_1+p_2)+\frac{\sqrt{2}}{2}(p_3+p_4)\right),$$
	where $p_1,p_2,p_3,p_4$ are four different points on $\mathbb P^1$. Then $K_X+B\sim_{\mathbb R}0$. However, since the coefficients of $B$ are not rational, $n(K_X+B)\not\sim0$ for any positive integer $n$. In particular, for any $B^+\geq B$, $n(K_X+B^+)\not\sim0$. Hence there does not exist any monotonic $n$-complement of $(X,B)$. 

	On the other hand, we may write $$K_X+B=a_1(K_X+B_1)+a_2(K_X+B_2),$$
	as an $\Rr$-linear sum of two lc pairs with rational coefficients, where $(a_1,a_2):=(\frac{2-\sqrt{2}}{2},\frac{\sqrt{2}}{2})$, and $(B_1,B_2):=(p_1+p_2,p_3+p_4)$.
	Moreover, both $(X,B_1)$ and $(X,B_2)$ are monotonic $1$-complements of themselves respectively. In this sense, $(X,B=a_1B_1+a_2B_2)$ could be regarded as an ``irrational monotonic $1$-complement'' of itself, and we say $(X,B)$ is an \emph{$(1,\{a_1,a_2\})$-decomposable $\Rr$-complement} of itself.
\end{ex}

As we have seen in Example \ref{ex: irr complement}, the non-existence of monotonic $n$-complement is due to the irrational coefficients of the pair. We will show the existence of uniform lc rational polytopes and uniform $\Rr$-complementary rational polytopes in this section. 


\subsection{Uniform lc rational polytopes}

We need the following useful result of Nakamura which is related to accumulation points of log canonical thresholds, and the proof is based on some ideas in \cite{HMX14}.
\begin{thm}[{\cite[Theorem 1.6]{Nak16}}]\label{thm: Nak16}
	Fix three positive integers $d,c,m$. Let $r_0:=1$, $r_1,\ldots,r_{c}$ be positive real numbers which are linearly independent over $\Qq$. Let $s_1,\ldots, s_m: \Rr^{c+1}\rightarrow\Rr$ be $\Qq$-affine functions. Then there exists a positive real number $\epsilon$ depending only on $d,c,m$, $r_1,\ldots, r_{c}$ and $s_1,\ldots, s_m$ satisfying the following. For any $\Qq$-Gorenstein normal variety $X$ of dimension $d$ and $\Qq$-Cartier Weil divisors $D_1,\ldots,D_m\ge0$ on $X$, if 
	$$(X,\sum_{i=1}^ms_i(r_0,\ldots,r_{c})D_i)$$ is lc, then
	$$(X,\sum_{i=1}^ms_i(r_0,\ldots,r_{c-1},t)D_i)$$
	is lc for any $t$ such that $|t-r_{c}|\leq\epsilon$.
\end{thm}

In Theorem \ref{thm: Nak16}, we will show that $r_1,\ldots,r_c$ are not necessarily positive, $X$ is not necessarily $\Qq$-Gorenstein, and $D_1,\ldots,D_c$ are not necessarily $\Qq$-Cartier. See Corollary \ref{cor: strong Nakamura approximation} for more details.

\medskip

We first show the following lemma which should be known to experts.

\begin{lem}\label{lem: R boundary Q Cartier}
	Let $X$ be a normal variety, $f:X\rightarrow Z$ a contraction, and $D$ an $\mathbb R$-Cartier $\mathbb R$-divisor on $X$. Suppose that there exist a positive integer $n$, real numbers $r_1,\ldots,r_n$, and $\mathbb Q$-divisors $D_1,\ldots,D_n$ on $X$, such that
	\begin{enumerate}
		\item $D=\sum_{i=1}^nr_iD_i$, 
		\item $r_1,\ldots,r_n$ are linearly independent over $\Qq$, and
		\item $D\sim_{\mathbb R,Z}0$,
	\end{enumerate}
	then for any $1\leq i\leq n$, $D_i\sim_{\mathbb Q,Z}0$. In particular, $D_i$ is $\mathbb Q$-Cartier. 
\end{lem}

\begin{proof}
	Since $D\sim_{\mathbb R,Z}0$, we may write
	$$D=-\sum_{i=n+1}^{m'}r_i'D_i',$$
	where for every $n+1\leq i\leq m'$, either $D_i'=f^*B_i'$ for some $\mathbb Q$-Cartier Weil divisor $B_i$ on $Z$, or $D_i'=\text{Div}(f_i)$ for some rational function $f_i$ on $X$. Therefore, there exist an integer $n\leq m\leq m'$ and real numbers $r_{n+1},\dots,r_m$ which are linearly independent over $\Qq$, such that
	$$D=-\sum_{j=n+1}^{m} r_jD_j,$$
	where $D_j$ is a $\Qq$-linear combination of $D_{n+1}',\dots,D'_{m'}$ for every $n+1\leq j\leq m$. Possibly reordering $j$ for $n+1\le j\le m$, we may assume that $\{r_1,\ldots,r_k\}$ is a basis of $\Span_{\Qq}(\{r_1,\ldots,r_m\})$ over $\Qq$ for some $k\ge n$.
	
	There exist $a_{j,i}\in\Qq$ for $k+1\le j\le m,1\le i\le k$, such that 
	$$r_j=\sum_{i=1}^k a_{j,i}r_i.$$
	We have 
	$$\sum_{i=1}^k r_i(D_i+\sum_{j=k+1}^m a_{j,i}D_j)=0.$$
	Thus $D_i+\sum_{j=k+1}^m a_{j,i}D_j=0$ for every $i$, which implies that $D_i\sim_{\mathbb Q,Z}0$ and $D_i$ is $\mathbb Q$-Cartier for every $1\le i\le n$.
\end{proof}

\begin{lem}\label{lem: dltmodelQfactorialcoeff1}
	Let $c,m$ be two positive integers, $r_1,\dots,r_c$ real numbers, $s_1,\dots,s_m$ $\Qq$-linear functions: $\Rr^{c+1}\rightarrow\Rr$, $X$ a normal variety and $D_1,\dots,D_m$ Weil divisors on $X$, such that $r_0:=1,r_1,\dots,r_c$ are linearly independent over $\Qq$ and $(X,\sum_{j=1}^ms_j(r_0,\dots,r_c)D_j)$ is a pair. Then
	\begin{enumerate}
	    \item $K_X+\sum_{j=1}^ms_j(1,x_1,\dots,x_c)D_j$ is $\Rr$-Cartier for any real numbers $x_1,\dots,x_c$, 
	    \item for every prime divisor $E$ over $X$, if $a(E,X,\sum_{j=1}^ms_j(r_0,\dots,r_c)D_j)=0$, then $a(E,X,\sum_{j=1}^ms_j(1,x_1,\dots,x_c)D_j)=0$ for any real numbers $x_1,\dots,x_c$, and
	    \item if $X\to Z$ is a contraction, and $K_X+\sum_{j=1}^ms_j(r_0,\dots,r_c)D_j\sim_{\Rr,Z}0$, then $K_X+\sum_{j=1}^ms_j(1,x_1,\dots,x_c)D_j\sim_{\Rr,Z}0$ for any real numbers $x_1,\dots,x_c$.
	\end{enumerate}
\end{lem}
\begin{proof}
For any $1\leq j\leq m$, suppose that $s_j(x_0,\dots,x_c)=\sum_{i=0}^cm_{i,j}x_i$ for some rational numbers $m_{0,j},\dots,m_{i,j}$. Then 
$$K_X+\sum_{j=1}^ms_j(1,x_1,\dots,x_c)D_j=K_X+\sum_{j=1}^mm_{0,j}D_j+\sum_{i=1}^cx_i(\sum_{j=1}^mm_{i,j}D_j)$$
for any real numbers $x_1,\dots,x_c$. (1) and (3) follow from Lemma \ref{lem: R boundary Q Cartier} immediately.

By Lemma \ref{lem: R boundary Q Cartier}, $\sum_{j=1}^mm_{i,j}D_j$ is $\Qq$-Cartier for any $1\leq i\leq c$. Therefore, for any prime divisor $E$ over $X$ such that $a(E,X,\sum_{j=1}^ms_j(r_0,\dots,r_c)D_j)=0$, we have
\begin{align*}
&a(E,X,\sum_{j=1}^ms_j(1,x_1,\dots,x_c)D_j)\\
=&a(E,X,\sum_{j=1}^mm_{0,j}D_j)-\sum_{i=1}^{c}x_{i}\mult_{E}(\sum_{j=1}^mm_{i,j}D_j)
\end{align*}
for any real numbers $x_1,\dots,x_c$. In particular,
\begin{align*}
    0&=a(E,X,\sum_{j=1}^ms_j(r_0,\dots,r_c)D_j)\\
    &=a(E,X,\sum_{j=1}^mm_{0,j}D_j)-\sum_{i=1}^{c}r_{i}\mult_{E}(\sum_{j=1}^mm_{i,j}D_j).
\end{align*}
Thus $a(E,X,\sum_{j=1}^mm_{0,j}D_j)=\mult_{E}(\sum_{j=1}^mm_{i,j}D_j)=0$ for any $i$, which implies (2).
\end{proof}

\begin{cor}\label{cor: strong Nakamura approximation}
	Fix three positive integers $d, c, m$ and two non-negative real numbers $\delta>\delta'\ge0$. Let $r_0:=1$, $r_1,\ldots,r_{c}$ be real numbers which are linearly independent over $\Qq$. Let $s_1,\ldots,s_m: \Rr^{c+1}\rightarrow\Rr$ be $\Qq$-linear functions. Then there exist a positive real number $\epsilon$ depending only $d,c,m,r_1,\ldots, r_{c}, s_1,\ldots, s_m$, and a positive real number $\epsilon'$ depending only on $d,c,m, r_1,\ldots, r_{c}, s_1,\ldots, s_m,\delta,\delta'$ satisfying the following. For any normal variety $X$ of dimension $d$ and Weil divisors $D_1,\ldots,D_c\geq 0$ on $X$, let $$\Delta(t):=\sum_{i=1}^ms_i(r_0,\ldots,r_{c-1},t)D_i,$$
	then
	\begin{enumerate}
	\item	 if $(X,\Delta(r_{c}))$ is lc, then $(X,\Delta(t))$ is lc for any $t$ such that $|t-r_{c}|\leq\epsilon$, and
	\item	if $(X,\Delta(r_{c}))$ is $\delta$-lc (resp. $\delta$-klt, $\delta$-plt), then $(X,\Delta(t))$ is $\delta'$-lc (resp. $\delta'$-klt, $\delta'$-plt) for any $t$ such that $|t-r_{c}|\leq\epsilon'$.
\end{enumerate}
\end{cor}

\begin{proof}

We first show (1). Let $f:Y\rightarrow X$ be a dlt modification of $(X,\Delta(r_{c}))$, such that  
$$K_Y+\sum_{i=1}^ms_i(r_0,\ldots,r_{c})D_{Y,i}+E=f^{*}(K_X+\Delta(r_{c})),$$
where $D_{Y,i}$ is the strict transform of $D_i$ on $Y$ for any $1\leq i\leq m$, and $E$ is the sum of the reduced exceptional divisors of $f$. 

By Lemma \ref{lem: dltmodelQfactorialcoeff1},
$K_X+\Delta(t)$ is $\Rr$-Cartier and
$$K_Y+\sum_{i=1}^ms_i(r_0,\ldots,r_{c-1},t)D_{Y,i}+E=f^*(K_X+\Delta(t))$$
for any $t\in\Rr$. 

For any $i$, suppose that $s_i(x_0,\dots,x_c)=\sum_{j=0}^cm_{i,j}x_j$. Let 
$$s_i'(x_0,\dots,x_c):=\sum_{j=0}^c(\frac{|r_j|}{r_j}\cdot m_{i,j})x_j.$$
Then $s_i'$ is a $\Qq$-linear function, and $s_i(r_0,\ldots,r_{c})=s_i'(|r_0|,\ldots,|r_{c}|)$ for every $i$.
Since $(Y,\sum_{i=1}^m s_i'(|r_0|,\ldots,|r_{c}|)D_{Y,i}+E)$ is $\Qq$-factorial dlt, by Theorem \ref{thm: Nak16}, there exists a positive real number $\epsilon$ depending only on $d,c,m$, $r_1,\ldots, r_{c}$ and $s_1,\ldots, s_m$, such that
$(Y,\sum_{i=1}^ms_i(r_0,\ldots,r_{c-1},t)D_{Y,i}+E)$ is lc for any $t$ such that $|t-r_{c}|\leq\epsilon$.
Thus $(X,\Delta(t))$ is lc for any $t$ such that $|t-r_{c}|\leq\epsilon$.

\medskip

(2) follows from (1) by taking $\epsilon':=\frac{\delta-\delta'}{\delta}\epsilon$.
\end{proof}

The next theorem shows the existence of uniform lc rational polytopes, which follows from Corollary \ref{cor: strong Nakamura approximation}.

Recall that we say $V\subset\Rr^m$ is the \emph{rational envelope} of $\bm{v}\in\Rr^m$ if $V$ is the smallest affine subspace containing $\bm{v}$ which is defined over the rationals.

\begin{thm}[Uniform lc rational polytopes]\label{thm: Uniform perturbation of lc pairs}
Let $d,m$ be two positive integers, $\delta>\delta'\geq 0$ two real numbers, $\bm{v}=(v_1,\ldots,v_m)\in\Rr^m$ a point, and $V\subset\mathbb R^m$ the rational envelope of $\bm{v}$. Then there exists an open set $U\ni\bm{v}$ of $V$ depending only on $d,m,\delta,\delta'$ and $\bm{v}$ satisfying the following.  

Let $X$ be a normal variety of dimension $d$, and $B_1,\ldots,B_m\geq 0$ Weil divisors on $X$. If $(X,\sum_{i=1}^m v_iB_{i})$ is lc, then $(X,\sum_{i=1}^m v_i^0B_{i})$ is lc for any point $(v_1^0,\ldots,v_m^0)\in U$.	

Moreover, if $(X,\sum_{i=1}^m v_iB_{i})$ is $\delta$-plt (resp. $\delta$-klt, $\delta$-lc), then $(X,\sum_{i=1}^m v_i^0B_{i})$ is $\delta'$-plt (resp. $\delta'$-klt, $\delta'$-lc).
\end{thm}

\begin{rem}
If $v_i\in\Qq$ for any $i$, then $U=V=\{\bm{v}\}$, and the theorem is trivial. The key point of Theorem \ref{thm: Uniform perturbation of lc pairs} is that $U$ does not depend on $X$.
\end{rem}
\begin{proof}[Proof of Theorem \ref{thm: Uniform perturbation of lc pairs}]

There exist real numbers $r_0:=1,r_1,\ldots,r_c$, such that $\{r_0,\ldots,r_{c}\}$ is a basis of $\Span_{\Qq}(\{1,v_1,\ldots,v_{m}\})$ over $\Qq$ for some $0\le c\le m$. When $c=0$, $U=V=\{\bm{v}\}$, and the theorem holds. In the following, we may assume that $c\geq 1$. 

There exist $\Qq$-linear functions $s_i:\Rr^{c+1}\to \Rr$, such that $s_i(r_0,\ldots,r_{c})=v_i$ for $1\le i\le m$, and the map 
$$(s_1(1,x_1,\ldots,x_{c}),\ldots,s_m(1,x_1,\ldots,x_{c})):\Rr^{c}\to V$$
is one-to-one. It suffices to show that there exists an open set $U_c$ of $\Rr^{c}$, such that $(r_1,\ldots,r_c)\in U_c$, and $(X,\sum_{i=1}^m s_i(1,x_1,\ldots,x_{c})B_{i})$ is lc (resp. $\delta$-plt, $\delta$-klt, $\delta$-lc) for any $(x_1,\ldots,x_c)\in U_c$. 

We prove the theorem by induction on $c$. When $c=1$, the theorem follows from Corollary \ref{cor: strong Nakamura approximation}.

Assume that $c\ge 2$. For simplicity, for every $1\leq i\leq m$, we let
$$s_i(t):=s_i(r_0,\ldots,r_{c-1},t).$$ 
By Corollary \ref{cor: strong Nakamura approximation}, there 
exist two positive real numbers $\epsilon_1,\epsilon_2$ depending only on $d,c,m$, $r_1,\ldots,r_c$, $s_1,\ldots,s_m$, $\delta'$ and $\delta$, such that $r_c+\epsilon_1,r_c-\epsilon_2\in\Qq$, and both  $(X,\sum_{i=1}^m s_i(r_c+\epsilon_1)B_i)$ and $(X,\sum_{i=1}^m s_i(r_c-\epsilon_2)B_i)$ are lc (resp. $\frac{\delta+\delta'}{2}$-plt, $\frac{\delta+\delta'}{2}$-klt, $\frac{\delta+\delta'}{2}$-lc). 

By induction, there exists an open set $U_{c-1}$ of $\Rr^{c-1}$, such that $(r_1,\ldots,r_{c-1})\in U_{c-1}$, and both $$(X,\sum_{i=1}^m s_i(1,x_1,\ldots,x_{c-1},r_{c}+\epsilon_1)B_{i})$$
and 
$$(X,\sum_{i=1}^m s_i(1,x_1,\ldots,x_{c-1},r_{c}-\epsilon_2)B_{i})$$
are lc (resp. $\delta'$-plt, $\delta'$-klt, $\delta'$-lc) for any $(x_1,\ldots,x_{c-1})\in U_{c-1}$.

We get the desired $U_c$ by letting $U_c=U_{c-1}\times (r_{c}-\epsilon_2,r_{c}+\epsilon_1)$.
\end{proof}

\subsection{Uniform $\Rr$-complementary rational polytopes}

\begin{defn}\label{defn: r affine functional divisor}
Let $X$ be a normal variety, $D_i$ distinct prime divisors, and $d_i(t)$ $\Rr$-affine functions: $\Rr\to\Rr$. Then we call the formal finite sum $\sum d_i(t)D_i$ an \emph{$\Rr$-affine functional divisor}.
 
\end{defn}

\begin{defn}
Let $c$ be a non-negative real number, and $\Ii\subset[0,+\infty)$ a set of real numbers. For any $\Rr$-affine functional divisor $\Delta(t)=\sum_{i} d_i(t)D_i$, we write $\Delta(t)\in\mathcal{D}_c(\Ii)$ when the following conditions are satisfied.
\begin{enumerate}
	\item For any $i$, either $d_i(t)=1$, or $d_i(t)$ is of the form $\frac{m-1+f+kt}{m}$, where $m\in\mathbb{N}^{+}$, $f\in\{0\}\cup\{\sum_{1\leq j\leq l}i_j\mid i_j\in\Ii,l\in\mathbb N^+\}$, $k\in\mathbb{Z}$, and
	\item $f+kt$ above can be written as $f+kt=\sum_{j}(f_j+k_jt)$, where $f_j\in\Ii\cup\{0\}$, $k_j\in\mathbb{Z}$, and $f_j+k_jc\ge0$ holds for any $j$.
\end{enumerate}
\end{defn}

\begin{defn}
Let $d$ be a positive integer and  $\Ii\subset[0,+\infty)$ a set of real numbers. We define $\mathcal{B}_{d}(\Ii)\subset [0,+\infty)$ as follows: $c\in\mathcal{B}_{d}(\Ii)$ if and only if there exist a $\Qq$-factorial normal projective variety $X$ and an $\Rr$-affine functional divisor $\Delta(t)$ on $X$ satisfying the following.
\begin{enumerate}
	\item $\dim X\le d$, 
	\item $\Delta(t)\in\mathcal{D}_{c}(\Ii)$,
	\item $(X,\Delta(c))$ is lc,
	\item $K_X+\Delta(c)\equiv0$, and
	\item $K_X+\Delta(c')\not\equiv0$ for some $c'\neq c$.
\end{enumerate}
\end{defn}
We need the following result of Nakamura.
\begin{thm}[{\cite[Theorem 3.8]{Nak16}}]\label{thm:nakacc}
Let $d\ge2$ be an integer and $\Ii\subset[0,+\infty)$ a finite set. The accumulation points of $\mathcal{B}_d(\Ii)$ belong to $\mathcal{B}_{d-1}(\Ii)$. In particular, the accumulation points of $\mathcal{B}_d(\Ii)$ belong to $\Span_{\Qq}(\Ii\cup\{1\})$. 
\end{thm}
\begin{rem}
Theorem \ref{thm:nakacc} is about the accumulation points of numerically trivial pairs for $\Rr$-affine functional divisors $\Delta(t)$. We remark that the global ACC also holds for $\Rr$-affine functional divisors $\Delta(t)$. We refer the readers to \cite{HLQ17} for more details.
\end{rem}

\begin{lem}\label{lem:dltmodel linearmap linear functional div} Let $d,c,m$ be three positive integers, $X$ a normal quasi-projective variety of dimension $d$, and $r_0:=1$, $r_1,\ldots,r_{c}$ real numbers which are linearly independent over $\Qq$. Let $s_1,\ldots,s_m: \Rr^{c+1}\rightarrow\Rr$ be $\Qq$-linear functions, $D_1,\ldots,D_m\geq 0$ Weil divisors on $X$, and  $$\Delta(t):=\sum_{i=1}^ms_i(r_0,\ldots,r_{c-1},t)D_i.$$
Suppose that $(X/Z,\Delta(r_c))$ is $\Rr$-complementary and $(X,\Delta(t_0))$ is lc for some real number $t_0$. Then there exist a birational morphism $f:Y\to X$, a non-negative integer $m'\le m$, and an $\Rr$-affine functional divisor $\Delta_{Y}(t):=\sum_{i=1}^{m'}s_i(r_0,\ldots,r_{c-1},t)D_{Y,i}+E$ on $Y$ satisfying the following.
\begin{enumerate}
\item $D_{Y,i}$ is the strict transform of $D_i$ on $Y$ for any $1\le i\le m'$,
\item $E$ is a reduced divisor,
	\item $(Y/Z,\Delta_Y(r_c))$ is $\Rr$-complementary,
	\item $K_Y+\Delta_{Y}(t_0)\ge f^{*}(K_X+\Delta(t_0))$,
	\item $-(K_Y+\Delta_Y(t_0))\sim_{\Rr,Z} u(K_Y+\Delta)$ for some positive real number $u$ and $\Qq$-factorial dlt pair $(Y/Z,\Delta)$, and
	\item if $X$ is of Fano type over $Z$, then $Y$ is of Fano type over $Z$.
\end{enumerate}
In particular, if $-(K_X+\Delta(t_0))$ is not pseudo-effective over $Z$, then $-(K_Y+\Delta_{Y}(t_0))$ is not pseudo-effective over $Z$, and if $(Y/Z,\Delta_{Y}(t_0))$ is $\Rr$-complementary, then $(X/Z,\Delta(t_0))$ is $\Rr$-complementary.
\end{lem}

\begin{proof}
	Let $(X/Z,\Delta(r_c)+G)$ be an $\Rr$-complement of $(X/Z,\Delta(r_c))$ and $R:=\lfloor \Delta(r_c)+G\rfloor$.
    Let $f: Y\rightarrow X$ be a dlt modification of $(X,\Delta(r_c)+G)$, 
    $$K_Y+\tilde{\Delta(r_c)}+E+\tilde{G}:=f^{*}(K_X+\Delta(r_c)+G),$$
    where $\tilde{\Delta(t)}$ and $\tilde{G}$ are the strict transforms of $\Delta(t)-\Delta(r_c)\wedge R$ and $G-G\wedge R$ on $Y$ respectively, and $E$ is the sum of the reduced exceptional divisors of $f$ and the strict transform of $R$. 
    
   	Let $\Delta_{Y}(t):=\tilde{\Delta(t)}+E$. Then $(Y/Z,\Delta_Y(r_c)+\tilde{G})$ is an $\Rr$-complement of $(Y/Z,\Delta_Y(r_c))$, and $K_Y+\Delta_{Y}(t_0)\ge f^{*}(K_X+\Delta(t_0))$.
   	
   	Since $(Y,\tilde{\Delta(r_c)}+E+\tilde{G})$ is dlt and $\lfloor\tilde{\Delta(r_c)}+E+\tilde{G}\rfloor=E$, there exists a positive real number $\epsilon<1$, such that $\tilde{\Delta(r_c)}-\epsilon\tilde{\Delta(t_0)}\ge0$ and $(Y,\Delta)$ is $\Qq$-factorial dlt, where $\Delta:=\frac{\tilde{\Delta(r_c)}-\epsilon\tilde{\Delta(t_0)}}{1-\epsilon}+\frac{\tilde{G}}{1-\epsilon}+E$. Then 
   	\begin{align*}
   	-\epsilon(K_Y+\Delta_Y(t_0))&=-\epsilon (K_Y+\tilde{\Delta(t_0)}+E)\\
   	                  &\sim_{\Rr,Z} (1-\epsilon)(K_Y+\frac{\tilde{\Delta(r_c)}-\epsilon\tilde{\Delta(t_0)}}{1-\epsilon}+\frac{\tilde{G}}{1-\epsilon}+E)\\
   	                  &=(1-\epsilon)(K_Y+\Delta).
   	\end{align*}
   	
	If $X$ is of Fano type over $Z$, then by Lemma \ref{lem:Fanotypedlt}, $Y$ is of Fano type over $Z$. 
	\end{proof}

\begin{conj}[Existence of good minimal models]\label{conj: exist gmm}
    Let $d$ be a positive integer. Let $(X,B)$ be a $\Qq$-factorial dlt pair of dimension $d$ and $X\to Z$ a contraction. Suppose that $K_X+B$ is pseudo-effective over $Z$. Then there exists a good minimal model $(Y,B_Y)$ of $(X,B)$ over $Z$, that is, $(Y,B_Y)$ is a minimal model of $(X,B)$ over $Z$ and $K_Y+B_Y$ is semiample over $Z$.
\end{conj}

\begin{thm}\label{thm: pepolytope linear maps}
	Fix three positive integers $d, c, m$. Let $r_0:=1,r_1,\ldots,r_{c}$ be real numbers which are linearly independent over $\Qq$. Let $s_1,\ldots, s_m: \Rr^{c+1}\rightarrow\Rr$ be $\Qq$-linear functions. Then there exists a positive real number $\epsilon$ depending only on $d,c,m$, $r_1,\ldots, r_{c}$ and $s_1,\ldots, s_m$ satisfying the following. 
	
  Let $X$ be a normal quasi-projective variety of dimension $d$, $D_1,\ldots,D_m\ge0$ Weil divisors on $X$, and $X\to Z$ a contraction. If $(X/Z,\Delta(r_{c}))$ is $\Rr$-complementary, where
	$$\Delta(t):=\sum_{i=1}^ms_i(r_0,\ldots,r_{c-1},t)D_i,$$	
	then
	\begin{enumerate}
		\item $(X,\Delta(t))$ is lc and $-(K_X+\Delta(t))$ is pseudo-effective over $Z$ for any $t$ such that $|t-r_{c}|\leq\epsilon$, 
		\item if $X$ is of Fano type over $Z$, then $(X/Z,\Delta(t))$ is $\Rr$-complementary for any $t$ such that $|t-r_{c}|\leq\epsilon$, and
		\item suppose that Conjecture \ref{conj: exist gmm} holds in dimension $d$, then $(X/Z,\Delta(t))$ is $\Rr$-complementary for any $t$ such that $|t-r_{c}|\leq\epsilon$. 
	\end{enumerate}
\end{thm}
\begin{proof}We first show (1). Suppose it does not hold. Then by Corollary \ref{cor: strong Nakamura approximation}, for any positive integer $k$, there exist a normal quasi-projective variety $X_k$ of dimension $d$, Weil divisors $D_1^{k},\ldots,D_{m}^k$ on $X_{k}$, and $t_k\in \Rr$ such that $(X_k/Z_k,\Delta_{k}(r_{c}))$ is $\Rr$-complementary, $(X_k,\Delta_k(t_k))$ is lc, $-(K_{X_k}+\Delta_k(t_k))$ is not pseudo-effective over $Z_k$, and $\lim_{k\to +\infty} t_k=r_c$, where
$$\Delta_k(t):=\sum_{i=1}^m s_i(r_0,\ldots,r_{c-1},t)D_{i}^k.$$
	
	By Lemma \ref{lem:dltmodel linearmap linear functional div} and Corollary \ref{cor: strong Nakamura approximation}, possibly passing to a subsequence of $X_k$ and replacing $X_k$ with a birational model, we may assume that  $-(K_{X_k}+\Delta_{k}(t_k))\sim_{\Rr,Z_k} u_k(K_{X_k}+\Delta_k)$, where each $u_k$ is a positive real number which may depend on $X_k$, and each $({X_k},\Delta_k)$ is $\Qq$-factorial dlt.

	Thus we may run a $-(K_{X_k}+\Delta_{k}(t_k))$-MMP over $Z_k$ with scaling of an ample divisor over $Z_k$, and reach a Mori fiber space $Y_k\to Z_k'$ over $Z_k$, such that $-(K_{Y_{k}}+\Delta_{Y_k}(t_k))$ is antiample over $Z_k'$, where $\Delta_{Y_k}(t)$ is the strict transform of $\Delta_{k}(t)$ on $Y_k$ for any $t$. Since $-(K_{X_{k}}+\Delta_{k}(r_c))$ is pseudo-effective over $Z_k$, $-(K_{Y_{k}}+\Delta_{Y_k}(r_c))$ is nef over $Z_k'$.
	
	Thus there exist real numbers $\eta_k$, such that $\min\{t_k,r_c\}\le \eta_k\le \max\{t_k,r_c\}$, $\lim_{k\to+\infty}\eta_k=r_c$, and 
	$$K_{F_k}+\Delta_{F_k}(\eta_k):=(K_{Y_k}+\Delta_{Y_k}(\eta_k))|_{F_k}\equiv 0,$$
	where $F_k$ is a general fiber of $Y_k\to Z_k'$. Since $(X_k/Z_k,\Delta_{k}(r_{c}))$ is $\Rr$-complementary, $({Y_k},\Delta_{Y_k}(r_c))$ is lc. Possibly passing to a subsequence of $t_k$, by Corollary \ref{cor: strong Nakamura approximation}, we may assume that $({Y_k},\Delta_{Y_k}(t_k))$ is lc for any $k$. Thus $({F_k},\Delta_{F_k}(\eta_k))$ is lc. Since $\lim_{k\rightarrow+\infty} \eta_k=r_{c}$, by Theorem \ref{thm:nakacc}, $r_{c}\in\Span_{\Qq}(\{r_0,\ldots,r_{c-1}\})$, a contradiction.
	
	\medskip 
	
	We now show (2)--(3). For any $t_0\in [r_c-\epsilon,r_c+\epsilon]$, by Lemma \ref{lem:dltmodel linearmap linear functional div}, possibly replacing $X$ with a birational model, we may assume that $-(K_{X}+\Delta_{k}(t_0))\sim_{\Rr,Z} u(K_{X}+\Delta)$ for some positive real number $u$ and $\Qq$-factorial dlt pair $({X},\Delta)$.
	
	Suppose that either $X$ is of Fano type over $Z$ or Conjecture \ref{conj: exist gmm} holds in dimension $d$. Then there exists a good minimal model 
	$-(K_{Y_{t_0}}+\Delta_{Y_{t_0}}(t_0))$ of $-(K_X+\Delta(t_0))$ over $Z$, where $\Delta_{Y_{t_0}}(t)$ is the strict transform of $\Delta(t)$ on $Y_{t_0}$ for any $t$. Since $(X/Z,\Delta(r_c))$ is $\Rr$-complementary, $(Y_{t_0},\Delta_{Y_0}(r_c))$ is lc. Therefore, by Corollary \ref{cor: strong Nakamura approximation}, $(Y_{t_0},\Delta_{Y_{t_0}}(t_0))$ is lc. Since $-(K_{Y_{t_0}}+\Delta_{Y_{t_0}}(t_0))$ is semiample over $Z$, $(Y_{t_0}/Z,\Delta_{Y_{t_0}}(t_0))$ is $\Rr$-complementary. Hence $(X/Z,\Delta(t_0))$ is $\Rr$-complementary. 
\end{proof}

\begin{thm}[Uniform $\Rr$-complementary rational polytopes]\label{thm: Uniform perturbation of R complements}
Let $d,m$ be two positive integers, $\bm{v}=(v_1,\ldots,v_m)\in\Rr^m$ a point, and $V\subset\mathbb R^s$ the rational envelope of $\bm{v}$. Then there exists an open set $U\ni\bm{v}$ of $V$ depending only on $d,m$ and $\bm{v}$ satisfying the following.  

Let $(X,B:=\sum_{i=1}^m v_iB_i)$ be an lc pair of dimension $d$ and $X\to Z$ a contraction, such that each $B_i\geq 0$ is a Weil divisor, and $(X/Z,B)$ is $\Rr$-complementary. Suppose that either $X$ is of Fano type over $Z$ or Conjecture \ref{conj: exist gmm} holds in dimension $d$. Then $(X/Z,\sum_{i=1}^m v_i^0B_{i})$ is $\Rr$-complementary for any point $(v_1^0,\ldots,v_m^0)\in U$.

\end{thm}
The proof of Theorem \ref{thm: Uniform perturbation of R complements} is of the same lines as the proof of Theorem \ref{thm: Uniform perturbation of lc pairs}. For readers' convenience, we give a full proof here.
\begin{proof}
There exist real numbers $r_0:=1,r_1,\ldots,r_c$, such that $\{r_0,\ldots,r_{c}\}$ is a basis of $\Span_{\Qq}(\{1,v_1,\ldots,v_{m}\})$ over $\Qq$ for some $0\le c\le m$.  When $c=0$, $U=V=\{\bm{v}\}$, and the theorem holds. In the following, we may assume that $c\geq 1$.

There exist $\Qq$-linear functions $s_1,\dots, s_m:\Rr^{c+1}\to \Rr$, such that $s_i(r_0,\ldots,r_{c})=v_i$ for any $1\le i\le m$, and the map 
$$(s_1(1,x_1,\ldots,x_{c}),\ldots,s_m(1,x_1,\ldots,x_{c})):\Rr^{c}\to V$$
is one-to-one. It suffices to show there exists an open set $U_c$ of $\Rr^{c}$, such that $(r_1,\ldots,r_c)\in U_c$ and $(X/Z,\sum_{i=1}^m s_i(1,x_1,\ldots,x_{c})B_{i})$ is $\Rr$-complementary for any $(x_1,\ldots,x_c)\in U_c$. 

We prove the theorem by using induction on $c$. If $c=1$, then the theorem follows from Theorem \ref{thm: pepolytope linear maps}.

Assume that $c\ge 2$. For every $1\leq i\leq m$, we let
$$s_i(t):=s_i(r_0,\ldots,r_{c-1},t).$$ 
By Theorem \ref{thm: pepolytope linear maps}, there 
exist two positive real numbers $\epsilon_1,\epsilon_2$ depending only on $d,c,m$, $r_1,\ldots,r_c$, and $s_1,\ldots,s_m$, such that $r_c+\epsilon_1,r_c-\epsilon_2\in\Qq$, and both  $(X/Z,\sum_{i=1}^m s_i(r_c+\epsilon_1)B_i)$ and $(X/Z,\sum_{i=1}^m s_i(r_c-\epsilon_2)B_i)$ are $\Rr$-complementary. 

By induction, there exists an open set $U_{c-1}$ of $\Rr^{c-1}$, such that $(r_1,\ldots,r_{c-1})\in U_{c-1}$, and for any $(x_1,\ldots,x_{c-1})\in U_{c-1}$, both $$(X/Z,\sum_{i=1}^m s_i(1,x_1,\ldots,x_{c-1},r_{c}+\epsilon_1)B_{i})$$
and 
$$(X/Z,\sum_{i=1}^m s_i(1,x_1,\ldots,x_{c-1},r_{c}-\epsilon_2)B_{i})$$
are $\Rr$-complementary respectively.

We get the desired $U_c$ by letting $U_c=U_{c-1}\times (r_{c}-\epsilon_2,r_{c}+\epsilon_1)$.
\end{proof}


\subsection{Proof of Theorem \ref{thm: existence ni1i2 complement}}
Theorem \ref{thm: dcc limit lc divisor} and Theorem \ref{thm: dcc limit divisor} are inspired by \cite[\S 4]{PS09} and \cite[Proposition 2.50]{Bir19}. 
\begin{lem}\label{lem: accdcc projection}
Let $\alpha$ be a positive real number, $M\geq 1$ a real number, $\Ii\subset [0,M]$ a DCC set, and $\Ii''=\bar{\Ii''}$ an ACC set. Then there exist a finite set $\Ii'\subset \bar\Ii$ and a projection $g:\bar\Ii\to \Ii'$ (i.e., $g\circ g=g$) depending only on $\alpha,M$, $\Ii$ and $\Ii''$ satisfying the following properties:
 \begin{enumerate}
 	\item $\gamma+\alpha \ge g(\gamma)\ge \gamma$ for any $\gamma\in \Ii$, 
 	\item $g(\gamma')\ge g(\gamma)$ for any $\gamma'\ge \gamma$ such that $\gamma,\gamma'\in\Ii$, and
 	\item for any $\beta\in \Ii''$ and $\gamma\in \Ii$, if $\beta\ge \gamma$, then $\beta\ge g(\gamma)$.
 \end{enumerate}
\end{lem}
\begin{proof}	
	We may replace $\Ii$ with $\bar\Ii$ and therefore assume that $\Ii=\bar\Ii$. Let $N:=\lceil\frac{M}{\alpha}\rceil,\Ii_0:=\Ii\cap [0,\frac{1}{N}]$, and
	$$\Ii_k:=\Ii\cap (\frac{k}{N},\frac{k+1}{N}]$$
	for any $1\leq k\leq N-1$.
	
	For any $\gamma\in\Ii$, there exists a unique $0\leq k\leq N-1$ such that $\gamma\in\Ii_{k}$. If $\gamma\in\Ii$ and $\gamma>\max_{\beta\in\Ii''}\{\beta\}$, then $g(\gamma):=\max_{\beta\in \Ii_k}\{\beta\}$ has the required properties. In the following, we may assume that for any $\gamma\in\Ii$, $\gamma\le\max_{\beta\in\Ii''}\{\beta\}$.
	
	Since $\Ii=\bar\Ii$ and $\Ii''=\bar\Ii''$, we may define $f:\Ii\to\Ii'', g:\Ii\to\Ii$ in the following ways:
	$$f(\gamma):=\min\{\beta\in\Ii''\mid \beta\geq \gamma\},\text{and}\ g(\gamma):=\max\{\beta\in\Ii_k\mid\beta\leq f(\gamma),\gamma\in\Ii_k\}.$$
	For any $\gamma\in\Ii$, it is clear that 
	$$0\le g(\gamma)-\gamma\leq\frac{1}{N}\leq\alpha.$$
	Since $f(\gamma)\ge g(\gamma)\geq \gamma$, 
	$f(g(\gamma))=f(\gamma)$ and $g(g(\gamma))=g(\gamma)$. We will show that $\Ii':=\{g(\gamma)\mid\gamma\in\Ii\}$ has the required properties.
	
	It suffices to show that $\Ii'$ is a finite set. Since $\Ii'\subseteq\Ii$, $\Ii'$ satisfies the DCC. We only need to show that $\Ii'$ satisfies the ACC. Suppose that there exists a strictly increasing sequence $g(\gamma_1)<g(\gamma_2)<\ldots,$ where $\gamma_j\in \Ii$. Since $f(\gamma_j)$ belongs to the ACC set $\Ii''$, possibly passing to a subsequence, we may assume that $f(\gamma_j)$ is decreasing. Thus $g(\gamma_j)$ is decreasing, and we get a contradiction. Therefore $\Ii'$ satisfies the ACC. 	
\end{proof}

\begin{thm}\label{thm: dcc limit lc divisor}
Let $d$ be a positive integer, $\alpha$ a positive real number, and $\Ii\subset [0,1]$ a DCC set. Then there exist a finite set $\Ii'\subset \bar\Ii$ and a projection $g:\bar\Ii\to \Ii'$ depending only on $d,\alpha$ and $\Ii$ satisfying the following. 

Let $(X,B:=\sum_{i=1}^s b_iB_i)$ be an lc pair of dimension $d$, such that each $b_i\in\Ii$ and each $B_i\geq 0$ is a $\Qq$-Cartier Weil divisor. Then


\begin{enumerate}
	\item  $\gamma+\alpha \ge g(\gamma)\ge \gamma$ for any $\gamma\in\Ii$, 
\item $g(\gamma')\ge g(\gamma)$ for any $\gamma,\gamma'\in\Ii$ such that $\gamma'\ge \gamma$, and
	\item  $(X,\sum_{i=1}^s g(b_i)B_i)$ is lc.
\end{enumerate}
\end{thm}

\begin{proof}

We may replace $\Ii$ with $\bar\Ii$ and assume that $\Ii=\bar\Ii$. Let $\Ii'':=\overline{\LCT(\Ii,\mathbb N,d)}$. By Theorem \ref{thm: acc lct}, $\Ii''$ is an ACC set. 

By Lemma \ref{lem: accdcc projection}, there exist a finite set $\Ii'\subset\bar{\Ii}$, and a projection $g:\bar\Ii\to \Ii'$ satisfying Lemma \ref{lem: accdcc projection}(1)--(3).

It suffices to show that $(X,\sum_{i=1}^s g(b_i)B_i)$ is lc. Otherwise, there exists some $0\le j\le s-1$, such that $(X,\sum_{i=1}^j g(b_i)B_i+\sum_{i=j+1}^s b_iB_i)$ is lc, and $(X,\sum_{i=1}^{j+1} g(b_i)B_i+\sum_{i=j+2}^s b_iB_i)$ is not lc. Let 
$$\beta:=\lct(X,\sum_{i=1}^{j} g(b_i)B_i+\sum_{i=j+2}^s b_iB_i;B_{j+1}).$$
Then $g(b_{j+1})> \beta\ge b_{j+1}$. Since $g(b_i),b_i\in\Ii$ for any $i$, we have $\beta\in \Ii''$ which contradicts Lemma \ref{lem: accdcc projection}(3).
\end{proof}

Lemma \ref{lem:dltmodel linearmap DCC coeff} is similar to Lemma \ref{lem:dltmodel linearmap linear functional div}. Recall that for any $\Rr$-divisor $B$, $\Coeff(B)$ stands for the set of all the coefficients of $B$.
\begin{lem}\label{lem:dltmodel linearmap DCC coeff}
Let $(X,B)$ and $(X,B')$ be two lc pairs such that $\Supp B=\Supp B'$, and $X\to Z$ a contraction. Suppose that $(X/Z,B)$ is $\Rr$-complementary. Then there exist a birational morphism $f:Y\to X$, and two $\Qq$-factorial lc pairs $(Y,B_Y)$ and $(Y,B'_Y)$ satisfying the following.
\begin{enumerate}
	\item $\Coeff(B_Y)\subset \Coeff(B)\cup\{1\}$, $\Coeff(B_Y')\subset \Coeff (B')\cup\{1\}$,
	\item $(Y/Z,B_Y)$ is $\Rr$-complementary,
		\item $K_Y+B'_Y\ge f^{*}(K_X+B')$,
	\item $-(K_Y+B_Y')\sim_{\Rr,Z} u(K_Y+\Delta)$ for some positive real number $u$ and dlt pair $(Y,\Delta)$, and
	\item if $X$ is of Fano type over $Z$, then $Y$ is of Fano type over $Z$.
\end{enumerate}
In particular, if $-(K_X+B')$ is not pseudo-effective over $Z$, then $-(K_Y+B'_Y)$ is not pseudo-effective over $Z$, and if $(Y/Z,B'_Y)$ is $\Rr$-complementary, then $(X/Z,B')$ is $\Rr$-complementary.
\end{lem}

\begin{proof}
	Let $(X/Z,B+G)$ be an $\Rr$-complement of $(X/Z,B)$ and $R:=\lfloor B+G\rfloor$.
    Let $f: Y\rightarrow X$ be a dlt modification of $(X,B+G)$, such that
    $$K_Y+\tilde{B}+E+\tilde{G}=f^{*}(K_X+B+G),$$
    where $\tilde{B}$ and $\tilde{G}$ are the strict transforms of $B-B\wedge R$ and $G-G\wedge R$ on $Y$ respectively, and $E$ is the sum of the reduced exceptional divisors of $f$ and the strict transform of $R$. 
    
   	Let $\tilde{B}'$ be the strict transform of $\tilde{B}'-\tilde{B}'\wedge R$ on $Y$, $B_Y:=\tilde{B}+E$, and $B'_Y:=\tilde{B}'+E$. Then $\Coeff(B_Y)\subset \Coeff(B)\cup\{1\}$, $\Coeff(B_Y')\subset \Coeff(B')\cup\{1\}$,  $(Y/Z,B_Y+\tilde{G})$ is an $\Rr$-complement of $(Y/Z,B_Y)$, and $K_Y+B'_Y\ge f^{*}(K_X+B')$.
   	
   	Since $\Supp \tilde{B}'=\Supp \tilde{B}$, $(Y,\tilde{B}+E+\tilde{G})$ is dlt and $\lfloor\tilde{B}+E+\tilde{G}\rfloor=E$, there exists a positive real number $\epsilon<1$, such that $\tilde{B}-\epsilon\tilde{B}'\ge0$, and $(Y,\Delta)$ is dlt, where $\Delta:=\frac{\tilde{B}-\epsilon\tilde{B}'}{1-\epsilon}+\frac{\tilde{G}}{1-\epsilon}+E$. Then 
   	\begin{align*}
   	-\epsilon(K_Y+B_Y')&=-\epsilon (K_Y+\tilde{B}'+E)\\
   	                  &\sim_{\Rr,Z} (1-\epsilon)(K_Y+\frac{\tilde{B}-\epsilon\tilde{B}'}{1-\epsilon}+\frac{\tilde{G}}{1-\epsilon}+E)\\
   	             &=(1-\epsilon)(K_Y+\Delta).
   	\end{align*}
    
   	
	If $X$ is of Fano type over $Z$, then by Lemma \ref{lem:Fanotypedlt}, $Y$ is of Fano type over $Z$. 
\end{proof}

\begin{thm}\label{thm: dcc limit divisor}
	Let $d$ be a positive integer and $\Ii\subset [0,1]$ a DCC set. Then there exist a finite set $\Ii'\subset \bar\Ii$ and a projection $g:\bar\Ii\to \Ii'$ depending only on $d$ and $\Ii$ satisfying the following.

Let $(X,B:=\sum b_iB_i)$ be an lc pair of dimension $d$ and $X\to Z$ a contraction, such that each $b_i\in\Ii$ and each $B_i\geq 0$ is a $\mathbb Q$-Cartier Weil divisor. If $(X/Z,B)$ is $\Rr$-complementary, then
	\begin{enumerate}
		 \item  $\gamma+\alpha \ge g(\gamma)\ge \gamma$ for any $\gamma\in\Ii$, 
\item $g(\gamma')\ge g(\gamma)$ for any $\gamma,\gamma'\in\Ii$ such that $\gamma'\ge \gamma$,
		\item  $(X,\sum_{i} g(b_i)B_i)$ is lc, 
		\item $-(K_X+\sum_{i} g(b_i)B_i)$ is pseudo-effective over $Z$,
		\item if $Z$ is a closed point, then there exists an $\Rr$-divisor $D\ge0$, such that $-(K_X+\sum_{i} g(b_i)B_i)\sim_{\Rr} D$,
		\item if $X$ is of Fano type, then $(X/Z,\sum_{i} g(b_i)B_i)$ is $\Rr$-complementary, and
		\item suppose that Conjecture \ref{conj: exist gmm} holds in dimension $d$, then $(X/Z,\sum_{i} g(b_i)B_i)$ is $\Rr$-complementary. 
	\end{enumerate}
\end{thm}
\begin{proof}We first show that there exist a finite set $\Ii'\subset \bar\Ii$ and a projection $g:\bar\Ii\to \Ii'$ satisfying (1)--(4). Otherwise, by Theorem \ref{thm: dcc limit lc divisor}, there exist a sequence of lc pairs $(X_{k},B_{(k)}:=\sum_{i} b_{k,i}B_{k,i})$ of dimension $d$ and a sequence of projections $g_{k}:\bar\Ii\to \bar{\Ii}$, such that for any $k,i$, $b_{k,i}\in\Ii$, $B_{k,i}\geq 0$ is a $\mathbb Q$-Cartier Weil divisor, $(X_{k}/Z_k,B_{(k)})$ is $\Rr$-complementary, $(X_{k},B_{(k)}':=\sum_{i} g_k(b_{k,i})B_{k,i})$ is lc, $b_{k,i}+\frac{1}{k}\ge g_k(b_{k,i})\ge b_{k,i}$, and $-(K_{X_{k}}+B_{(k)}')$ is not pseudo-effective over $Z_k$. 
	
	By Lemma \ref{lem:dltmodel linearmap DCC coeff} and Theorem \ref{thm: dcc limit lc divisor}, possibly passing to a subsequence of $X_k$ and replacing $X_k$ with a birational model, we may assume that $-(K_{X_{k}}+B_{(k)}')\sim_{\Rr,Z_k}u_k(K_{X_k}+\Delta_k)$, where each $u_k$ is a positive real number which may depend on $X_k$, and each $(X_k,\Delta_k)$ is a $\Qq$-factorial dlt pair.
	
	Thus we may run a $-(K_{X_{k}}+B_{(k)}')$-MMP over $Z_k$ with scaling of an ample divisor over $Z_k$, and reach a Mori fiber space $Y_k\to Z_k'$ over $Z_k$, such that $-(K_{Y_k}+B_{(Y_{k})}')$ is antiample over $Z_k'$, where $B_{(Y_{k})}'$ is the strict transform of $B_{(k)}'$ on $Y_k$. Since $-(K_{X_k}+B_{(k)})$ is pseudo-effective over $Z_k$, $-(K_{Y_k}+B_{(Y_{k})})$ is nef over $Z_k'$, where $B_{(Y_{k})}$ is the strict transform of $B_{(k)}$ on $Y_k$. 
	
	For each $k$, there exist a positive integer $k_j$ and a positive real number $0\le b_k^{+}\le 1$, such that $b_{k,k_j}\le b_k^{+}< g_k(b_{k,k_j})$ and $K_{F_k}+B_{F_k}^{+}\equiv 0$, where
	$$K_{F_k}+B_{F_k}^{+}:=\left(K_{Y_{k}}+(\sum_{i<k_j}g_k(b_{k,i})B_{Y_k,i})+b_k^{+}B_{Y_k,k_j}+(\sum_{i>k_j}b_{k,i}B_{Y_k,i})\right)|_{F_k},$$
	$B_{Y_k,i}$ is the strict transform of $B_{k,i}$ on $Y_k$ for any $i$, and $F_k$ is a general fiber of $Y_k\to Z_k'$. Since $(X_{k}/Z_k,B_{(k)})$ is $\Rr$-complementary, $(Y_{k},B_{(Y_{k})})$ is lc. Possibly passing to a subsequence, by Theorem \ref{thm: dcc limit lc divisor}, we may assume that $(Y_{k},B_{(Y_{k})}')$ is lc for any $k$. Thus $(F_k,B_{F_k}^{+})$ is lc.
	
	Since $g_k(b_{k,k_j})$ belongs to the DCC set $\bar{\Ii}$ for any $k$, possibly passing to a subsequence, we may assume that $g_k(b_{k,k_j})$ is increasing. Since $g_k(b_{k,k_j})-b_k^{+}>0$ and 
	$$\lim_{k\to +\infty} (g_k(b_{k,k_j})-b_k^{+})=0,$$ 
	by Lemma \ref{lem: strictly increasing seq b_k} below, passing to a subsequence again, we may assume that $b_{k}^{+}$ is strictly increasing.

	Now $K_{F_k}+B_{F_k}^{+}\equiv0$, the coefficients of $B_{F_k}^{+}$ belong to the DCC set $\bar{\Ii}\cup\{b_k^{+}\}_{k=1}^{\infty}$, and $b_{k}^{+}$ is strictly increasing. This contradicts the global ACC \cite[Theorem 1.4]{HMX14}.
	
	\medskip
	
	Next we show (5). If $Z$ is a closed point, then by Lemma \ref{lem:dltmodel linearmap DCC coeff}, possibly replacing $X$ with a birational model, we may assume $-(K_X+\sum_{i} g(b_i)B_i)\sim_{\Rr} u(K_X+\Delta)$ for some positive real number $u$ and $\Qq$-factorial dlt pair $(X,\Delta)$. Since $X$ is of Calabi–Yau type, by \cite[Theorem 1.2]{Has17}, there exists an $\Rr$-divisor $D\ge0$, such that $-(K_X+\sum_{i} g(b_i)B_i)\sim_{\Rr} D$.
	
	\medskip

	It suffices to show (6)--(7). Suppose that either $X$ is of Fano type over $Z$ or Conjecture \ref{conj: exist gmm} holds in dimension $d$. Then there exists a good minimal model 
	$-(K_Y+\sum g(b_i)B_{Y,i})$ of $-(K_X+\sum g(b_i)B_i)$ over $Z$, where $B_{Y,i}$ are strict transforms of $B_i$ on $Y$. Since $(X/Z,\sum b_iB_i)$ is $\Rr$-complementary, $(Y,\sum b_iB_{Y,i})$ is lc. Therefore, $(Y,\sum g(b_i)B_{Y,i})$ is lc by the construction of $g$. Since $-(K_Y+\sum g(b_i)B_{Y,i})$ is semiample over $Z$, $(Y/Z,\sum g(b_i)B_{Y,i})$ is $\Rr$-complementary. Hence $(X/Z,\sum g(b_i)B_i)$ is $\Rr$-complementary.
	\end{proof}

The proof of the following lemma is elementary.
\begin{lem}\label{lem: strictly increasing seq b_k}
Let $\{a_k\}_{k=1}^{\infty}$ be an increasing sequence of real numbers, and $\{b_k\}_{k=1}^{\infty}$ a sequence of real numbers, such that $b_k<a_k$ for any $k$, and $\lim_{k\to+\infty}(a_k-b_k)=0$. Then passing to a subsequence, we may assume that $b_k$ is strictly increasing.
\end{lem}

\begin{proof}
     It suffices to show that for any $k\ge 1,$ there exists an integer $k'>k$ such that $b_{k'}>b_k.$ Let $c_j=a_j-b_j$ for any $j,$ then $c_j>0$ and $\lim_{j\to +\infty}c_j=0$ by assumption. Thus there exists an integer $k'>k,$ such that $c_{k'}<c_k.$ Since $a_{k'}>a_k,b_{k'}=a_{k'}-c_{k'}>a_{k}-c_{k}=b_k.$
\end{proof}

Now we are ready to prove Theorem \ref{thm: existence ni1i2 complement}.

\begin{proof}[Proof of Theorem \ref{thm: existence ni1i2 complement}]
Possibly replacing $(X,B)$ with its dlt modification, we may assume that $(X,B)$ is $\Qq$-factorial dlt. Write $B=\sum b_iB_i$, where $B_i$ are distinct prime divisors. Possibly shrinking $Z$ to a neighborhood of $z$, according to Theorem \ref{thm: dcc limit divisor}, there exist a finite set $\Ii'\subset\bar{\Ii}$ and a projection $g:\bar\Ii\to \Ii'$ depending only on $d$ and $\Ii$, such that $(X/Z,B':=\sum g(b_i)B_i)$ is $\Rr$-complementary and $B'\ge B$. 

By Theorem \ref{thm: Uniform perturbation of R complements}, there exist a finite set $\Ii_0=\{a_1,a_2,\dots,a_k\}\subset (0,1]$ and a finite set $\Ii_1$ of non-negative rational numbers depending only on $d$ and $\Ii'$, such that $\sum_{i=1}^ka_i=1$, and $$K_{X}+B'=\sum_{i=1}^k a_i(K_{X}+B'_{i})$$
for some $\Qq$-divisors $B_1',\dots,B_k'\in\Ii_1$ such that $({X}/Z,B'_{i})$ is $\Rr$-complementary for every $i$. Moreover, if $\bar\Ii\subset\mathbb Q$, then we may pick $\Ii_0=\{1\}$ and $B_1'=B'$.

 For each $i$, we may run a $-(K_X+B_i')$-MMP over $Z$ and reach a minimal model $X\dashrightarrow X_i$, such that $-(K_{X_i}+B_{X_i}')$ is nef over $Z$, where $B_{X_i}'$ is the strict transform of $B_i'$ on $X_i$. Since $(X/Z,B_i')$ is $\Rr$-complementary, $(X_i,B_{X_i}')$ is lc. By Theorem \ref{thm:Bir19thm1.8}, $(X_i/Z,B_{X_i}')$ has a monotonic $n$-complement. Therefore, $(X/Z,B_i')$ has a monotonic $n$-complement $(X/Z,B_i^+)$ for any $i$. We get the desired $B^+$ by letting $B^{+}:=\sum_{i=1}^k a_iB_i^{+}$.
\end{proof}

\section{Existence of $n$-complements for pairs with DCC coefficients}

In this section, we will show Theorem \ref{thm: dcc existence n complement}. We first show that Theorem \ref{thm: dcc existence n complement with part two} below implies Theorem \ref{thm: dcc existence n complement}, and then prove Theorem \ref{thm: dcc existence n complement with part two}, which follows from Theorem \ref{thm: existence ni1i2 complement} and Diophantine approximation.

\begin{thm}\label{thm: dcc existence n complement with part two}
	Let $d,p,s$ be three positive integers, $\epsilon$ a positive real number, $\Ii\subset [0,1]$ a DCC set, $||.||$ a norm of $\mathbb R^s$, $\bm{v}=(v_1,\dots,v_s)\in\mathbb R^s$ a point such that $\bm{v}\notin\Qq^s$, $V\subset\mathbb R^s$ the rational envelope of $\bm{v}$, and $\bm{e}=(e_1,\dots,e_s)\in V$ a unit vector, i.e. $||\bm{e}||=1$. Then there exist a positive integer $n$ and a point $\bm{v}_n\in V$ depending only on $d,p,\epsilon,\Ii,||.||,\bm{v}$ and $\bm{e}$ satisfying the following. 
	
	Assume that  $(X,B)$ is an lc pair of dimension $d$, and $X\to Z$ is a contraction, such that
	\begin{itemize}
		\item $X$ is of Fano type over $Z$,
		\item $B\in\Ii$, and
		\item $(X/Z,B)$ is $\Rr$-complementary.
	\end{itemize}
	Then for any point $z\in Z$,
	\begin{enumerate}
	    \item (existence of $n$-complements) there exists an $n$-complement $(X/Z\ni z,B^{+})$ of $(X/Z\ni z,B)$, moreover, if $\bar{\Ii}\subseteq \Qq$, then we may pick $B^{+}\ge B$,
		\item (divisibility) $p\mid n$,
		\item (rationality) $n\bm{v}_{n}\in \mathbb Z^s$,
		\item (approximation) $||\bm{v}-\bm{v}_n||<\frac{\epsilon}{n}$, and
		\item (anisotropic)
		$$||\frac{\bm{v}-\bm{v}_n}{||\bm{v}-\bm{v}_n||}-\bm{e}||<\epsilon.$$
	\end{enumerate}	
\end{thm}
At the first glance, (2)--(5) in Theorem \ref{thm: dcc existence n complement with part two} seem to be technical. However, each of them will be used in the proof of Theorem \ref{thm: dcc existence n complement}. It is expected that Theorem \ref{thm: dcc existence n complement with part two} also holds for the existence of $(\epsilon',n)$-complements (cf. \cite{CH20}) when $\epsilon'$ is a rational number. 

\begin{proof}[Proof of Theorem \ref{thm: dcc existence n complement}] By Theorem \ref{thm: dcc existence n complement with part two}, it suffices to prove the moreover part of Theorem 1.10. 

By Theorem \ref{thm: existence ni1i2 complement}, possibly replacing $\Ii$ with a finite subset of $\bar{\Ii}$ and $(X/Z,B)$ with an $(n,\Ii_0)$-decomposable $\Rr$-complement of $(X/Z,B)$, we may assume that $\Ii=\{v_1,\ldots,v_s\}$ is a finite set. 

Since $\Span_{\Qq_{\geq 0}}(\Ii\backslash\Qq)\cap\mathbb Q\backslash\{0\}=\emptyset,$ by Lemma \ref{lem:finitesetmonotonic}, there exist real numbers $r_0:=1,r_1,\ldots,r_c$ which are linearly independent over $\Qq$, such that $\Ii\subset\Span_{\Qq_{\ge0}}(\{r_0,r_1,\ldots,r_c\})$. If $\Ii\subset \Qq$, i.e., $c=0$, then the theorem follows from Theorem \ref{thm: dcc existence n complement with part two}. If $c\ge 1$, then let $\Ii_{>0}:= \Ii\cap (0,+\infty)$, $\bm{v}:=(v_1,\ldots,v_s)$, $V\subset\Rr^s$ the rational envelope of $\bm{v}$, and $||.||_{\infty}$ the maximal norm on $\Rr^s$. There exist a unit vector $\bm{e}\in V$ and a positive real number $\epsilon<\min \Ii_{>0}$, such that 
$v_1^{+}\ge v_1,\ldots,v_s^{+}\ge v_s$ for any $\bm{v}^{+}=(v_1^{+},\ldots,v_s^{+})\in V$ satisfying
	$$||\frac{\bm{v}-\bm{v}^{+}}{||\bm{v}-\bm{v}^{+}||_{\infty}}-\bm{e}||_{\infty}<\epsilon.$$


By Theorem \ref{thm: dcc existence n complement with part two}, there exist a positive integer $n$ and a point $\bm{v}^{+}=(v_1^{+},\ldots,v_s^{+})\in V$ depending only on $d,p$ and $\Ii$, such that
	\begin{enumerate}
	    \item there exists an $n$-complement $(X/Z\ni z,B^{+})$ of $(X/Z\ni z,B)$, 
	    \item $p\mid n$,
		\item $n\bm{v}^{+}\in \mathbb Z^s$,
		\item $||\bm{v}-\bm{v}^{+}||_{\infty}<\frac{\epsilon}{n}$, and
		\item 		$$||\frac{\bm{v}-\bm{v}^{+}}{||\bm{v}-\bm{v}^{+}||_{\infty}}-\bm{e}||_{\infty}<\epsilon.$$
	\end{enumerate}	
	Let $B:=\sum b_iB_i$ and $B^{+}:=\sum b_i^{+}B_i$, where $B_i$ are distinct prime divisors.
It suffices to show that $b_i^{+}\ge b_i$ for any $i$. For any given $i$, possibly reordering $i$, we may assume that $v_i=b_i$. 
Since $nb_i^{+},nv_i^{+}\in \Zz$ and $v_i^{+}\ge v_i=b_i$, it is enough to show that $nb_i^{+}-nv_i^{+}>-1$ which follows from
$$nb_i^{+}-nv_i^{+}>(n+1)b_i-1-nv_i^{+}=v_i+n(v_i-v_i^{+})-1> v_i-\epsilon-1>-1.$$
\end{proof}

\begin{lem}\label{lem: convexset n vectors}
	Let $\mathcal{D}$ be a compact convex set in $\Rr^n$, then there exist $n+1$ points $\bm{v}_1,\ldots,\bm{v}_{n+1}$ in $\Rr^n$, such that $\mathcal{D}$ is contained in the interior of the convex hull of $\bm{v}_1,\ldots,\bm{v}_{n+1}$.
	
	Moreover, there exists a positive real number $\epsilon$, such that for any $\bm{v}_1',\ldots,\bm{v}_{n+1}'\in\Rr^n$, if $||\bm{v}_i-\bm{v}_i'||<\epsilon$ for any $1\le i\le n+1$, then $\mathcal{D}$ is contained in the convex hull of $\bm{v}_1',\ldots,\bm{v}_{n+1}'$.
\end{lem}
\begin{proof}
	Since $\mathcal{D}$ is a bounded set, there exists a positive real number $M$, such that  
	$$\mathcal{D}\subseteq \{(x_1,\ldots,x_n)\in\Rr^n\mid n\sum_{i=1}^n|x_i|<M\}.$$
	Let $\bm{v}_{n+1}:=(-3M,\ldots,-3M)$ and $\bm{v}_i:=(0,\ldots,3M,\ldots,0)$ for $1\le i\le n$, where $3M$ is the $i$-th coordinate.
	
	For any point $\bm{b}=(b_1,\ldots,b_n)\in\mathcal{D}$, there exists a positive real number $0<a<\frac{1}{3n}$, such that $3aM+b_i\ge 0$ for any $1\le i\le n$. We have 
	$$\bm{b}=\sum_{i=1}^n\frac{3aM+b_i}{3M}\bm{v}_{i}+a\bm{v}_{n+1}+(1-a-\sum_{i=1}^n\frac{3aM+b_i}{3M})(\frac{1}{n+1}\sum _{i=1}^{n+1} \bm{v}_i).$$
	Since 
	$$a+\sum_{i=1}^n\frac{3aM+b_i}{3M}< \frac{1}{3}+\frac{2M}{3M}=1,$$
	$\mathcal{D}$ is contained in the interior of convex hull $\mathcal{V}$ of $\bm{v}_1,\ldots,\bm{v}_{n+1}$. Let $d:=dist(\partial\mathcal{V},\mathcal{D})>0$. Since $d$ is a continuous function of $\bm{v}_1,\ldots,\bm{v}_{n+1}$, there exists a positive real number $\epsilon$, such that for any  $||\bm{v}_i-\bm{v}_i'||<\epsilon$, $d':=dist(\partial\mathcal{V}',\mathcal{D})>0$, where $\mathcal{V}'$ is the convex hull of $ \bm{v}_1',\ldots,\bm{v}_{n+1}'$. In particular, $\mathcal{D}$ in contained in the convex hull of $\bm{v}_1',\ldots,\bm{v}_{n+1}'$.
\end{proof}
\begin{lem}\label{lem:finitesetmonotonic}
Let $\Ii$ be a set of non-negative real numbers. Then the following are equivalent.
\begin{enumerate}
		\item For any finite set $\Ii_0$ of $\Ii$, there exist real numbers $r_0:=1,r_1,\ldots,r_c$ which are linearly independent over $\Qq$, such that $$\Ii_0\subset \Span_{\Qq_{\ge0}}(\{r_0,r_1,\ldots,r_c\}),$$
	\item $\Span_{\Qq_{\ge0}}(\Ii\backslash\Qq)\cap (\Qq\backslash\{0\})=\emptyset$.
\end{enumerate}
\end{lem}
\begin{proof}
Suppose that $\Ii$ satisfies (1). For any $a\in\Span_{\Qq_{\ge0}}(\Ii\backslash\Qq)\cap(\Qq\backslash\{0\})$, there exist a positive integer $m$, positive rational numbers $\lambda_1,\dots,\lambda_m$, and $\alpha_1,\dots,\alpha_m\in\Ii\backslash\Qq$, such that $a=\sum_{i=1}^m\lambda_i\alpha_i$. By (1), there exist a positive integer $c$, positive real numbers $r_0:=1,r_1,\ldots,r_c$ which are linearly independent over $\Qq$, and non-negative rational numbers $d_{i,j}$ for every $1\leq i\leq m$ and $0\leq j\leq c$, such that $\alpha_i=\sum_{j=0}^cd_{i,j}r_j$. By our assumption, for every $1\leq i\leq m$, $d_{i,j}\not=0$ for some $1\leq j\leq c$. Therefore, 
$$a=\sum_{i=1}^m\sum_{j=0}^c\lambda_id_{i,j}r_j=\sum_{j=0}^c(\sum_{i=1}^m\lambda_id_{i,j})r_j\in\Span_{\Qq_{\ge0}}(\{r_0,r_1,\ldots,r_c\})\backslash\Qq,$$
a contradiction.

Suppose that $\Ii$ satisfies (2). For any finite set $\Ii_0=\{\alpha_1,\ldots,\alpha_m\}$ of $\Ii$, there exist positive real numbers $1,r_1',\ldots,r_c'$ which are linearly independent over $\Qq$, such that $\Ii_0\subset \Span_{\Qq}(\{1,r_1',\ldots,r_c'\})$. Let $\bm{r}':=(r_1',\ldots,r_c')$. Then there exist $\bm{a}_1',\ldots,\bm{a}_m'\in\Qq^c$, such that $\alpha_i-{\bm{a}_i}'\cdot{\bm{r}}'\in\Qq$ for any $1\le i\le m$. Possibly reordering the indices, we may assume that $\alpha_1,\ldots,\alpha_{m'}\notin\Qq$ and $\alpha_{m'+1},\ldots,\alpha_{m}\in\Qq$  for some $0\le m'\le m$. Then $\bm{a}'_1,\ldots,\bm{a}'_{m'}\neq \bm{0}$ and $\bm{a}'_{m'+1}=\ldots=\bm{a}'_m=\bm{0}$, where $\bm{0}=(0,0,\ldots,0)$. 

If $m'=0$, then $c=0$, $\Ii_0$ is a finite set of rational numbers, and $r_0:=1$, hence $(1)$ holds. Thus we may assume $m'\ge 1$. 

Suppose that $\bm{0}$ belongs to the convex hull of $\{\bm{a}'_1,\ldots,\bm{a}'_{m'}\}$. There exist non-negative rational numbers $\lambda_i$, such that $\sum_{i=1}^{m'}\lambda_i=1$, and $\sum_{i=1}^{m'} \lambda_i\bm{a}_i'=\bm{0}$. Then $\sum_{i=1}^{m'}\lambda_i\bm{a}_i'\cdot \bm{r}'=\bm{0}\cdot\bm{r}'=0$ and $0\neq\sum_{i=1}^{m'}\lambda_i\alpha_i=\sum_{i=1}^{m'}\lambda_i(\alpha_i-{\bm{a}_i}'\cdot{\bm{r}}')\in\Qq$, which contradicts (2). Hence $\bm{0}$ does not belong to the convex hull of $\{\bm{a}'_1,\ldots,\bm{a}'_{m'}\}$.

By Hahn-Banach theorem, there exist a rational point $\bm{t}\in\Qq^{c}$ and a positive rational number $b$, such that the hyperplane $H:=\{\bm{x}\in\Rr^{c}\mid \bm{t}\cdot\bm{x}=b\}$ intersects the segment $\overrightarrow{\bm{0}\bm{a}'_i}$ for any $1\le i\le m'$. Let $\mathcal{C}$ be the convex cone generated by $\overrightarrow{\bm{0}\bm{a}'_1},\ldots,\overrightarrow{\bm{0}\bm{a}'_{m'}}$. Then $\mathcal{C}\cap H$ is a compact convex set in $H$. By Lemma \ref{lem: convexset n vectors}, there exist rational points $\bm{v}_1,\ldots,\bm{v}_{c}\in H$, such that $\mathcal{C}\cap H$ is contained in the convex hull of $\bm{v}_1,\ldots,\bm{v}_{c}$. Hence $\mathcal{C}$ is contained in the cone generated by $\overrightarrow{\bm{0}\bm{v}_1},\ldots,\overrightarrow{\bm{0}\bm{v}_{c}}$. 
Since $\bm{v}_j$, $\bm{a}_i'$ are rational points, there exist non-negative rational numbers $a_{i,j}$, such that $\bm{a}_i'=\sum_{j=1}^{c} a_{i,j}\bm{v}_j$ for any $1\le i\le m'$. 

Let $r_j=\bm{v}_j\cdot \bm{r}'$ and $\bm{r}=(r_1,\ldots,r_c)$.
Then 
$$\alpha_i-{\bm{a}_i}'\cdot{\bm{r}}'=\alpha_i-\sum_{j=1}^{c} a_{i,j}\bm{v}_j\cdot\bm{r}'=\alpha_i-\sum_{j=1}^{c}a_{i,j}r_j\in\Qq$$
for any $1\le i\le m'$. Possibly replacing $(r_1,\dots,r_c)$ with $(r_1-t,\ldots,r_c-t)$ for some rational numbers $t\gg 0$, we may assume that $\alpha_i-\sum_{j=1}^{c}a_{i,j}r_j\ge 0$ for any $1\le i\le m'$. Therefore $\alpha_i\in\Span_{\Qq_{\ge0}}(\{r_0,\ldots,r_c\})$ for any $1\le i\le m$, and we are done.

\end{proof}

Now we are going to prove Theorem \ref{thm: dcc existence n complement with part two}. 

\begin{lem}\label{lem: equivnorm}
	Let $||.||_1,||$ and $||_2$ be two different norms defined on $\Rr^n$. Then there exists a positive real number $M$, such that
$$||\frac{\bm{a}}{||\bm{a}||_1}-\frac{\bm{b}}{||\bm{b}||_1}||_1\le M||\frac{\bm{a}}{||\bm{a}||_2}-\frac{\bm{b}}{||\bm{b}||_2}||_2$$
for any non-zero vectors $\bm{a},\bm{b}\in \Rr^n$.
\end{lem}
\begin{proof}Since any two norms defined on $\Rr^n$ are equivalent, there exist two positive real numbers $m_1,m_2$ such that for any $\bm{x}\in\Rr^n$,  $$m_1||\bm{x}||_2\le ||\bm{x}||_1\le m_2||\bm{x}||_2.$$
We claim that we may pick $M=\frac{2m_2}{m_1}$. 

Possibly replacing $\bm{a}$ with $\frac{\bm{a}}{||\bm{a}||_2}$ and $\bm{b}$ with $\frac{\bm{b}}{||\bm{b}||_2}$, we may assume that $||\bm{a}||_2=||\bm{b}||_2=1$. We have
\begin{align*}
||\frac{\bm{a}}{||\bm{a}||_1}-\frac{\bm{b}}{||\bm{b}||_1}||_1&=||\frac{\bm{a}}{||\bm{a}||_1}-\frac{\bm{a}}{||\bm{b}||_1}+\frac{\bm{a}}{||\bm{b}||_1}-\frac{\bm{b}}{||\bm{b}||_1}||_1\\
&\le ||\frac{\bm{a}}{||\bm{a}||_1}-\frac{\bm{a}}{||\bm{b}||_1}||_1+||\frac{\bm{a}}{||\bm{b}||_1}-\frac{\bm{b}}{||\bm{b}||_1}||_1\\
&=\frac{||||\bm{b}||_1-||\bm{a}||_1||_1}{||\bm{b}||_1}+\frac{||\bm{a}-\bm{b}||_1}{||\bm{b}||_1}\\
&\le \frac{2||\bm{a}-\bm{b}||_1}{||\bm{b}||_1}\le \frac{2m_2||\bm{a}-\bm{b}||_2}{m_1||\bm{b}||_2}=M||\bm{a}-\bm{b}||_2.
\end{align*}
\end{proof}

 Recall that the maximum norm on $\Rr^s$ is defined by 
 $||(x_1,\ldots,x_s)||_{\infty}=\max\{ |x_{i}|\mid i=1,2,\ldots,s\}$ for any vector $(x_1,\ldots,x_s)\in\Rr^s$.

\begin{lem}\label{lem: linearindependentdiufantu}
	Let $p_0,l,s_0$ be three positive integers, $k_0$ a non-negative integer, $\epsilon_1$ a positive real number, $r_0:=1,r_1,\ldots,r_{s_0+k_0}$ positive real numbers which are linearly independent over $\Qq$, and $\bm{e}'=(e_1,\ldots,e_{s_0})\in\Rr_{\ge0}^{s_0}$ a nonzero vector. Let $||.||_{\infty}$ be the maximum norm, and $\bm{r}=(r_1,\ldots,r_{s_0+k_0})$. Then there exist a positive integer $n_0$ and a point $\bm{r}'=(r_1',\ldots,r_{s_0+k_0}')\in \Rr^{s_0+k_0}$, such that 
	\begin{enumerate}
		\item $p_0\mid n_0$,
		\item $n_0\bm{r}'\in l\Zz^{s_0+k_0}$,
		\item $||\bm{r}-\bm{r}'||_{\infty}<\frac{\epsilon_1}{n_0},$
		 and
		\item $||\frac{\bm{c}-\bm{d}}{||\bm{c}-\bm{d}||_{\infty}}-\frac{\bm{e}'}{||\bm{e}'||_{\infty}}||_{\infty}<\epsilon_1,$
		where $\bm{c}=(r_1,\ldots,r_{s_0})$ and $\bm{d}=(r_1',\ldots,r_{s_0}').$
	\end{enumerate}	
\end{lem}
\begin{proof}
	Possibly reordering coordinates, we may assume that $e_1=||\bm{e}'||_{\infty}>0$. Let $\alpha_1:=\frac{1}{2l}\epsilon_1$. Since $0\le e_i\le e_1$ for any $2\le i\le s_0$, we may pick positive real numbers $\alpha_i$, such that $\alpha_i<\alpha_1$, and
$$|\frac{\alpha_i}{\alpha_1}-\frac{e_i}{e_1}|<\epsilon_1$$
for any $2\le i\le s_0$.  

By the continuity of functions $\frac{\alpha_i+x}{\alpha_1-x}$ and $\frac{\alpha_i-x}{\alpha_1+x}$, there exists a positive real number $\epsilon_2$, such that $\epsilon_2<\alpha_i<\alpha_1-2\epsilon_2$ for any $2\le i\le s_0$, and $$|\frac{\alpha_i+\epsilon_2}{\alpha_1-\epsilon_2}-\frac{e_i}{e_1}|<\epsilon_1,\,|\frac{\alpha_i-\epsilon_2}{\alpha_1+\epsilon_2}-\frac{e_i}{e_1}|<\epsilon_1$$
for any $1\le i\le s_0$. 

Let $\alpha_{j}:=0$ for any $s_0+k_0\ge j>s_0$, and $\bm{\alpha}:=(\alpha_1,\ldots,\alpha_{s_0+k_0})$. By Kronecker's theorem or Weyl's equidistribution theorem, there exist a positive integer $n_1$ and a vector $\bm{\beta}:=(\beta_1,\ldots,\beta_{s_0+k_0})\in \Zz^{s_0+k_0}$, such that
$$||n_1\frac{p_0\bm{r}}{l}-\bm{\beta}-\bm{\alpha}||_{\infty}<\epsilon_2.$$
We claim that we may pick $n_0=n_1p_0$ and $\bm{r}'=\frac{l}{n_0}\bm{\beta}$. 

It is clear that $p_0\mid n_0$, $n_0\bm{r}'\in l\Zz^{s_0+k_0}$, and
\begin{align*}
||\bm{r}-\bm{r}'||_{\infty}&=\frac{l}{n_0}||n_1\frac{p_0\bm{r}}{l}-\bm{\beta}||_{\infty}\\
&\le\frac{l}{n_0}(||n_1\frac{p_0\bm{r}}{l}-\bm{\beta}-\bm{\alpha}||_{\infty}+||\bm{\alpha}||_{\infty})\\
&<\frac{l}{n_0}(\epsilon_2+\alpha_1)<\frac{\epsilon_1}{n_0}.
\end{align*}
It suffices to show $(4)$. Since
$$\frac{n_0}{l}|r_i-r_i'|=|\frac{n_0}{l}r_i-\beta_i|\le |\frac{n_0}{l}r_i-\beta_i-\alpha_i|+\alpha_i\le \epsilon_2+\alpha_i<\alpha_1-\epsilon_2$$
for any $s_0\ge i\ge 2$, and
$$\frac{n_0}{l}|r_i-r_i'|\ge \alpha_i-|\frac{n_0}{l}r_i-\beta_i-\alpha_i|\ge \alpha_i-\epsilon_2$$
for any $s_0\ge i\ge 1$, we have
$$\frac{n_0}{l}|r_1-r_1'|>\frac{n_0}{l}|r_i-r_i'|$$ 
for any $s_0\ge i\ge 2$, and
$$||\bm{c}-\bm{d}||_{\infty}=\max\{|r_i-r_i'|\,|\,i=1,2,\ldots,s_0\}=|r_1-r_1'|.$$
Thus
\begin{align*}
&||\frac{\bm{c}-\bm{d}}{||\bm{c}-\bm{d}||_{\infty}}-\frac{\bm{e}'}{||\bm{e}'||_{\infty}}||_{\infty}=\max\{|\frac{r_i-r_i'}{| r_1-r_1'|}-\frac{e_i}{e_1}|\mid i=1,2,\ldots,s_0\}\\
\le& \max\{|\frac{\alpha_i+\epsilon_2}{\alpha_1-\epsilon_2}-\frac{e_i}{e_1}|,|\frac{\alpha_i-\epsilon_2}{\alpha_1+\epsilon_2}-\frac{e_i}{e_1}|\mid i=2,\ldots,s_0\}<\epsilon_1,
\end{align*}
where we use the fact that $\alpha_i-\epsilon_2>0$ for the first equality.
\end{proof}

\begin{lem}\label{lem: vectordiufantu}
Let $p_0,s,k$ be three positive integers, $\epsilon_0$ a positive real number, $||.||$ a norm of $\mathbb R^s$, $||.||_{\infty}$ the maximum norm, and $\bm{a}=(a_1,\ldots,a_k)\in\Rr_{>0}^k$ and $\bm{v}\in\Rr^s$ two points, such that $\sum_{i=1}^k a_i=1$ and $\bm{v}\notin\Qq^s$. Let $V\subset\Rr^s$ be the rational envelope of $\bm{v}$, and $\bm{e}\in V$ a nonzero vector. Then there exist a positive integer $n_0$, and two vectors $\bm{v}'\in V$ and $\bm{a}'=(a_1',\ldots,a_k')\in\Rr_{>0}^{k}$, such that 
\begin{enumerate}
    \item $p_0\mid n_0$,
    \item $\sum_{i=1}^k a_i'=1$,
    \item $n_0(\bm{v}',\bm{a}')\in \Zz^{s+k}$,
    \item $||\bm{v}-\bm{v}'||<\frac{\epsilon_0}{n_0},||\bm{a}-\bm{a}'||_{\infty}<\frac{\epsilon_0}{n_0},$ and
    \item $||\frac{\bm{v}-\bm{v}'}{||\bm{v}-\bm{v}'||}-\frac{\bm{e}}{||\bm{e}||}||<\epsilon_0.$
\end{enumerate}

\end{lem}
\begin{proof}
Let $\bm{v}_1,\ldots,\bm{v}_{s_0}$ be a basis of $V$, such that $\bm{v}_i\in\Qq^s$ for any $1\leq i\leq s_0$, and $\bm{e}\in\Span_{\Rr_{\ge0}}(\{\bm{v}_1,\ldots,\bm{v}_{s_0}\})$. Then there exist a non-negative integer $k_0$, real numbers $r_1,\ldots,r_{s_0+k_0}$ and  $e_1,\ldots,e_{s_0}\ge0$, such that $r_0:=1,r_1,\ldots,r_{s_0+k_0}$ are 
linearly independent over $\Qq$, $a_i\in \Span_{\Qq}(\{r_0,\ldots,r_{s_0+k_0}\})$ for any $1\le i\le k$, and $$\bm{v}=\sum_{i=1}^{s_0}r_i\bm{v}_i,\bm{e}=\sum_{i=1}^{s_0} e_i\bm{v}_i.$$
Hence there exist $\Qq$-linear functions $a_i(x_0,x_1,\ldots,x_{s_0+k_0})$ and positive integers $l,M_1$, such that for any $i$, $a_i(r_0,r_1,\ldots,r_{s_0+k_0})=a_i$, $la_i(x_0,x_1\ldots,x_{s_0+k_0})$ is a $\Zz$-linear function and 
$$|a_i(x_0,x_1,\ldots,x_{s_0+k_0})|\le M_1\max\{|x_i|\mid i=0,1,\ldots,s_0+k_0\}$$
for any $(x_0,x_1,\ldots,x_{s_0+k_0})\in\Rr^{s_0+k_0+1}$.

We define a norm $||.||_{*}$ on $V$ as follows. For any $\bm{y}\in V$, let
$||\bm{y}||_{*}:=\max\{|y_i|\,|\,i=1,2,\ldots,s_0\}$, where $y_1,\ldots,y_{s_0}$ are unique real numbers such that $\bm{y}=\sum_{i=1}^{s_0} y_i\bm{v}_i$.

By Lemma \ref{lem: equivnorm}, there exists a positive real number $M_2$, such that
$$||\frac{\bm{y}}{||\bm{y}||}-\frac{\bm{z}}{||\bm{z}|||}||\le M_2||\frac{\bm{y}}{||\bm{y}||_{*}}-\frac{\bm{z}}{||\bm{z}||_{*}}||_{*}$$
for any non-zero vectors $\bm{y},\bm{z}\in V$. Moreover, possibly replacing $M_2$ with a larger number, we may assume that  $$||\bm{y}-\bm{z}||\le M_2||\bm{y}-\bm{z}||_{*}$$
for any vectors $\bm{y},\bm{z}\in V$. 

Let $\bm{r}:=(r_1,\ldots,r_{s_0+k_0})$, $a:=\min\{a_1,\ldots,a_k\}$, and $\epsilon_1$ a positive real number such that $$\epsilon_1<\min\{\frac{\epsilon_0}{M_1},\frac{a}{M_1},\frac{\epsilon_0}{M_2}\}.$$ 
Let $\bm{e}':=(e_1,\ldots,e_{s_0})$, and $\bm{c}:=(r_1,\ldots,r_{s_0})$. By Lemma \ref{lem: linearindependentdiufantu}, there exist a positive integer $n_0$ and a point $\bm{r}'\in\Rr^{s_0+k_0}$, such that  
	\begin{itemize}
	\item $p_0\mid n_0$,
	\item $n_0\bm{r}'\in l\Zz^{s_0+k_0}$,
	\item $||\bm{r}-\bm{r}'||_{\infty}<\frac{\epsilon_1}{n_0},$
	and
	\item $||\frac{\bm{c}-\bm{d}}{||\bm{c}-\bm{d}||_{\infty}}-\frac{\bm{e}'}{||\bm{e}'||_{\infty}}||_{\infty}<\epsilon_1,$
	where $\bm{d}=(r_1',\ldots,r_{s_0}')$.
\end{itemize}	
Let $\bm{v}':=\sum_{i=1}^{s_0} r_i'\bm{v}_i$, $a_i':=a_i(1,r_1',\ldots,r_{s_0+k_0}')$ for any $1\le i\le k$, and $\bm{a}':=(a_1',\ldots,a_{k}')$. Since $\sum_{i=1}^k a_i=1$ and $r_0,r_1,\ldots,r_{s_0+k_0}$ are linearly independent over $\Qq$, $\sum_{i=1}^k a_i(1,x_1,\ldots,x_{s_0+k_0})=1$ for any $(x_1,\ldots,x_{s_0+k_0})\in\Rr^{s_0+k_0}$. In particular, $\sum_{i=1}^k a_i'=1$. 

We conclude that $a_i'$ are positive integers since
\begin{align*}
a_i'&\ge a_i-|a_i-a_i'|\ge a_i- M_1||\bm{r}-\bm{r}'||_{\infty}\\
&> a-M_1\cdot\frac{\epsilon_1}{n_0}>0.
\end{align*}
Moreover, we have 
$$n_0a_i'=\frac{1}{l}\cdot ln_0a_i'=\frac{1}{l}\cdot la_i(n_0,n_0r_1',\ldots,n_0r_{s_0+k_0}')\in \Zz,$$

$$||\bm{v}-\bm{v}'||\le M_2||\bm{v}-\bm{v}'||_{*}=M_2||\bm{c}-\bm{d}||_{\infty}\le M_2||\bm{r}-\bm{r}'||<\frac{M_2\epsilon_1}{n_0}<\frac{\epsilon_0}{n_0},$$
$$||\bm{a}-\bm{a}'||_{\infty}\le M_1||\bm{r}-\bm{r}'||_{\infty}<
\frac{M_1\epsilon_1}{n_0}<
\frac{\epsilon_0}{n_0},$$
 and 
 $$||\frac{\bm{v}-\bm{v}'}{||\bm{v}-\bm{v}'||}-\frac{\bm{e}}{||\bm{e}||}||\le M_2||\frac{\bm{c}-\bm{d}}{||\bm{c}-\bm{d}||_{\infty}}-\frac{\bm{e}'}{||\bm{e}'||_{\infty}}||_{\infty}<M_2\epsilon_1<\epsilon_0.$$
\end{proof}

\begin{proof}[Proof of Theorem \ref{thm: dcc existence n complement with part two}]
		By Theorem \ref{thm: existence ni1i2 complement}, there exist a positive integer $n_1$ and a finite set $\Ii_0=\{a_1,\ldots,a_k\}\subseteq (0,1]$ depending only on $d$ and $\Ii$, such that $\sum_{i=1}^{k}a_i=1$, and for any pair $(X,B)$, contraction $X\rightarrow Z$ and $z\in Z$ satisfying the assumptions, there exist an lc pair $(X,\tilde{B})$, distinct reduced divisors $B_1,\dots,B_m$, and lc pairs $(X,B^i:=\sum_{j=1}^m b_{i,j}B_j)$ for any $1\le i\le k$, such that
		\begin{itemize}
		\item $(X/Z\ni z,\tilde B)$ is an $\Rr$-complement of $(X/Z\ni z,B)$,
		\item $\sum_{i=1}^ka_iB^i=\tilde B$, and
		    \item $(X/Z\ni z,B^i)$ is an $n_1$-complement of itself for every $i$. In particular, $b_{i,j}\in\frac{1}{n_1}\mathbb N\cap [0,1]$.
		\end{itemize}
Moreover, if $\bar\Ii\subset\Qq$, we may pick $\Ii_0=\{1\}$ and $\tilde B=B^1$.
We define
		$$T:=\{\frac{a_1}{kn_1^2},\frac{a_2}{kn_1^2},\dots,\frac{a_k}{kn_1^2}\}\cup\{\frac{1}{kn_1}(1-\frac{1}{n_1}\sum_{i=1}^k a_im_{i})\mid m_{i}\in\mathbb N\},$$ 	
	$T_{>0}:=T\cap (0,+\infty)$, and $\epsilon_0$ a positive real number such that
	$$\epsilon_0<\min\{\frac{\epsilon}{n_1},T_{>0}\}.$$
	
		By Lemma \ref{lem: vectordiufantu}, there exist a positive integer $n_0$, $\bm{v}'\in\Rr^{s}$, and $\bm{a}':=(a_1',\ldots,a_k')\in\Rr^k$ such that 
	\begin{itemize}
		\item $p\mid n_0$,
		\item $\sum_{i=1}^k a_i'=1$,
		\item $n_0(\bm{v}',\bm{a}')\in \Zz^{s+k}$,
		\item $||\bm{v}-\bm{v}'||<\frac{\epsilon_0}{n_0}
		,||\bm{a}-\bm{a}'||_{\infty}<\frac{\epsilon_0}{n_0},$ 	    
		and
		\item $||\frac{\bm{v}-\bm{v}'}{||\bm{v}-\bm{v}'||}-\frac{\bm{e}}{||\bm{e}||}||<\epsilon_0.$
	\end{itemize}
   Let $n:=n_0n_1$ and $B^{+}:=\sum_{i=1}^k a_i'B^i$. Then $p\mid n$, $n\bm{v}'\in \Zz^s$, $$||\bm{v}-\bm{v}'||<\frac{\epsilon_0}{n_0}<\frac{\epsilon}{n}
   ,\,||\bm{a}-\bm{a}'||_{\infty}<\frac{\epsilon_0}{n_0}<\frac{\epsilon}{n},$$
   and	
   $$||\frac{\bm{v}-\bm{v}'}{||\bm{v}-\bm{v}'||}-\frac{\bm{e}}{||\bm{e}||}||<\epsilon_0<\epsilon.$$ 
   In particular, when $\bar\Ii\subset\Qq$, $a_1'=1$, $B^+=B^1=\tilde B$, and $(X/Z\ni z,B^+)$ is a monotonic $n$-complement of $(X/Z\ni z,B)$.
   
    We claim that $(X/Z\ni z,B^{+})$ is an $n$-complement of $(X/Z\ni z,\tilde B)$, hence an $n$-complement of $(X/Z\ni z,B)$, and we are done. Possibly shrinking $Z$ to a neighborhood of $z$, we have 
   $$n(K_X+B^{+})=n\sum_{i=1}^k a_i'(K_X+B^i)=\sum_{i=1}^k\frac{a_i'n_0n_1}{n_1}\cdot n_1(K_X+B^i)\sim_{Z}0,$$
   and $(X,B^{+})$ is lc. In particular, $n\sum_{i=1}^k a_i'b_{i,j}$ is an integer for any $1\le j\le m$.
   
   If $\sum_{i=1}^k a_ib_{i,j}=1$ for some $1\le j\le m$, then $b_{i,j}=1$ for any $1\le i\le k$ as $\sum_{i=1}^k a_i=1$ and $0\le b_{i,j}\le 1$. Thus $\sum_{i=1}^k a_i'b_{i,j}=1$.
   
   It suffices to show that if $\sum_{i=1}^k a_ib_{i,j}<1$ for some $1\le j\le m$, then
   $$n\sum_{i=1}^k a_i'b_{i,j}\ge \lfloor (n+1)\sum_{i=1}^k a_ib_{i,j}\rfloor.$$ 
   We may assume that $\sum_{i=1}^k a_i'b_{i,j}<1$ and $\sum_{i=1}^k a_ib_{i,j}>0$. Since $n\sum_{i=1}^k a_i'b_{i,j}\ge \lfloor (n+1)\sum_{i=1}^k a_i'b_{i,j}\rfloor,$  
   we only need to prove 
      $$\lfloor (n+1)\sum_{i=1}^k a_ib_{i,j}\rfloor=\lfloor (n+1)\sum_{i=1}^k a_i'b_{i,j}\rfloor=n\sum_{i=1}^k a_i'b_{i,j},$$ 
      or 
   $$n\sum_{i=1}^k a_i'b_{i,j}+1>(n+1)\sum_{i=1}^k a_ib_{i,j}\ge n\sum_{i=1}^k a_i'b_{i,j}.$$
   The above inequality holds since
   $$n\sum_{i=1}^k(a_i'-a_i)b_{i,j}\le nk\cdot \frac{\epsilon_0}{n_0}=n_1k\epsilon_0\le \sum_{i=1}^k a_ib_{i,j},$$
   and
   $$n\sum_{i=1}^k(a_i-a_i')b_{i,j}\le nk\cdot\frac{\epsilon_0}{n_0}=n_1k\epsilon_0< 1-\sum_{i=1}^k a_ib_{i,j}.$$

	\end{proof}


Example \ref{ex: ex acc no complement} provides a counterexample to Theorem \ref{thm: dcc existence n complement with part two} if we replace the assumption ``$\Ii$ is a DCC set'' with ``$\Ii\backslash\{0\}$ is a set bounded away from $0$''. 

\begin{ex}\label{ex: ex acc no complement}
	Let $p_1,p_2,p_3,p_4$ be four different closed points on $\mathbb P^1$. For any integer $m\geq 2$, we consider the lc pairs 
	$$(X,B_m):=\left(\mathbb P^1, (\frac{1}{2}-\frac{1}{m})(p_1+p_2)+(\frac{1}{2}+\frac{1}{m})(p_3+p_4)\right).$$
	
	If $n=2k+1\ge1$ is odd, suppose that $(X,B_{2k+2})$ has an $n$-complement $(X,B_{2k+2}^{+})$, then 
	$$B_{2k+2}^{+}\ge \frac{k}{2k+1}(p_1+p_2)+\frac{k+2}{2k+1}(p_3+p_4),$$
	hence $\deg B_{2k+2}^{+}>2$, a contradiction.
	
	If $n=2k\ge 2$ is even, suppose that $(X,B_{2k})$ has an $n$-complement $(X,B_{2k}^{+})$, then 
	$$B_{2k}^{+}\ge \frac{k-1}{2k}(p_1+p_2)+\frac{k+2}{2k}(p_3+p_4),$$
	hence $\deg B_{2k}^{+}>2$, a contradiction.
	
	In conclusion, Theorem \ref{thm: dcc existence n complement} does not hold for the set $$\Ii:=\{\frac{1}{2}-\frac{1}{m},\frac{1}{2}+\frac{1}{m}\,|\,2\leq m,m\in\mathbb N^+\}.$$
\end{ex}


Example-Proposition \ref{prop: optimum diophantine approximation}, inspired by \cite{BS03}, shows that the ``approximation" part of Theorem \ref{thm: dcc existence n complement with part two} cannot be improved when $s\ge 2$. We remark that it is not optimal when $s=1$.

\begin{exprop}\label{prop: optimum diophantine approximation}
	Let $||.||_{\infty}$ be the maximal norm on $\mathbb R^2$, $\bm{e}:=(1,1)\in\mathbb R^2$, and $g: \mathbb R_{\ge0}\rightarrow\mathbb R_{\ge0}$ an increasing function, such that $\lim_{x\rightarrow+\infty}g(x)=+\infty$. Let $m$ be a positive integer such that $g(x)\geq 2$ for any $x\ge m$. Then there exists a point $\bm{v}=(r_1,r_2)\in\mathbb R^2$ satisfying the following. 
	
	For any positive integer $n\ge m$ and any point $\bm{v}_n\in\mathbb R^2$ such that $n\bm{v}_n\in\mathbb Z^2$, we have
	\begin{enumerate}
		\item $1,r_1,r_2$ are linearly independent over $\mathbb Q$. In particular, the rational envelope of $\bm{v}$ is $\mathbb R^2$, and
		\item either $$||\bm{v}-\bm{v}_n||_{\infty}>\frac{1}{ng(n)},$$ or
		$$||\frac{\bm{v}-\bm{v}_n}{||\bm{v}-\bm{v}_n|||_{\infty}}-\bm{e}||_{\infty}>1.$$
	\end{enumerate}
\end{exprop}

\begin{proof}
	Since $g(x)$ is an increasing function such that $\lim_{x\rightarrow+\infty}g(x)=+\infty$, we may pick a sequence of positive integers $\{d_i\}_{i=0}^{\infty}$ satisfying the following:
$$\text{1) } d_0:=3, \text{ 2) } d_i\mid d_{i+1}, \text{ 3) } g(d_{i+1})>2d_i, \text{ and 4) } d_{i+1}>2d_i^2.$$

	We show that we may pick $\bm{v}=(r_1,r_2)$, where
	$$r_1:=1-\sum_{k=0}^{+\infty}\frac{1}{d_{4k}}\text{ and }r_2:=1-\sum_{k=0}^{+\infty}\frac{1}{d_{4k+2}}.$$
	
	According to 1) and 4), $r_1,r_2$ are well-defined, and $r_1,r_2>0$.
	
	We first show (1). Suppose that there exist integers $a,b,c$ such that $|a|+|b|+|c|>0$ and $a+br_1+cr_2=0$. Then $|b|+|c|>0$. Pick $l>0$ such that $d_{4l}>|b|+|c|$. Consider the rational approximation
	$$r_{1,l}:=1-\sum_{k=0}^l\frac{1}{d_{4k}} \text{ and } r_{2,l}:=1-\sum_{k=0}^l\frac{1}{d_{4k+2}}.$$
	
	If $|c|>0$, then on the one hand we have
	\begin{align*}
	|a+br_{1,l}+cr_{2,l}|=&|b(r_{1,l}-r_{1})+c(r_{2,l}-r_{2})|\le\frac{2(|b|+|c|)}{d_{4l+4}}\\
	<&\frac{|b|+|c|}{d_{4l+3}^2}<\frac{|b|+|c|}{4d_{4l+2}^2}<\frac{1}{4d_{4l+2}}.
	\end{align*}
	On the other hand, $d_{4l+2}(a+br_{1,l}+cr_{2,l})$ is an integer, and $d_{4l+2}(a+br_{1,l}+cr_{2,l})\neq 0$ since $d_{4l+2}(a+br_{1,l}+cr_{2,l})\equiv -c\not=0\,(\text{mod}\ \frac{d_{4l+2}}{d_{4l}})$, a contradiction.
	
	If $|c|=0$ and $|b|>0$, then on the one hand we have
	$$|a+br_{1,l}+cr_{2,l}|=|b(r_{1,l}-r_{1})|\le\frac{2|b|}{d_{4l+4}}<\frac{|b|}{d_{4l+3}^2}<\frac{1}{4d_{4l}}.$$
	On the other hand, $d_{4l}(a+br_{1,l}+cr_{2,l})$ is an integer, and $d_{4l}(a+br_{1,l}+cr_{2,l})\neq 0$ since $d_{4l}(a+br_{1,l}+cr_{2,l})\equiv -b\not=0\,(\text{mod}\ \frac{d_{4l}}{d_{4l-4}})$, a contradiction. Hence we prove (1).
	
	\medskip
	
	We now prove (2). We may assume that $$||\frac{\bm{v}-\bm{v}_n}{||\bm{v}-\bm{v}_n||_{\infty}}-\bm{e}||_{\infty}\leq 1.$$ 
	Since $\bm{e}=(1,1)$, all the coordinates of $\bm{v}-\bm{v}_n$ are non-negative. In particular, for any $\bm{v}_n:=(v_{n,1},v_{n,2})$ such that $n\bm{v}_n\in\Zz^2$, we have
	\begin{align*}
	||\bm{v}-\bm{v}_n||=&\max\{r_1-v_{n,1},r_2-v_{n,2}\}=\frac{1}{n}\max\{nr_1-nv_{n,1},nr_2-nv_{n,2}\}\\
	\geq&\frac{1}{n}\max\{\{nr_1\},\{nr_2\}\}.
	\end{align*}
	It suffices to show that for any positive integer $n\ge m$, we have $$\max\{\{nr_1\},\{nr_2\}\}>\frac{1}{g(n)}.$$
	
	For any $n\ge m$, there are two possibilities: either $d_{4k+1}\le n\le d_{4k+3}$ or $d_{4k-1}\le n\le d_{4k+1}$ for some $k\ge0$ (here we set $d_{-1}:=1$).
	
	First suppose that $d_{4k+1}\le n\le d_{4k+3}$. We will show that $\{nr_1\}>\frac{1}{g(n)}.$ If $d_{4k}\mid n$, then 
	$$\{nr_1\}=1-\sum_{i=k+1}^{+\infty}\frac{n}{d_{4i}}>1-\frac{1}{2d_{4k+2}}\geq\frac{1}{2}\geq\frac{1}{g(n)}.$$
	
	If $d_{4k}\nmid n$, then
	\begin{align*}
	\{nr_1\}=&\{(n-\sum_{i=0}^{k}\frac{n}{d_{4i}})-\sum_{i=k+1}^{+\infty}\frac{n}{d_{4i}}\}\\
	\geq&\frac{1}{d_{4k}}-\sum_{i=k+1}^{+\infty}\frac{n}{d_{4i}}>\frac{1}{2d_{4k}}>\frac{1}{g(d_{4k+1})}\ge\frac{1}{g(n)}.
	\end{align*}
	
	Now we suppose that $d_{4k-1}\le n\le d_{4k+1}$. We will show that $\{nr_2\}>\frac{1}{g(n)}.$ If either $d_{4k-2}\mid n$ or $k=0$, then
	$$\{nr_2\}=1-\sum_{i=k}^{+\infty}\frac{n}{d_{4i+2}}>1-\frac{1}{2d_{4k}}\geq\frac{1}{2}\geq\frac{1}{g(n)}.$$
	
	If $k\geq 1$ and $d_{4k-2}\nmid n$, then
	\begin{align*}
	\{nr_2\}=&\{(n-\sum_{i=0}^{k-1}\frac{n}{d_{4i+2}})-\sum_{i=k}^{+\infty}\frac{n}{d_{4i+2}}\}\\
	\geq&\frac{1}{d_{4k-2}}-\sum_{i=k}^{+\infty}\frac{n}{d_{4i+2}}>\frac{1}{2d_{4k-2}}>\frac{1}{g(d_{4k-1})}\ge\frac{1}{g(n)}.
	\end{align*}
	
	In conclusion, $\max\{\{nr_1\},\{nr_2\}\}>\frac{1}{g(n)}$, hence we prove (2).
\end{proof}

\section{Proof of the main theorems} 

\subsection{Proof of Theorem \ref{thm: acc ld dcc coefficient bdd}}
The main goal of this subsection is to prove Theorem \ref{thm: acc ld dcc coefficient bdd}. First we need to prove a lemma which is similar to \cite[Proposition 6.2]{PS01} and \cite[Proposition 6.7]{Bir19}. For readers' convenience, we give a full proof here. 

\begin{lem}\label{lem:extendcomplement1}
	Let $n$ be a positive integer, $X\to Z$ a contraction, $(X,B)$ a pair, and $z\in Z$ a closed point. Over a neighborhood of $z$, assume that  
	\begin{enumerate}
		\item $(X,B)$ is lc,
		\item $-(K_X+B)$ is nef over $Z$,
		\item $K_S+B_S:=(K_X+B)|_{S}$,
		\item $nB,nB_{S}$ are Weil divisors,
		\item $(X,C)$ is plt for some $C$,
		\item $-(K_X+C)$ is big and nef over $Z$,
		\item $S=\lfloor C\rfloor\subset\lfloor B\rfloor$ is a prime divisor,
		\item $S$ intersects the fiber of $X\to Z$ over $z$, and 		
		\item there exists a monotonic $n$-complement $(S,B_S^+)$ of $(S,B_S)$.
	\end{enumerate}
	
	Then there exists a monotonic $n$-complement $(X/Z\ni z,B^+)$ of $(X/Z\ni z,B)$ such that $(K_X+B^{+})|_{S}=K_S+B_S^{+}$. 
\end{lem}
\begin{proof}
	Let $g:W\to X$ be a log resolution of $(X,B+C)$. We define
	$$-N_W:=K_W+B_W:=g^{*}(K_X+B),$$
	and
	$$ K_W+C_W:=g^*(K_X+C).$$
	
	Let $S_W$ be the strict transform of $S$ on $W$. Since $S$ is normal, the induced birational morphism $g|_{S_W}: S_W\to S$ is a birational contraction. Therefore,
	$$-N_W|_{S_W}=(K_W+B_W)|_{S_W}=g|_{S_W}^{*}(K_S+B_S).$$
	By assumption, there exists $R_S\ge0$, such that $(S,B_S+R_S)$ is an $n$-complement of $(S,B_S)$. Let $R_{S_W}$ be the pullback of $R_S$ on $S_W$, we have
	$$nN_W|_{S_W}\sim nR_{S_W}\ge0.$$
	In the following, we want to lift $R_{S_W}$ from $S_W$ to $W$.
	
	Let $T_W:=\lfloor B_W^{\ge0}\rfloor$, $\Delta_W:=B_W-T_W$, and
	$$L_W:=-nK_W-nT_W-\lfloor(n+1)\Delta_W\rfloor.$$ 
	We may replace $C$ with $(1-a)C+aB$ for some $a\in(0,1)$ sufficiently close to 1, and assume that  
	$$||C_W-S_W-\Delta_W+\{(n+1)\Delta_W\}||<1.$$
	
	We claim that there exists a divisor $P_W$ on $W$, such that $(W,\Lambda_W)$ is plt and $\lfloor\Lambda_W\rfloor=S_W$, where
	$$\Lambda_W:=C_W+n\Delta_W-\lfloor(n+1)\Delta_W\rfloor+P_W.$$
	More precisely, we let $\mult_{S_W} P_W:=0$ for each prime divisor $D_W\neq S_W$, and let 
	\begin{align*}
	\mult_{D_W}P_W:&=-\mult_{D_W}\lfloor C_W+n\Delta_W-\lfloor(n+1)\Delta_W\rfloor\rfloor\\
	&=-\mult_{D_W}\lfloor C_W-\Delta_W+\{(n+1)\Delta_W\}\rfloor.
	\end{align*}
	
	This implies that $0\le \mult_{D_W}P_{W}\le 1$. Moreover, $P_W$ is exceptional over $X$. Indeed, suppose that $D_W$ is an irreducible component of $P_W$ which is not exceptional over $X$. Then $D_W\neq S_W$. Since $nB$ is integral, $\mult_{D_W}(n\Delta_W-\lfloor(n+1)\Delta_W\rfloor)=0$, and $\mult_{D_W}P_W=-\mult_{D_W}\lfloor C_W\rfloor=0$, a contradiction.
	
	Since $||n\Delta_{S_W}-\lfloor(n+1)\Delta_{S_W}\rfloor||<1$, we have
	\begin{align*}
	(L_W+P_W)|_{S_W}&=(-nK_W-nT_W-\lfloor(n+1)\Delta_W\rfloor+P_W)|_{S_W}\\
	&=(nN_W+n\Delta_W-\lfloor(n+1)\Delta_W\rfloor+P_W)|_{S_W}\\
	\sim G_{S_W}:&=nR_{S_W}+n\Delta_{S_W}-\lfloor(n+1)\Delta_{S_W}\rfloor+P_{S_W}\ge0,
	\end{align*}
	where $\Delta_{S_W}:=\Delta_W|_{S_W}$ and $P_{S_W}:=P_W|_{S_W}$.
	
	We have 
	\begin{align*}
	L_W+P_W-S_W&=nN_W+n\Delta_W-\lfloor(n+1)\Delta_W\rfloor+P_W-S_W\\
	&=nN_W+\Lambda_W-C_W-S_W\\
	&=K_W+(\Lambda_W-S_W)-(K_W+C_W)+nN_W.
	\end{align*}
	Since $(W,\Lambda_W-S_W)$ is klt and $-(K_W+C_W)+nN_W$ is big and nef over $Z$, $R^1h_{*}(\mathcal{O}_{W}(L_W+P_W-S_W))=0$ by relative Kawanata-Viehweg vanishing theorem \cite[Theorem 1-2-5]{KMM87}, where $h$ is the induced morphism $W\to Z$. From the exact sequence
	$$0\to\mathcal{O}_W(L_W+P_W-S_W)\to \mathcal{O}_W(L_W+P_W)\to \mathcal{O}_{S_W}((L_W+P_W)|_{S_W})\to 0,$$
	we deduce that
	$$H^0(W,L_W+P_W)\rightarrow H^0(S_W,(L_W+P_W)|_{S_W})$$
	is surjective, and there exists $G_W\ge0$ on $W$ such that $G_W|_{S_W}=G_{S_W}$.
	
	Let $G,L,P,\Delta$ be the strict transforms of $G_W,L_W,P_W,\Delta_W$ on $X$ respectively. Since $P_W$ is exceptional over $X$, $P=0$. Since $nB$ is integral, we have $n\Delta=\lfloor(n+1)\Delta\rfloor$ and
	\begin{align*}
	-n(K_X+B)&=-nK_{X}-nT-n\Delta\\
	&=-nK_X-nT-\lfloor(n+1)\Delta\rfloor\\
	&=L=L+P\sim G\ge0.
	\end{align*}
	Let $nR:=G,B^{+}:=B+G$, then $n(K_X+B^{+})\sim 0$.
	
	It suffices to show that $(X,B^{+})$ is lc. First we construct the desired $R_W$ on $W$, such that $R_W|_{S_W}=R_{S_W}$. Let
	\begin{align*}
	nR_{W}:&=G_{W}-P_{W}-n\Delta_W+\lfloor(n+1)\Delta_W\rfloor\\
	&\sim L_W-n\Delta_W+\lfloor(n+1)\Delta_W\rfloor\\
	&=nN_W\sim_{\Qq,X}0,
	\end{align*}
	we get $g_{*}(nR_W)=G=nR$, and $R_W$ is the pullback of $R$ on $W$. Now
	\begin{align*}
	nR_{S_W}&=G_{S_W}-P_{S_W}-n\Delta_{S_W}+\lfloor(n+1)\Delta_{S_W}\rfloor\\
	&=(G_{W}-P_{W}-n\Delta_W+\lfloor(n+1)\Delta_W\rfloor)|_{S_W}=nR_{W}|_{S_W},
	\end{align*}
	which means $R_{S_W}=R_{W}|_{S_W}$ as required. In particular, $R|_{S}=R_S$, and 
	$$n(K_S+B_S^{+})=n(K_X+B^{+})|_{S}.$$ 
	By inversion of adjunction, $(X,B^{+})$ is lc near $S$. Suppose that $(X,B^{+})$ is not lc, then 
	there exists a real number $a\in(0,1)$ which is sufficiently close to 1, such that $(X,aB^{+}+(1-a)C)$ is not lc and is lc near $S$. In particular, the lc locus of $(X,aB^{+}+(1-a)C)$ is not connected. Since
	$$-(K_X+aB^{+}+(1-a)C)=-a(K_X+B^{+})-(1-a)(K_X+C)$$
	is big and nef, we get a contradiction according to the Shokurov--Koll\'{a}r connectedness principle. 
\end{proof}

\begin{lem}\label{lem:extendcomplement}
	Let $I$ be a positive integer, $(X,C)$ a plt pair, $(X,B)$ an lc pair, and $S:=\lfloor C\rfloor\subset \lfloor B\rfloor$ a prime divisor. Let $K_S+B_S:=(K_X+B)|_{S}$. Suppose that
	\begin{enumerate} 
		\item $IB,IB_S$ are Weil divisors, and
		\item $I(K_{S}+B_S)$ is Cartier,
	\end{enumerate}
	then $I(K_X+B)$ is Cartier near $S$.
	
\end{lem}
\begin{proof}
	Apply Lemma \ref{lem:extendcomplement1} to the identity map $X\to X$ and any closed point $s\in S$, we deduce that $I(K_X+B)$ is Cartier near $s$ for any closed point $s\in S$. 
\end{proof}
By Lemma \ref{lem:extendcomplement}, we can bound the Cartier indices of Weil divisors on $\mathbb Q$-factorial weak $\epsilon$-plt blow ups.
\begin{prop}\label{prop: bdd index near E}
	Let $d$ be a positive integer, $\epsilon$ a positive real number, and $\Ii\subset [0,1]$ a finite set of rational numbers. Then there exists a positive integer $I$ depending only on $d,\epsilon$ and $\Ii$, such that for any normal variety $X$ of dimension $d$ and closed point $x\in X$, if
	\begin{enumerate}
		\item there exists a klt germ $(X\ni x,\Delta)$ and a $\mathbb Q$-factorial weak $\epsilon$-plt blow-up $f: Y\rightarrow X$ of $(X\ni x,\Delta)$ with the reduced component $E$, and
		\item there exists an $\mathbb R$-divisor $B_Y\in\Ii$ on $Y$, such that $(Y/X\ni x,B_Y+E)$ is $\Rr$-complementary, 
	\end{enumerate}
	then $I(K_Y+B_Y+E)$ is Cartier near $E$. Moreover, the Cartier index of any component of $B_Y$ near $E$ is not larger than $I$. 
\end{prop}

\begin{proof}
	Write $B_Y=\sum b_iB_i$, where $B_i$ are the irreducible components of $B$. By Lemma \ref{lem: bdd lemma}, $(E,B_{E}:={\Diff}_{E}(B_Y))$ and $(E,B_{E,i}:={\Diff}_{E}(b_iB_i))$ are log bounded, and all the coefficients of $B_E$ and $B_{E,i}$ belong to a finite set of rational numbers. By \cite[Lemma 2.24]{Bir19}, there exists a positive integer $I$ depending only on $d,\epsilon$ and $\Ii$ such that both $I(K_E+B_{E,i})$ and $I(K_E+\Diff_{E}(0))$ are Cartier. By Lemma \ref{lem:extendcomplement}, both $I(K_Y+b_iB_i+E)$ and $I(K_Y+E)$ are Cartier near $E$. Possibly replacing $I$ with a bounded multiple, we may assume that  $IB_i$ and $I(K_Y+B_Y+E)$ are Cartier near $E$.
\end{proof}
\begin{prop}\label{prop: bdd index of E qfactorial}
	Let $d$ be a positive integer and $\epsilon_0,\epsilon$ two positive real numbers. Then there exists a positive integer $I$ depending only on $d,\epsilon_0$ and $\epsilon$, such that for any $\epsilon_0$-lc germ $(X\ni x,B)$ of dimension $d$, if there exists either an $\epsilon$-plt blow-up or a $\mathbb Q$-factorial weak $\epsilon$-plt blow-up $f: Y\rightarrow X$ of $(X\ni x,B)$ with the reduced component $E$, then $IE$ is Cartier.
\end{prop}

\begin{proof}
	By Proposition \ref{prop:CartierindexofEE}, there exists a positive integer $I$ depending only on $d,\epsilon_0$ and $\epsilon$, such that $IE|_{E}$ is Cartier. By Theorem \ref{thm:pltind}, $IE$ is Cartier.
\end{proof}

\begin{defn}
    Let $(X,B)$ be a pair and $A$ an $\Rr$-Cartier $\Rr$-divisor on $X$. The $\Rr$-linear system of $A$ is defined by $|A|_{\Rr}=\{L \mid A\sim_{\Rr} L\ge0\}$.
    We define the lc threshold of $|A|_{\Rr}$ with respect to $(X, B)$ as
    $$\lct(X,B;|A|_{\Rr}):= \inf\{\lct(X, B, L)\mid L\in |A|_{\Rr}\}.$$
\end{defn}

The following theorem is proved by Birkar. It would be useful in the proof of Theorem \ref{thm: log discrepancies less than 1}.

\begin{thm}[{\cite[Theorem 1.6]{Bir16}}]\label{thm: lct system lower bound 1.6}
	Let $d,r$ be two positive integers and $\epsilon$ a positive real number. Then there exists a positive real number $t$ depending only on $d,r$ and $\epsilon$ satisfying the following. Assume that 
	\begin{enumerate}
		\item $(X,B)$ is a projective $\epsilon$-lc pair of dimension $d$,
		\item $A$ is a very ample divisor on $X$ such that $A^d\leq r$,
		\item $A-B$ is ample, and 
		\item $M\geq 0$ is an $\mathbb R$-Cartier $\mathbb R$-divisor on $X$ with $|A-M|_{\mathbb R}\not=\emptyset$.
	\end{enumerate}
	
	Then 
	$$\lct(X,B;|M|_{\mathbb R})\geq \lct(X,B;|A|_{\mathbb R})>t.$$
\end{thm}

\begin{prop}\label{prop: local lct linear system}
Let $d$ be a positive integer and $\epsilon,\delta$ two positive real numbers. Then there exists a positive real number $t$ depending only on $d,\epsilon$ and $\delta$ satisfying the following. Assume that
\begin{enumerate}
    \item  $(X\ni x,B)$ is a klt germ of dimension $d$,
        \item there exists a  $\mathbb Q$-factorial weak $\epsilon$-plt blow up $f: Y\rightarrow X$ of $(X\ni x,B)$ with the reduced component $E$, and
    \item all the (non-zero) coefficients of $B$ are $\geq\delta$,
\end{enumerate}
then $(Y,(1+t)B_Y+E)$ is plt near $E$, where $B_Y$ is the strict transform of $B$ on $Y$. 
\end{prop}

\begin{proof}
     Let 
	$$K_E+B_E:=(K_Y+B_Y+E)|_{E}.$$
	By Lemma \ref{lem: bdd lemma}, $(E,B_{E})$ is log bounded. Thus there exists a positive real number $r$ depending only on $d,\epsilon$ and $\delta$, and a very ample Cartier divisor $A$ on $E$, such that $A^{d-1}\leq r$ and $A-B_{E}$ is ample. Since $B_E\geq B_Y|_{E}$, by Theorem \ref{thm: lct system lower bound 1.6}, there exists a positive real number $t$ depending only on $d,\epsilon$ and $\delta$, such that $\lct(E,B_{E};B_{Y}|_{E})>t$. By Theorem \ref{thm: adjunction}, $(Y,(1+t)B_{Y}+E)$ is plt near $E$.
\end{proof}

Now we can prove the main result of this section.

\begin{thm}\label{thm: log discrepancies less than 1}
	Let $d$ be a positive integer, $\epsilon_0,\epsilon,M$ three positive real numbers, and $\Ii\subset [0,1]$ a DCC set. Then there exists an ACC set $\Ii'$ depending only on $d,\epsilon_0,\epsilon,M$ and $\Ii$, and an ACC set $\Ii''$ depending only on $d,\epsilon_0,\epsilon$ and $\Ii$ satisfying the following.	Assume that
	\begin{itemize}
		\item $(X\ni x,B)$ is an $\epsilon_0$-lc germ of dimension $d$,
		\item $B\in \Ii$, 
		\item there exists a $\mathbb Q$-factorial weak $\epsilon$-plt blow-up $f:Y\rightarrow X$ of $(X\ni x,B)$, and
		\item $F$ is a prime divisor over $X\ni x$ such that $a(F,X,B)\leq M$,
		\end{itemize}
		then
		\begin{enumerate}
		\item $a(F,X,B)\in\Ii'$, and
		\item $\mld(X\ni x,B)\in\Ii''$.
		\end{enumerate}
\end{thm}
\begin{proof} 
	First we show (1). Let $E$ be the reduced component of $f$ and $B_Y$ the strict transform of $B$ on $Y$. For simplicity, we let $e:=a(E,X,B)$. By Theorem \ref{thm: existence ni1i2 complement}, there exist two positive integers $k,n$, a finite set $\Ii_0=\{a_1,\dots,a_k\}\subset (0,1]$ depending only $d$ and $\Ii$, and $\Qq$-divisors $B_1,\dots,B_k\ge0$ on $Y$, such that
	\begin{itemize}
	\item $\sum_{i=1}^k a_i=1$ and $B_{Y}'+E:=\sum_{i=1}^k a_i(B_{i}+E)\geq B_Y+E$, and
	\item each $(Y/X\ni x,B_i+E)$ is a monotonic $n$-complement of itself.
	\end{itemize}
	By Proposition \ref{prop: bdd index near E} and Proposition \ref{prop: bdd index of E qfactorial}, there exists a positive integer $I$ depending only on $d,\epsilon_0,\epsilon$ and $\Ii$, such that 
	\begin{itemize}
	\item $IE$ is Cartier, and
	\item $ID$ is Cartier near $E$ for every irreducible component $D$ of $B_Y'$.
	\end{itemize}
We have
	\begin{align*}
	M\geq a(F,X,B)&=a(F,Y,B_Y+(1-e)E)=e\mult_{F}E+a(F,Y,B_Y+E)\\
	&=e\mult_{F}E+a(F,Y,B_Y'+E)+\mult_F(B_Y'-B_Y)\\
	&=e\mult_{F}E+\sum_{i=1}^k a_ia(F,Y,B_i+E)+\mult_F(B_Y'-B_Y).
	\end{align*}
	
	It suffices to show that $e\mult_{F}E$, $\sum_{i=1}^k a_ia(F,Y,B_i+E)$, and $\mult_F(B_Y'-B_Y)$ belong to an ACC set.
	
	 Since $IE$ is Cartier and $\mult_FE\leq\frac{M}{e}\leq\frac{M}{\epsilon_0}$, $\mult_FE$ belongs to a finite set. By Proposition \ref{prop: acc ld kc}, $e$ belongs to an ACC set, thus $e\mult_{F}E$ belongs to an ACC set. 
	 
	 Since each $n(K_Y+B_{i}+E)$ is Cartier near $E$, $\sum_{i=1}^k a_ia(F,Y,B_i+E)$ belongs to a finite set. 
	 
	Finally, we show that $\mult_F(B_Y'-B_Y)$ belong to an ACC set. Since the coefficients of $B'_Y-B_Y$ belong to an ACC set of non-negative real numbers and $ID$ is Cartier near $E$ for every irreducible component $D$ of $B'_Y$, it is enough to show that $\mult_FD$ is bounded from above.

	By Proposition \ref{prop: local lct linear system}, there exists a positive real number $t$ depending only $d,\epsilon$ and $\Ii$, such that $(Y,(1+t)B_{Y}+E)$ is plt near $E$. We have
	\begin{align*}
	0&\leq a(F,Y,(1+t)B_{Y}+E)=a(F,Y,B_{Y}+E)-t\mult_FB_Y\\
	&\leq a(F,X,B)-t\mult_FB_Y\leq M-t\mult_FB_Y,
	\end{align*}
	which implies that $\mult_{F}B_{Y}\leq\frac{M}{t}$. Since $\Ii$ is a DCC set, for any irreducible component $D$ of $B_Y$, $\mult_FD$ is bounded from above.

	For any irreducible component $D$ of $B_Y'$ such that $D\not \subseteq \Supp B_Y$, we have
		\begin{align*}
	    M&\geq\mult_F(B_Y'-B_Y)\geq\mult_F\big((\mult_D(B_Y'-B_Y))D\big)\\
	    &=\mult_F\big((\mult_D(B_Y'))D\big)\geq\frac{1}{n}\min_{1\le i\le k}\{a_i\}\mult_F D.
	\end{align*}
	Thus $\mult_FD\leq\frac{Mn}{\min_{1\le i\le k}\{a_i\}}$. 
	
	Hence $\mult_FD$ is bounded from above for every irreducible component $D$ of $B_Y'$, and we prove (1).

	Since $\mld(X\ni x,B)\leq a(E,X,B)$ and $a(E,X,B)$ belongs to an ACC set by Proposition \ref{prop: acc ld kc}, (2) follows immediately from (1).\end{proof}

\begin{proof}[Proof of Theorem \ref{thm: acc ld dcc coefficient bdd}] 
	Since $(X\ni x,B)$ admits an $\epsilon$-plt blow-up, according to Lemma \ref{lem: existence qfact weak plt blow up}, there exists a $\mathbb Q$-factorial weak $\epsilon$-plt blow-up of $(X\ni x,B)$. The theorem follows from Theorem \ref{thm: log discrepancies less than 1}.
\end{proof}

\subsection{Proofs of Theorem \ref{thm: ACC mld exc} and Theorem \ref{thm: acc mld dcc e0lc eplt}} 

\begin{thm}\label{thm: discrete ld plt}
	Let $d$ be a positive integer, $\epsilon_0,\epsilon$ two positive real numbers, and $\Ii\subset [0,1]$ a finite set. Assume that $(X\ni x,B)$ is an lc germ of dimension $d$, such that 
		\begin{enumerate}
		\item $B\in\Ii$,
		\item  $(X\ni x,\Delta)$ is an $\epsilon_0$-lc germ for some boundary $\Delta$, and
		\item $(X\ni x,\Delta)$ admits an $\epsilon$-plt blow-up.
	\end{enumerate}
	Then for any prime divisor $F$ over $X\ni x$,	$a(F,X,B)$
	belongs to a discrete set depending only on $d,\epsilon_0,\epsilon$ and $\Ii$. Moreover, $\mld(X\ni x,B)$ belongs to a finite set depending only on $d,\epsilon_0,\epsilon$ and $\Ii$.
\end{thm}
\begin{proof} By Theorem \ref{thm: Uniform perturbation of lc pairs}, there exist a finite set of positive real numbers $\Ii_0=\{a_1,\dots,a_k\}\subset (0,1]$ and a finite set of rational numbers $\Ii_1$ depending only on $d$ and $\Ii$, and $\Qq$-divisors $B_1,\dots,B_k\geq 0$ on $X$, such that 
\begin{itemize}
    \item $(X\ni x_i,B_i)$ is lc for every $i$,
    \item $B_i\in\Ii_1$ for every $i$, and
    \item $\sum_{i=1}^ka_i=1$ and $B=\sum_{i=1}^ka_iB_i$.
\end{itemize}
 Then for any prime divisor $F$ over $X\ni x$, we have
	$$a(F,X,B)=\sum_{i=1}^k a_ia(F,X,B_i).$$
	Since $(X\ni x,\Delta)$ is an $\epsilon_0$-lc germ and admits an $\epsilon$-plt blow-up, by Theorem \ref{thm: bdd local index exceptional}, there exists a positive integer $I$ depending only on $d,\epsilon_0$ and $\epsilon$ such that $I(K_X+B_i)$ is Cartier near $x$ for every $i$. Therefore, $a(F,X,B_i)$ belong to the discrete set $\frac{1}{I}\mathbb N$. Thus $a(F,X,B)$ belongs to a discrete set.
	
	In particular, $\mld(X\ni x,B)$ belongs to a discrete set. Let $E$ be the reduced component of an $\epsilon$-plt blow-up of $(X\ni x,\Delta)$. By Proposition \ref{prop: acc ld kc},  $a(E,X,B)$ is bounded from above. Since $\mld(X\ni x,B)\leq a(E,X,B)$, $\mld(X\ni x,B)$ is bounded from above. Thus $\mld(X\ni x,B)$ belongs to a finite set.
\end{proof}

\begin{cor}\label{cor: finite ld fixed germ}
	Fix a normal variety $X$, a closed point $x\in X$, and a finite set $\Ii\subset [0,1]$. Suppose that $(X\ni x,\Delta)$ is a klt germ. Then
	$$\{a(E,X,B)\mid \Center_XE=x, (X\ni x,B)\ \textit{is lc}, B\in\Ii\}$$
	is a discrete set. In particular,
	$$\{\mld(X\ni x,B)\mid (X\ni x,B)\ \textit{is lc}, B\in\Ii\}$$
	is a finite set.
\end{cor}

\begin{proof}
This follows from Lemma \ref{lem: existence plt blow up} and Theorem \ref{thm: discrete ld plt}. 
\end{proof}

\begin{rem}
 Corollary \ref{cor: finite ld fixed germ} can be viewed as a generalization of \cite[Theorem 1.1, Theorem 1.2]{Kaw11} as we do not need to assume that $B\geq\Delta$.
\end{rem}

\begin{proof}[Proof of Theorem \ref{thm: acc mld dcc e0lc eplt}] 

Let $E$ be the reduced component of $f$, $a:=a(E,X,B)$, and $B_Y$ the strict transform of $B$ on $Y$. 

For any prime divisor $F$ over $X\ni x$ such that $F\not=E$,
$$a(F,X,B)=a(F,Y,B_Y+(1-a)E)\geq a(F,Y,B_Y+E)>\epsilon,$$
which implies (2).

Next we show (1). By Proposition \ref{prop: acc ld kc}, $a$ belongs to an ACC set (resp. an ACC set whose only possible accumulation point is $0$). Therefore, we may assume that $\mld(X\ni x,B)$ is not attained at $E$. By (2), we may assume that $a>\epsilon$. It follows that $(X\ni x,B)$ is an $\epsilon$-lc germ which admits an $\epsilon$-plt blow-up. In particular, when $\Ii$ is a finite set, (1) follows from Theorem \ref{thm: discrete ld plt}.

Since $a$ belongs to an ACC set, there exists a positive integer $M$ depending only on $d,\epsilon$ and $\Ii$, such that $\mld(X\ni x,B)\leq a(E,X,B)\leq M$. Thus 
$$\mld(X\ni x,B)\in\mathcal{LD}(d,\epsilon,\epsilon,\Ii)\cap [0,M],$$
which is an ACC set by Theorem \ref{thm: acc ld dcc coefficient bdd}, and we prove (1). \end{proof}

\begin{proof}[Proof of Theorem \ref{thm: ACC mld exc}] By Lemma \ref{lem: reduced component exc sing} and Lemma \ref{lem: exp sing epsilon plt}, we may assume that $d\geq 2$, $(X\ni x,B)$ is a klt germ, and there exists a positive real number $\epsilon$ depending only on $d$ and $\Ii$ such that $(X\ni x,B)$ admits an $\epsilon$-plt blow-up. Theorem \ref{thm: ACC mld exc} follows from Theorem \ref{thm: acc mld dcc e0lc eplt}. 
\end{proof}

\section{Applications}

\subsection{Monotonic klt local complements} By applying a similar argument as the proof of Theorem \ref{thm: log discrepancies less than 1}, we show the existence of monotonic klt $n$-complements for $\epsilon_0$-lc germs admitting an $\epsilon$-plt blow-up.

\begin{thm}[Existence of monotonic klt local complements]\label{thm: existence local klt complement}
	Let $d$ be a positive integer, $\epsilon_0,\epsilon$ two positive real numbers, and $\Ii\subset [0,1]$ a DCC set.
	There exists a positive integer $n$ depending only on $d,\epsilon_0,\epsilon$ and $\Ii$ satisfying the following.
	
	Let $(X\ni x,B)$ be an $\epsilon_0$-lc germ of dimension $d$, such that
	\begin{enumerate}
		\item $B\in\Ii$, 
		\item $(X\ni x,B)$ admits an $\epsilon$-plt blow-up, and
		\item all the irreducible components of $B$ are $\mathbb Q$-Cartier near $x$.
	\end{enumerate}
	Then there exists a monotonic klt $n$-complement $(X\ni x,B^+)$ of $(X\ni x,B)$.
\end{thm}

\begin{proof} By Lemma \ref{lem: existence qfact weak plt blow up}, there exists a $\mathbb Q$-factorial weak $\epsilon$-plt blow-up $f:Y\rightarrow X$ of $(X\ni x,B)$ with the reduced component $E$.

Let $B_Y$ be the strict transform of $B$ on $Y$. By Lemma \ref{prop: local lct linear system}, there exists a positive real number $t\le 1$ depending only one $d,\epsilon$ and $\Ii$, such that $(Y,(1+t)B_Y+E)$ is plt near $E$. Since all the irreducible components of $B$ are $\mathbb Q$-Cartier near $x$, $K_X$ is $\mathbb Q$-Cartier near $x$, and $f$ is also a $\mathbb Q$-factorial weak $\epsilon$-plt blow-up of $(X\ni x,0)$. By Proposition \ref{prop: acc ld kc}, there exists a positive real number $M\geq 1$ depending only on $d$ and $\epsilon$ such that $a(E,X,0)\leq M$. Possibly replacing $\epsilon_0$ with $\min\{\epsilon_0,1\}$, we may assume that $\epsilon_0\leq 1$. Define $s_0:=\min\{1,\gamma>0\mid \gamma\in\Ii\}$ and $\delta:=\frac{\epsilon_0s_0t}{4M}.$
By Theorem \ref{thm: dcc limit lc divisor}, there exists a finite set $\Ii'$ depending only on $d,\epsilon_0,\epsilon$ and $\Ii$, and an $\mathbb R$-divisor $B'\geq B$, such that
\begin{itemize}
\item $(X\ni x,B')$ is lc,
\item $B'\in\Ii'$,
\item $\Supp B'=\Supp B$, and
\item $\mult_D(B'-B)<\delta$ for every irreducible component $D$ of $B$.
\end{itemize}
Thus
$\frac{\delta}{s_0}B\geq\delta\Supp B\geq B'-B$, which implies that $(\frac{\delta}{s_0}+1)B\geq B'$. 
We have
\begin{align*}
	a(E,X,(1+\delta)B')=&a(E,X,B)-\mult_{E}((1+\delta)B'-B)\\
	\ge& a(E,X,B)-\mult_{E}((1+\delta)(\frac{\delta}{s_0}+1)B-B)\\
	\ge&\epsilon_0-((1+\delta)(\frac{\delta}{s_0}+1)-1)\mult_{E}B\\
	\ge&\epsilon_0-\frac{3\epsilon_0t}{4M}\mult_{E}B\\
	=&\epsilon_0-\frac{3\epsilon_0t}{4M}(a(E,X,0)-a(E,X,B))\\
	\ge&\epsilon_0-\frac{3\epsilon_0 t}{4}\geq\frac{1}{4}\epsilon_0.
	\end{align*}
	Hence $K_Y+(1+\delta)B_Y+E\geq f^{*}(K_X+(1+\delta)B')$, 
 $$a(F,X,(1+\delta)B')\ge a(F,Y,(1+\delta)B_Y+E)>0$$
 for any prime divisor $F\neq E$ over $X\ni x$, and
 $(X\ni x,(1+\delta)B')$ is klt.
 
 Since $\Ii'$ is a finite set, there exists a finite set $\Ii_0$ of rational numbers depending only on $d,\epsilon_0,\epsilon$ and $\Ii$, and a $\mathbb Q$-divisor $B^+\in\Ii_0$ on $X$, such that $(1+\delta)B'\geq B^+\geq B'$. Now the existence of $n$ follows from Theorem \ref{thm: bdd local index exceptional}.
 \end{proof}


\subsection{Miscellaneous results for exceptional singularities}

\subsubsection{ACC for $a$-lc thresholds of exceptional singularities}

$a$-lc thresholds is an important algebraic invariant. In particular, when $a=0,1$, we get the lc thresholds and the canonical thresholds respectively. The third author conjectured that the set of $a$-lc thresholds satisfies the ACC:

\begin{conj}[ACC for $a$-lc thresholds]\label{conj: ACC for aLCTs}Let $d$ be a positive integer, $a$ a non-negative real number, and $\Ii$ a DCC set. Then
	$$\{a\text{-}\lct(X\ni x,B;D)\mid \dim X=d, (X\ni x,B)\text{ is lc, }B,D\in \Ii\}$$
	is an ACC set.
\end{conj}
Hacon-M\textsuperscript{c}Kernan-Xu \cite{HMX14} proved Conjecture \ref{conj: ACC for aLCTs} for $a=0$. Birkar-Shokurov \cite{BS10} showed that Conjecture \ref{conj: ACC for aLCTs} follows from Conjecture \ref{conj: ACC for mlds}, hence Conjecture \ref{conj: ACC for aLCTs} holds for surfaces. When $\dim X\ge3$, Conjecture \ref{conj: ACC for aLCTs} is still open in general. \cite{Kaw18} showed that Conjecture \ref{conj: ACC for mlds} and Conjecture \ref{conj: ACC for aLCTs} are equivalent for any fixed germ, and \cite{Liu18} showed that Conjecture \ref{conj: ACC for mlds} and Conjecture \ref{conj: ACC for aLCTs} are equivalent for non-canonical singularities.

As an application of Theorem \ref{thm: ACC mld exc}, we can prove the ACC for $a$-lc thresholds of exceptional singularities.

\begin{thm}\label{thm: ACC aLCT exc}
	Let $d$ be a positive integer, $a$ a non-negative real number, and $\Ii\subset [0,+\infty)$ a DCC set. Then for any lc germ $(X\ni x,B)$ of dimension $d$ and any $\Rr$-Cartier $\Rr$-divisor $D$ on $X$, if 
	\begin{enumerate}
		\item $(X\ni x,B)$ is exceptional, and
		\item $B,D\in \Ii$,
	\end{enumerate}
	then 
	$$a\text{-}\lct(X\ni x,B;D)$$
	belongs to an ACC set depending only on $d,a$ and $\Ii$.
\end{thm}

\begin{proof} We follow the argument in \cite{BS10}. Suppose that the theorem does not hold. Then there exist a sequence of lc germs $(X_i\ni x_i,B_i)$ of dimension $d$, and a strictly increasing sequence of positive real numbers $t_i$, such that for every $i$, 
	\begin{itemize}
		\item $B_i\in\Ii$,
		\item  $(X_i\ni x_i,B_i)$ is exceptional, and
		\item there exists an $\Rr$-Cartier $\Rr$-divisor $D_i$ on $X_i$, such that $D_i\in\Ii$ and $t_i=a$-$\lct(X_i\ni x_i,B_i;D_i)$.
	\end{itemize}
	
	By Lemma \ref{lem: reduced component exc sing}, $(X_i\ni x_i,B_i)$ is a klt germ. Let $t_0:=\lim_{i\rightarrow+\infty}t_i$. By Theorem \ref{thm: acc lct}, possibly passing to a subsequence, we may assume that $a>0$ and $(X_i\ni x_i, B_i+t_0D_i)$ is lc for every $i$. 
	
	Let $\{\epsilon_i\}_{i=1}^{\infty}$ be a strictly decreasing sequence which converges to $0$, such that $0<\epsilon_i<1$ for any $i$. Let $a_i:=\mld(X_i\ni x_i,B_i+t_0D_i)$ and  $t_i':=t_i+\epsilon_i(t_0-t_i)$ for any $i$. Then $(X_i\ni x_i,B_i+t_i'D_i)$ is a klt germ. Since $t_i<t_i'<t_0$, possibly passing to a subsequence, we may assume that $t_i'$ is strictly increasing and all the coefficients of $B_i+t_i'D_i$ belong to a DCC set. Since $(X_i\ni x_i,B_i)$ is exceptional and $(X_i\ni x_i,B_i+t_i'D_i)$ is a klt germ, $(X_i\ni x_i,B_i+t_i'D_i)$ is exceptional. By Theorem \ref{thm: ACC mld exc}, $\mld(X_i\ni x_i,B_i+t_i'D_i)$ belongs to an ACC set. By the convexity of minimal log discrepancies, we have
	\begin{align*}
	a&> \mld(X_i\ni x_i,B_i+t_i'D_i)\\
	&=\mld(X_i\ni x_i,\frac{t_i'-t_i}{t_0-t_i}(B_i+t_0D_i)+\frac{t_0-t_i'}{t_0-t_i}(B_i+t_iD_i))\\
	&\geq\frac{t_i'-t_i}{t_0-t_i}\mld(X_i\ni x_i,B_i+t_0D_i)+\frac{t_0-t_i'}{t_0-t_i}\mld(X_i\ni x_i,B_i+t_iD_i)\\
	&=\frac{t_i'-t_i}{t_0-t_i}a_i+\frac{t_0-t_i'}{t_0-t_i}a=a-\frac{(t_i'-t_i)(a-a_i)}{t_0-t_i}\\
	&=a-\epsilon_i(a-a_i)\geq (1-\epsilon_i)a.
	\end{align*}
	Therefore, possibly passing to a subsequence, we may assume that $\mld(X_i\ni x_i,B_i+t_i'D_i)$ is strictly increasing and converges to $a$, which contradicts Theorem \ref{thm: ACC mld exc}.
\end{proof} 

\subsubsection{ACC for normalized volumes of exceptional singularities}

\cite{Li18} introduced the normalized volumes to study K-stability and the existence of K\"{a}hler-Einstein metrics on Fano varieties. We refer the readers to \cite{Tia97,LX16,Li17} for works in this direction. It is conjectured that the set of normalized volumes satisfies the ACC:
\begin{conj}[ACC for normalized volumes, {cf. \cite[Question 6.12]{LLX18}}]\label{conj: ACC for NV} Let $d\ge 2$ be an integer and $\Ii\subset [0,1]$ a DCC (resp. finite) set. Then
	$$\{\widehat{\vol}(X\ni x,B)\mid\dim X=d, (X\ni x,B) \text{ is a klt germ}, \text{ and }B\in \Ii\}$$
	is an ACC set (resp. ACC set whose only possible accumulation point is $0$).
\end{conj}
The first author, Liu, and Qi announced a proof for Conjecture \ref{conj: ACC for NV} when $X\ni x$ is a fixed germ \cite{HLQ20}.

As an application of Theorem \ref{thm: ACC mld exc}, we can prove Conjecture \ref{conj: ACC for NV} for exceptional singularities:

\begin{thm}\label{thm: ACC NV exc} Let $d\geq 2$ be an integer and $\Ii\subset [0,1]$ a DCC set. Assume that $(X\ni x,B)$ is a klt exceptional singularity of dimension $d$ such that $B\in\Ii$. Then
	\begin{enumerate}
	\item $\widehat{\vol}(X\ni x,B)$ belongs to an ACC set depending only on $d$ and $\Ii$, and
       \item 	 if $\Ii$ is a finite set, then
       \begin{enumerate}
       \item  the only possible accumulation point of $\widehat{\vol}(X\ni x,B)$ is $0$, and
       \item $\widehat{\vol}(X\ni x,B)$ is bounded away from $0$ if and only if $\mld(X\ni x,B)$ is bounded away from $0$.
       \end{enumerate}
       \end{enumerate}
\end{thm}

\begin{proof}[Proof of Theorem \ref{thm: ACC NV exc}] This follows from Lemma \ref{lem: reduced component exc sing} and Theorem \ref{thm: nv acc 2}.
\end{proof}


\subsubsection{Correspondence between exceptional singularities and exceptional pairs}
We recall the definition of exceptional pairs.
\begin{defn}\label{defn: global exceptional}
    A pair $(X,B)$ is called \emph{exceptional} if 
    \begin{itemize}
    \item $X$ is projective,
        \item $(X,B)$ is $\Rr$-complementary, and
        \item every $\Rr$-complement $(X,B^+)$ of $(X,B)$ is klt.
    \end{itemize}
    \end{defn}

We will show the following theorem.
\begin{thm}\label{thm: local global exceptional correspondence}
Let $(X\ni x,B)$ be a klt germ. For any plt blow-up $f: Y\rightarrow X$ of $(X\ni x,B)$ with the reduced component $E$, we denote $K_E+B_E:=(K_Y+f^{-1}_{*}B+E)|_{E}$. Then the following are equivalent:
\begin{enumerate}
    \item $(X\ni x, B)$ is an exceptional singularity.
        \item For every reduced component $E$ of $(X\ni x, B)$, $(E,B_E)$ is an exceptional pair.
    \item There exists a reduced component $E$ of $(X\ni x, B)$, such that $(E,B_E)$ is an exceptional pair.
\end{enumerate}
\end{thm}

To simplify the notation in the arguments below, we generalize the notation of $\Rr$-affine functional divisors, as in Definition \ref{defn: r affine functional divisor}, to multiple variables. 

\begin{defn}
Let $c$ be a positive integer, $X$ a normal variety, and $D_i$ distinct prime divisors on $X$. Suppose that $d_i(x_1,\dots,x_c):\mathbb R^{c}\rightarrow\mathbb R$ are $\Rr$-affine functions (resp. $\Qq$-affine functions). Then we call the formal sum $D(x_1,\dots,x_c):=\sum_id_i(x_1,\dots,x_c)D_i$ an $\Rr$-affine functional divisor (of multiple variables) (resp. $\Qq$-affine functional divisor (of multiple variables)). 
\end{defn}


\begin{lem}\label{lem: coeffB R-complement}
Let $(X,B)$ be a pair, $X\rightarrow Z$ a contraction, and $z\in Z$ a point, such that $(X/Z\ni z,B)$ has an $\Rr$-complement $(X/Z\ni z,B^{+}:=B+G)$. Then there exists an $\Rr$-complement $(X/Z\ni z,B')$ of $(X/Z\ni z,B)$, such that \begin{enumerate}
        \item $B'\in\Span_{\Qq}(\Coeff(B)\cup\{1\})$, and
        \item for every prime divisor $E$ over $X$, $a(E,X,B^+)=0$ if and only if $a(E,X,B')=0$.
    \end{enumerate} 
\end{lem}

 

\begin{proof}
There exist two integers $0\leq c\leq c'$, two $\Qq$-affine functional divisors $B(x_1,\dots,x_c)$ and $G(x_1,\dots,x_{c'})$, and real numbers $r_0:=1,r_1,\dots, r_{c'}$ which are linearly independent over $\Qq$, such that  $B=B(r_1,\dots,r_c)$, $G=G(r_1,\dots,r_{c'})$, and $\Span_{\Qq}(\{r_0,\ldots,r_c\})=\Span_{\Qq}(\Coeff(B)\cup\{1\})$.

   If $c=c'$, then $B':=B^+$ satisfies the required properties. Therefore, we may assume that $c'>c$. By Lemma \ref{lem: dltmodelQfactorialcoeff1}, there exists an open neighborhood $U\subset\mathbb R^{c'-c}$ of $(r_{c+1},\dots,r_{c'})$ such that for any $(v_{c+1},\dots,v_{c'})\in U$,
   \begin{itemize}
      \item $(X/Z\ni z,B+G(r_1,\dots,r_c,v_{c+1},\dots,v_{c'}))$ is an $\Rr$-complement of $(X/Z\ni z,B)$, and
   \item for every prime divisor $E$ over $X$, $a(E,X,B^{+})=0$ if and only if $a(E,X,B+G(r_1,\dots,r_c,v_{c+1},\dots,v_{c'}))=0$ .
   \end{itemize}
Then 
$B':=B+G(r_1,\dots,r_c,v^0_{c+1},\dots,v^0_{c'})$ satisfies the required properties for any $(v^0_{c+1},\dots,v^0_{c'})\in U\cap\Qq^{c'-c}$.
\end{proof}

\begin{proof}[Proof of Theorem \ref{thm: local global exceptional correspondence}]
We first show that (1) implies (2). Suppose that (2) does not hold, then there exist a klt exceptional singularity $(X\ni x,B)$ and a plt blow-up $f: Y\rightarrow X$ of $(X\ni x,B)$ with the reduced component $E$, such that $(E,B_E)$ is not an exceptional pair. By Lemma \ref{lem: coeffB R-complement}, there exists an $\Rr$-complement $(E,B_E+G_E)$ of $(E,B_E)$, such that $(E,B_E+G_E)$ is not klt, and $G_E\subset \Span_{\Qq}(\Coeff(B)\cup\{1\})$.
 
 There exist real numbers $r_0:=1,r_1,\dots,r_c$ which are linearly independent over $\Qq$, a $\Qq$-affine functional divisor $B_Y(x_1,\dots,x_c)$ on $Y$, and $\Qq$-affine functional divisors $B_E(x_1,\dots,x_c)$ and $G_E(x_1,\dots,x_c)$ on $E$, such that $B_Y=B_Y(\bm{r})$, $B_E=B_E(\bm{r})$, and $G_E=G_E(\bm{r})$, where $\bm{r}:=(r_1,\dots,r_c)$, and $B_Y$ is the strict transform of $B$ on $Y$.
 
     By adjunction formula for $\Rr$-affine functional divisors (cf. \cite[Proposition 3.14]{HLQ17}), we have 
     $$(K_Y+B_Y(v_1,\dots,v_c)+E)|_E=K_E+B_E(v_1,\dots,v_c)$$ 
     for any $(v_1,\dots,v_c)\in \Rr^c$.
     
     By Lemma \ref{lem: dltmodelQfactorialcoeff1}, there exists a neighborhood $U\subset\mathbb R^c$ of $\bm{r}$, such that for any $\bm{v}\in U$,
    \begin{itemize}
        \item $(E,B_E(\bm{v})+G_E(\bm{v}))$ is an $\Rr$-complement of itself, 
        \item $(E,B_E(\bm{v})+G_E(\bm{v}))$ is not klt, and
        \item $-(K_Y+B_Y(\bm{v})+E)$ is ample over $X$.
    \end{itemize}

    Let $\bm{v}_1:=(v_1^1,\dots,v_c^1),\dots,\bm{v}_{c+1}:=(v_{1}^{c+1},\dots,v_{c}^{c+1})\in U\cap\Qq^c$ be $c+1$ points such that $\bm{r}$ belongs to the convex hull of these points. Let $a_1,\dots,a_{c+1}\in (0,1]$ be real numbers such that $\sum_{i=1}^{c+1}a_i=1$ and $\sum_{i=1}^{c+1}a_i\bm{v}_i=\bm{r}$. Then there exists a positive integer $n$, such that for any $1\leq i\leq c+1$,
    \begin{itemize}
        \item $(E,B_E(\bm{v}_i)+G_E(\bm{v}_i))$ is an $n$-complement of itself, and
        \item $nB(\bm{v}_i)$ and $nB_S(\bm{v}_i)$ are Weil divisors.
    \end{itemize} 
    By Lemma \ref{lem:extendcomplement1}, for each $i$, there exists a monotonic $n$-complement $(Y/X\ni x,B_Y(\bm{v}_i)+E+G_i)$ of $(Y/X\ni x,B_Y(\bm{v}_i)+E)$ such that $$(K_Y+B(\bm{v}_i)+E+G_i)|_{E}=K_E+B_E(\bm{v}_i)+G_E(\bm{v}_i).$$
    Let $G_Y:=\sum_{i=1}^{c+1}a_iG_i$. Then $(Y/X\ni x,B_Y+E+G_Y)$ is an $\Rr$-complement of itself, and $(Y,B_Y+E+G_Y)$ is not plt near $E$ by Theorem \ref{thm: adjunction}. Thus $(X\ni x,B)$ is not an exceptional singularity, a contradiction.

\medskip

Since (2) implies (3), it suffices to show that (3) implies (1). Suppose that $(E,B_E)$ is an exceptional pair, and $(X\ni x,B)$ is not an exceptional singularity. Then there exist two $\mathbb R$-Cartier $\mathbb R$-divisors $G\geq 0$ and $G'\geq 0$ on $X$ satisfying the following.
    \begin{itemize}
        \item $E$ is the unique lc place of $(X\ni x,B+G)$, and
        \item there exist at least two lc places of $(X\ni x,B+G')$.
    \end{itemize}
   
    For every real number $0\leq t\leq 1$, we define
    $$h(t):=\lct(X\ni x,B+tG';G).$$
    $h(t)$ is a monotonically decreasing function of $t$, $h(0)=1$ and $h(1)=0$. Let
    $$t_0:=\sup\{t\ge 0\mid E\text{ is  the unique lc place of }(X\ni x,B+tG'+h(t)G)\}.$$
    Then
    \begin{itemize}
\item $(X\ni x,\Delta:=B+t_0G'+h(t_0)G)$ is lc,
\item $E$ is an lc place of $(X\ni x,\Delta)$, and
\item $(X\ni x,\Delta)$ has at least two different lc places.
\end{itemize} 
Let $\Delta_Y$ be the strict transform of $\Delta$ on $Y$, and $K_E+\Diff_E(\Delta_Y):=(K_Y+\Delta_Y+E)|_{E}$. It follows that $\Diff_E(\Delta_Y)\geq B_E$, $K_E+\Diff_E(\Delta_Y)\sim_{\Rr}0$, and $(E,\Diff_E(\Delta_Y))$ is lc but not klt by Theorem \ref{thm: adjunction}. Hence $(E,B_E)$ is not an exceptional pair, a contradiction.
    \end{proof}

\subsection{Complete regularities}

\subsubsection{Definitions: dual complex, regularities, and complete regularities}
\begin{defn}[Dual complex of a simple normal crossing variety, cf. {\cite{Sho00}}]\label{defn: crt}
 Let $E$ be a simple normal crossing variety with irreducible components $\{E_i\mid i\in \mathcal{I}\}$. A stratum of $E$ is any irreducible component $F\subset \cap_{i\in \mathcal{J}}E_i$ for some $\mathcal{J}\subset\mathcal{I}$. A \emph{dual complex} of $E$, denoted by $\mathcal{D}(E)$, is a CW-complex whose vertices are labeled by the irreducible components of $E$ and for every stratum $F\subset \cap_{i\in \mathcal{J}}E_i$, we attach a $(|\mathcal{J}|-1)$-dimensional cell. 
 \end{defn}
 
 \begin{defn}[Dual complex of log canonical places] Let $(X,B)$ be an lc pair, $\pi: X\rightarrow Z$ a contraction, and $z\in Z$ a point. Let
     $$e(X/Z\supset z):=\{F\mid F\text{ is a prime divisor over }X, \pi(\Center_{X}F)\supset\bar{z}\}.$$
 \end{defn}
 
 The set of all the log canonical places of $(X,B)$ which belong to $e(X/Z\ni z)$ is denoted by $\LCP(X/Z\supset z,B)$.
    
    \begin{defn}[Regularities and complete regularities]
        Let $(X,B)$ be an lc pair, $X\rightarrow Z$ a contraction, and $z\in Z$ a point. Let $f:Y\to X$ be a dlt modification of $(X,B)$ over a neighborhood of $z$ with prime exceptional divisors $E_1,\ldots,E_k$, and $E:=\sum_{E_i\in e(X/Z\supset z)} E_i.$  The \emph{regularity} of $(X/Z\ni z,B)$ is
        $$\reg(X/Z\ni z,B):=\dim \mathcal{D}(E),$$ 
        where $\dim\mathcal{D}(\emptyset):=-1$. $\reg(X/Z\ni z,B)$ is independent of the choice of $f$, c.f. \cite[7.9]{Sho00}, \cite[Theorem 1]{dFKX17}. The \emph{complete regularity} of $(X/Z\ni z,B)$ is
          \begin{align*}
            \creg(X/Z\ni z,B):=\max\{-\infty,\reg(X/Z\ni z,B^+)\mid (X/Z\ni z,B^+)&\\
            \text{ is an }\Rr\text{-complement of }(X/Z\ni z,B)&\}.
        \end{align*}

                   For simplicity, when $X\to Z$ is the identity map, we may use the notation $\reg(X\ni x,B)$ and $\creg(X\ni x,B)$ instead of $\reg(X/X\ni x,B)$ and $\creg(X/X\ni x,B)$ respectively. 
    \end{defn}

The following lemma follows from the definition of the complete regularities.
    
        \begin{lem}\label{lem: creg and exc}
        Let $(X\ni x,B)$ be an lc germ of dimension $d$. Then 
        \begin{enumerate}
            \item $0\leq \creg(X\ni x,B)\leq d-1$, and
            \item $\creg(X\ni x,B)=0$ if and only if $(X\ni x,B)$ is an exceptional singularity.
        \end{enumerate}
            \end{lem}
    
    \subsubsection{Accumulation points of minimal log discrepancies of exceptional singularities}

    
    
    The following conjecture predicts a relation between the complete regularities and the accumulation points of minimal log discrepancies.
     
    \begin{conj}\label{conj: stronger ACC for mlds} Let $d$ and $r$ be two integers. We define
\begin{align*}
    \mld(d,r):=\{&\mld(X\ni x,B)\mid \dim X=d, (X\ni x,B)\text{ is lc},\\
    &B\in \{1-\frac{1}{n}\mid n\in\mathbb N^+\}\cup\{0\}, \creg(X\ni x,B)\le r\}.
\end{align*}
Then the accumulation points of $\mld(d,r)$ belong to $\mld(d-1,r-1)\cup\{0\}$.
\end{conj}
    
    \begin{thm}\label{thm: strong acc mld for exceptional singularities}
    Conjecture \ref{conj: stronger ACC for mlds} holds when $r\leq 0$.
    \end{thm}
    \begin{proof}
    Since $\mld(d,r)=\emptyset$ for any $r\leq -1$, we only need to show that the only possible accumulation point of $\mld(d,0)$ is $0$.
    
        For any positive integer $d$, assume that $(X\ni x,B)$ is an lc germ of dimension $d$ such that  $B\in\{1-\frac{1}{n}\mid n\in\mathbb N^+\}\cup\{0\}$ and $\creg(X\ni x,B)\leq 0$. By Lemma \ref{lem: creg and exc}, $(X\ni x,B)$ is an exceptional singularity. By  Lemma \ref{lem: exp sing epsilon plt}, there exists a finite set $\Ii_0\subset\{1-\frac{1}{n}\mid n\in\mathbb N^+\}\cup\{0\}$ depending only on $d$ such that $B\in\Ii_0$. Theorem  \ref{thm: strong acc mld for exceptional singularities} follows from Theorem \ref{thm: ACC mld exc}.
        \end{proof}
    
    \subsubsection{ACC for complete regularity thresholds}
\begin{defn}[Complete regularity thresholds and $\Rr$-complementary thresholds]
      Let $c\geq -1$ be an integer, $(X,B)$ a pair, $D\geq 0$ an $\Rr$-Cartier $\Rr$-divisor on $X$, $\pi: X\rightarrow Z$ a contraction, and $z\in Z$ a point.
      
     We define the \emph{c-complete regularity threshold} (crt$_c$ for short) of $D$ with respect to $(X/Z\ni z,B)$ to be
    $$\crt_c(X/Z\ni z,B;D):=\sup\{-\infty,t\mid t\geq 0,\creg(X/Z\ni z,B+tD)\geq c\}.$$
We define the \emph{$\Rr$-complementary threshold} (Rct for short) of $D$ with respect to $(X/Z\ni z,B)$ to be
     $$\Rct(X/Z\ni z,B;D):=\crt_{-1}(X/Z\ni z,B;D).$$
\end{defn}    

\begin{rem}
    We remark that by definition, when $(X/Z\ni z,B)$ is $\Rr$-complementary, 
         $$\Rct(X/Z\ni z,B;D)=\sup\{t\in\Rr\mid (X/Z\ni z,B+tD)\text{ is $\Rr$-complementary}\}.$$
         In particular, when $X\to Z$ is the identity map,
         $$\Rct(X/Z\ni z,B;D)=\lct(X\ni z,B;D).$$
\end{rem}

\begin{defn}\label{defn: set of crt}
Let $d>0$ and $-1\leq c\leq d-1$ be two integers, and $\Ii\subset [0,1]$ and $\Ii'\subset [0,+\infty)$ two sets of real numbers. We define
\begin{align*}
   \CR_{\ge c}(\Ii,\Ii',d):=\{(X/Z\ni z,B;D)\mid \dim X=d, (X,B)\text{ is lc, } B\in\Ii, D\in\Ii',&\\
  X\to Z \text{ is a contraction},z\in Z\text{ is a point}, \creg(X/Z\ni z,B)\geq c\},&
\end{align*}
$$\CRT_c(\Ii,\Ii',d):=\{\crt_c(X/Z\ni z,B;D)\mid (X/Z\ni z,B,D)\in\CR_{\ge c}(\Ii,\Ii',d)\},$$
\begin{align*}
\CRT_{c,\FT}(\Ii,\Ii',d):=\{\crt_c(X/Z\ni z,B;D)\mid (X/Z\ni z,B,D)\in\CR_{\ge c}(\Ii,\Ii',d),&\\
X \text{ is of Fano type over }Z\},&
\end{align*}
$$\RCT(\Ii,\Ii',d):=\CRT_{-1}(\Ii,\Ii',d),$$
and
$$\RCT_{\FT}(\Ii,\Ii',d):=\CRT_{-1,\FT}(\Ii,\Ii',d),$$
\end{defn}

It is expected that both $\CRT_c(\Ii,\Ii',d)$ and $\RCT(\Ii,\Ii',d)$ satisfy the ACC, provided that both $\Ii$ and $\Ii'$ satisfy the DCC:
\begin{conj}\label{conj: ACC for CRT}
Let $d>0$ and $-1\leq c\leq d-1$ be two integers, and $\Ii\subset [0,1]$ and $\Ii'\subset [0,+\infty)$ two DCC sets. Then $\CRT_c(\Ii,\Ii',d)$ satisfies the ACC. In particular, $\RCT(\Ii,\Ii',d)$ satisfies the ACC.
\end{conj}

We will show that Conjecture \ref{conj: ACC for CRT} holds when $X$ is of Fano type over $Z$. 

\begin{thm}\label{thm: crt acc}
Let $d>0$ and $-1\leq c\leq d-1$ be two integers, and $\Ii\subset [0,1]$ and $\Ii'\subset [0,+\infty)$ two DCC sets. Then $\CRT_{c,\FT}(\Ii,\Ii',d)$ satisfies the ACC. In particular, $\RCT_{\FT}(\Ii,\Ii',d)$ satisfies the ACC.
\end{thm}

We need the following proposition, which follows from Theorem \ref{thm: dcc existence n complement} and Theorem \ref{thm: existence ni1i2 complement}.


\begin{prop}\label{thm: dual complex r-n complementary}
Let $d,p$ be two positive integers and $\Ii\subset [0,1]$ a DCC set of real numbers. Then there exist a positive integer $n$ such that $p\mid n$ and a finite set $\Ii_0\subset (0,1]$ depending only on $d,p$ and $\Ii$ satisfying the following. 

Assume that $(X,B)$ is a pair of dimension $d$, $G\geq 0$ is an $\mathbb R$-Cartier $\mathbb R$-divisor on $X$, $X\rightarrow Z$ is a contraction, and $z\in Z$ is a point, such that
\begin{enumerate}
 \item $X$ is of Fano type over $Z$,
\item $B\in\Ii$, and
\item $(X/Z\ni z,B+G)$ is an $\mathbb R$-complement of $(X/Z\ni z,B)$.
\end{enumerate}
Then there exists an $(n,\Ii_0)$-decomposable $\Rr$-complement $(X/Z\ni z,\tilde B)$ of $(X/Z\ni z,B)$ and an $n$-complement $(X/Z\ni z,B^+)$ of $(X/Z\ni z,B)$, such that 
$$\LCP(X/Z\ni z,B+G)\subset\LCP(X/Z\ni z,\tilde B)\cap \LCP(X/Z\ni z,B^+).$$
\end{prop}

\begin{proof}
 Let $f: Y\rightarrow X$ be a dlt modification of $(X,B+G)$ over a neighborhood of $z$, such that
$$K_Y+B_Y+E+G_Y=f^*(K_X+B+G),$$
where $B_Y$ and $G_Y$ are the strict transforms of $B$ and $G$ on $Y$ respectively, and $E$ is the reduced exceptional divisor of $f$. Let $S_Y:=\lfloor B_Y+E+G_Y\rfloor$. We have
$$\LCP(Y/Z\ni z,B_Y+E+G_Y)=\LCP(X/Z\ni z,B+G)=\LCP(Y/Z\ni z,S_Y).$$
Let $B_Y':=B_Y+S_Y-B_Y\wedge S_Y.$ It follows that $B_Y'\in\Ii\cup\{1\}$, $B_Y+E+G_Y\ge B_Y'\ge B_Y+E$, and $(Y/Z\ni z,B_Y')$ is $\Rr$-complementary. 

By Theorem \ref{thm: existence ni1i2 complement} and Theorem \ref{thm: dcc existence n complement}, there exist a positive integer $n$ divisible by $p$ and a finite set $\Ii_0\subset (0,1]$ depending only on $d,p$ and $\Ii$, and two $\mathbb R$-divisors $\tilde B_Y\geq 0$ and $B_Y^+\geq 0$ on $Y$, such that
\begin{itemize}
\item $(Y/Z\ni z,\tilde B_Y)$ is an $(n,\Ii_0)$-decomposable $\Rr$-complement of $(Y/Z\ni z, B_Y')$, and
\item $(Y/Z\ni z, B_Y^+)$ is an $n$-complement of $(Y/Z\ni z, B_Y')$.
\end{itemize}
By construction, $\tilde B_Y\geq S_Y$, $B_Y^+\geq S_Y$.
Thus 
$$\LCP(Y/Z\ni z,S_Y)\subset\LCP(Y/Z\ni z,\tilde B_Y)\cap\LCP(Y/Z\ni z, B_Y^{+}),$$
and we get the desired $\tilde B$ and $B^+$ by letting $\tilde B:=f_*\tilde B_Y$ and $B^+:=f_*B_Y^+$.
\end{proof}

\begin{proof}[Proof of Theorem \ref{thm: crt acc}]
Suppose that the theorem does not hold. Then there exist a sequence of pairs $(X_i,B_i)$ of dimension $d$, contractions $X_i\rightarrow Z_i$, points $z_i\in Z_i$, $\Rr$-Cartier $\Rr$-divisors $D_i\geq 0$ on $X_i$, and a strictly increasing sequence of positive real numbers $t_i$, such that $X_i$ is of Fano type over $Z_i$, $B_i\in\Ii$, $D_i\in\Ii'$, $t_i=\crt_c(X_i/Z_i\ni z_i,B_i;D_i)$, and $t:=\lim_{i\rightarrow+\infty}t_i<+\infty$. Let $t_i'$ be strictly increasing positive real numbers, such that $0\leq t_i'<t_i$ for any $i$, and $\lim_{i\rightarrow+\infty}t_i'=t$.

Then $\creg(X_i/Z_i\ni z_i,B_i+t_i'D_i)\ge c$ for every $i$. In particular, there exists an $\Rr$-Cartier $\Rr$-divisor $G_i\geq 0$ on $X_i$, such that $\reg(X_i/Z_i\ni z_i,B_i+t_i'D_i+G_i)\ge c$ for every $i$. Since the coefficients of $B_i+t_i'D_i$ belong to a DCC set, by Proposition \ref{thm: dual complex r-n complementary}, there exist a positive integer $n$ and a finite set $\Ii_0\subset (0,1]$ depending only on $d$, $\Ii$ and $\Ii'$, and an $(n,\Ii_0)$-decomposable $\Rr$-complement $(X_i/Z_i\ni z_i,\tilde B_i)$ of $(X_i/Z_i\ni z_i,B_i+t_i'D_i)$, such that
 $$\LCP(X_i/Z_i\ni z_i,\tilde B_i)\supset\LCP(X_i/Z_i\ni z_i,B_i+t_i'D_i+G_i).$$
In particular, $\creg(X_i/Z_i\ni z_i,\tilde B_i)\geq\reg(X_i/Z_i\ni z_i,B_i+t_i'D_i+G_i)\ge c$.

By construction, the coefficients of $\tilde B_i$ belong to a finite set, and $\tilde B_i\geq B_i+t_i'D_i+G_i\geq B_i+t_i'D_i$ for every $i$. Thus possibly passing to a subsequence, we may assume that  $\tilde B_i\geq B_i+tD_i$ for every $i$. In particular, $$\creg(X_i/Z_i\ni z_i,B_i+tD_i)\geq\creg(X_i/Z_i\ni z_i,\tilde B_i)\ge c.$$
Thus $\crt_c(X_i/Z_i\ni z_i,B_i;D_i)\geq t>t_i$, a contradiction.
\end{proof}

\begin{rem}
The proof of Theorem \ref{thm: crt acc} also implies that when $X$ is of Fano type over $Z$, $$\crt_c(X/Z\ni z,B;D)=\max\{-\infty, t\mid t\geq 0, \creg(X/Z\ni z,B+tD)\geq c\}.$$
\end{rem}

\begin{cor}\label{cor: r-n complement same regularity}
Let $d>0$ be a positive integer and $\Ii\subset [0,1]$ a DCC set of real numbers. Then there exists a finite set $\Ii_0\subset\bar\Ii$ depending only on $d,c$ and $\Ii$ satisfying the following. 

Assume that $(X,B)$ is a pair of dimension $d$,  $X\rightarrow Z$ is a contraction, and $z\in Z$ is a point, such that
\begin{enumerate}
   \item $X$ is of Fano type over $Z$, 
\item $B\in\Ii$, and
\item $(X/Z\ni z,B)$ is $\Rr$-complementary.
\end{enumerate}
Then there exists an $(n,\Ii_0)$-decomposable $\Rr$-complement $(X/Z\ni z,\tilde B)$ of $(X/Z\ni z,B)$, such that
$$\reg(X/Z\ni z,\tilde B)=\creg(X/Z\ni z,B).$$
\end{cor}

\begin{proof}
Let $c:=\creg(X/Z\ni z,B)$ and $(X/Z\ni z, B+G)$ an $\Rr$-complement of $(X/Z\ni z, B)$ such that $\reg(X/Z\ni z, B+G)=c$. By Proposition \ref{thm: dual complex r-n complementary}, there exist a positive integer $n$ and a finite set $\Ii_0\subset (0,1]$ depending only on $d$ and $\Ii$, and an $(n,\Ii_0)$-decomposable $\Rr$-complement $(X/Z\ni z,\tilde B)$ of $(X/Z\ni z, B)$, such that $\LCP(X/Z\ni z,\tilde B)\supset\LCP(X/Z\ni z,B+G)$. In particular, $\reg(X/Z\ni z,\tilde B)\ge c$. Since $(X/Z\ni z,\tilde B)$ is an $\Rr$-complement of $(X/Z\ni z, B)$, $\reg(X/Z\ni z,\tilde B)\le c.$ Hence $\reg(X/Z\ni z,\tilde B)=c$.
\end{proof}

\subsection{$n$-complements for non-Fano type varieties}

\begin{conj}[{Existence of $n$-complements, \cite[Conjecture 1.1]{CH20}}]\label{conj nonft lc complement}
	Let $d,p$ be two positive integers, and $\Ii \subseteq [0,1]$ a DCC set. Then there exists a positive integer $n$ divisible by $p$ depending only on $d, p$ and $\Gamma$ satisfying the following.
	
	Assume that $(X,B)$ is a pair of dimension $d$, $X\to Z$ a contraction, and $z\in Z$ a point, such that
	\begin{enumerate}     
		\item $B\in \Ii$, and    
		\item $(X/Z\ni z,B)$ is $\Rr$-complementary.     
	\end{enumerate} 
	Then there exists an $n$-complement $(X/Z\ni z,B^{+})$ of $(X/Z\ni z,B).$ Moreover, if $\Span_{\Qq_{\ge0}}(\bar{\Ii}\backslash\Qq)\cap (\Qq\backslash\{0\})=\emptyset$, then we may pick $B^+\ge B$.
\end{conj}

\begin{thm}\label{thm: existence n complement nft q to r}
Let $d$ be a positive integer. Assume that Conjecture \ref{conj nonft lc complement} holds in  dimension $d$ for the case when $\Ii\subset \Qq$ is finite, and Conjecture \ref{conj: exist gmm} holds in dimension $d$. Then Conjecture \ref{conj nonft lc complement} holds in dimension $d$. 
\end{thm}

\begin{proof}
The proof follows from the same argument as in the proof of Theorem \ref{thm: dcc existence n complement}. 
\end{proof}

\subsection{Accumulation points of log canonical thresholds}
As an application of Theorem \ref{thm: dcc existence n complement}, we have the following theorem on accumulation points of log canonical thresholds. 
\begin{thm}\label{thm: accmu of lct}
	Let $d$ be a positive integer, and $\Ii\subset [0,1]$ and $\Ii'\subset [0,+\infty)$ two DCC sets. Then the accumulation points of
	$\LCT(\Ii,\Ii',d)$ belong to the set
	\begin{align*}
	&\{t\ge0 \mid \Span_{\Qq_{\ge0}}\left(((\bar{\Ii}+t\bar{\Ii'})\backslash\Qq)\cap[0,1]\right)\cap \Qq\backslash\{0\}\neq\emptyset\}\\
	\cup &\{t\ge0\mid (\bar{\Ii}+t\bar{\Ii'}_{>0})\cap [0,1]\cap\Qq\neq \emptyset\}\cup\{0\},
	\end{align*}
	where $\bar{\Ii'}_{>0}:=\bar{\Ii'}\cap(0,+\infty)$. In particular, when $\bar{\Ii}\subset [0,1]\cap\Qq$ and $\bar{\Ii}'\subset\Qq$, the accumulation points of
	$\LCT(\Ii,\Ii',d)$ are rational numbers.
\end{thm}
\begin{proof}
We may assume that $0\in\Ii\cap\Ii'$ and $1\in\Ii$. Let $t>0$ be an accumulation point of $\LCT(\Ii,\Ii',d)$. There exist a sequence of lc pairs $(X_i,B_i)$ of dimension $d$, $\Rr$-Cartier $\Rr$-divisors $G_i\ge0$, strictly decreasing real numbers $t_i\ge0$, such that $t_i=\lct(X_i,B_i;G_i)$ and $t:=\lim_{i\to+\infty} t_i$. 

Let $f_i:Y_i\to X_i$ be a dlt modification of $(X_i,B_i)$, $G_{Y_i}$ the strict transform of $G_i$ on $Y_i$, and $$K_{Y_i}+B_{Y_i}:=f^{*}_{i}(K_{X_i}+B_i).$$
Then $B_{Y_i}\in\Ii$, and $t_i=\lct(Y_i,B_{Y_i};G_{Y_i})$. Thus replacing $(X_i,B_i)$ with $(Y_i,B_{Y_i})$ and $G_i$ with $G_{Y_i}$, we may assume that $X_i$ is $\Qq$-factorial klt. There exists a closed point $x_i\in X_i$, such that $t_i=\lct(X_i\ni x_i,B_i;G_i)$ for any $i$.

Suppose that $\Span_{\Qq_{\ge0}}\left(((\bar{\Ii}+t\bar{\Ii'})\backslash\Qq)\cap[0,1]\right)\cap \Qq\backslash\{0\}=\emptyset$, and $(\bar{\Ii}+t\bar{\Ii'}_{>0})\cap [0,1]\cap\Qq=\emptyset$. Since the coefficients of $B_i+tG_i$ belong to $\Ii+t\Ii'$ for any $i$, by Theorem \ref{thm: dcc existence n complement}, there exist a positive integer $n$ depending only on $d,\Ii$ and $\Ii'$, and $\Rr$-divisors $B_i^{+}$ on $X_i$, such that $(X_i\ni x_i,B_i^{+})$ is a monotonic $n$-complement of $(X_i\ni x_i,B_i+tG_i)$ for any $i$. Possibly shrinking $X_i$ to a neighborhood of $x_i$, we may assume that there exists a positive integer $m$, such that $B_i^{+}=\sum_{j=1}^m b_{i,j}^{+}B_{i,j}$, $B_i=\sum_{j=1}^m b_{i,j}B_{i,j}$, and $G_i=\sum_{j=1}^m g_{i,j}B_{i,j}$, where $nb_{i,j}^{+}\in[0,1]$, and $B_{i,j}$ are distinct prime divisors on $X_i$. Since $t>0$, $g_{i,j}\le \frac{1}{t}$. Possibly passing to a subsequence, we may assume that $\{b_{i,j}\}_{i=1}^{\infty}$, $\{g_{i,j}\}_{i=1}^{\infty}$ are increasing sequences for every $1\leq j\leq m$, and there exist $b_j^{+}$, $b_i\in\bar{\Ii}$, and $g_i\in\bar{\Ii'}$ such that $b_j^{+}=b_{i,j}^{+}$ for any $i,j$, $b_j=\lim_{i\to+\infty}b_{i,j}$, $g_j=\lim_{i\to+\infty} g_{i,j}$, and $b_j^{+}\ge b_j+tg_j$ for any $j$. By the construction of $t$, if $g_j\neq 0$, then $b_j^{+}\neq b_j+tg_j$, and $t<\frac{b_{j}^{+}-b_j}{g_j}$. Thus possibly passing to a subsequence, we may assume that there exists $t_i'>t_i$, such that $b_j^{+}> b_j+t_i'g_j\ge b_{i,j}+t_i'g_{ij}$ for any $g_j\neq 0$ and any $i$. Hence $\lct(X_i,B_i;G_i)\ge t_i'>t_i$ near a neighborhood of $x_i$, a contradiction.
\end{proof}


\noindent\textbf{Acknowledgments.} This work began when the first author visited the second author at the University of Utah in November of 2017. The first author would like to thank their hospitality. We are grateful to Yuchen Liu who is so kind to send us his note (Appendix A). The first and the second author would like to thank Christopher Hacon and Chenyang Xu for their constant supports and encouragements. The authors would like to thank Caucher Birkar, Guodu Chen, Chen Jiang, Yuchen Liu, Zipei Nie, Yuri G. Prokhorov, Lu Qi, Ruixiang Zhang, Zili Zhang, Ziwen Zhu, and Ziquan Zhuang for useful discussions. The first author was partially supported by AMS Simons Travel Grant. The second author was partially supported by NSF research grants no: DMS-1300750, DMS-1265285 and by a grant from the Simons Foundation; Award Number: 256202.

\appendix

\section{Index of a plt divisor}
\begin{center}
	by Yuchen Liu\footnote[1]{Y.L. would like to thank Chenyang Xu for fruitful discussions. This material is based upon work supported by the National Science Foundation under Grant No. DMS-1440140 while Y.L. was in residence at the Mathematical Sciences Research Institute in Berkeley, California, during the Spring 2019 semester.}\\
	Department of Mathematics, Yale University,\\
	New Haven, CT 06511, USA,\\
	E-mail: yuchen.liu@yale.edu
\end{center}
\medskip

For any $\bQ$-Cartier $\bQ$-divisor $D$
on a normal variety $Y$, we denote the Cartier index of $D$ at 
a point $y\in Y$ by $\ind(y\in Y,D):=\min\{m\in\bZ_{>0}\mid mD\textrm{ is
	Cartier}\}$. 

\begin{thm}\label{thm:pltind}
	Let $(Y,\Delta+E)$ be a plt pair over an algebraically closed field $k$
	of characteristic zero
	where $E$ is the integral part. Then for any $\bQ$-Cartier Weil divisor
	$L$ on $Y$ and any point $y\in E$, we have 
	\[
	\ind(y\in Y, L)=\ind(y\in E,L|_E).
	\]
	In particular, we have $\ind(y\in Y,E)=\ind(y\in E,E|_E)$.
\end{thm}

Before proving Theorem \ref{thm:pltind}, we will work out restriction exact sequence
for plt surface pairs. 

\begin{notation} Let $y\in Y$ be a normal excellent surface
	singularity. Let $E$ be a reduced irreducible curve on $Y$ passing through $y$.
	Assume that $(Y,E)$ is a plt pair. Then from \cite[Corollary 3.31]{Kol13},
	$E$ is regular at $y$, and the extended dual graph
	$(\Gamma,E)$ of $(y\in Y,E)$ is of cyclic quotient type:

	
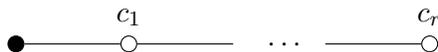
\begin{center}
		\begin{tikzpicture}[scale=1, baseline=(current bounding box.center)]
		\node[fill,circle,draw,inner sep=0pt,minimum size=6pt,label={}] (a) at (0,0){};
		\node[circle,draw,inner sep=0pt,minimum size=6pt,label={$c_1$}] (b) at (1.5,0){};
		\node[circle,inner sep=0pt,minimum size=6pt,label={}] (c) at (3,0){};
		\node[circle,inner sep=0pt,minimum size=6pt,label={},label={[shift={(-0.4,-0.35)}]$\cdots$}] (d) at (4,0){};
		\node[circle,draw,inner sep=0pt,minimum size=6pt,label={$c_r$}] (e) at (5.5,0){};
		
		\draw (a) edge (b);
		\draw (b) edge (c);
		\draw (d) edge (e);
		\end{tikzpicture}
	\end{center}
	Denote by $\sigma:\oY\to Y$ the minimal  
	resolution of $(Y,E)$.
	Here $\bullet$ represents the curve $\oE:=\sigma_*^{-1}E$, and $\circ$
	represent exceptional curves $C_i\cong\bP^1_{\kappa(y)}$ 
	with $-(C_i^2)=c_i\geq 2$ for $1\leq i\leq r$.
	
	Denote by $\Gamma_i$ and $\Gamma_i'$ the subgraph of $\Gamma$ spanned
	by $\{C_1,\cdots,C_i\}$ and $\{C_{r-i+1},\-\cdots,C_r\}$
	respectively. Let $d_i:=\det(\Gamma_i)$ and $d_i':=\det(\Gamma_i')$,
	with $d_0=d_0'=1$ by convention. Let $m:=\det(\Gamma)=d_r=d_r'$.
	From \cite[3.35]{Kol13}, we know that
	\begin{equation}\label{eq:discrep}
	\sigma^*E = \oE+\sum_{i=1}^r \frac{d_{r-i}'}{m} C_i,\quad
	\sigma^*(K_Y+E) = K_{\oY}+\oE+\sum_{i=1}^r \left(1-\frac{d_{i-1}}{m}\right)C_i.
	\end{equation}
	
\end{notation}

\begin{defn}\label{defn:pltsurf}
	With notation ($\star$), let $L$ be a Weil divisor on $Y$ whose support does not contain $E$.
	We define $L|_E$ to be the $\bQ$-divisor $\sigma_*((\sigma^*L)|_{\oE})$. We
	define $E|_E$ as the $\bQ$-divisor class (i.e. a $\bQ$-divisor unique
	up to $\bZ$-linear equivalence) $\sigma_*((\sigma^*E)|_{\oE})=\sigma_*(\oE|_{\oE})+
	\frac{d_{r-1}'}{m}y$.
\end{defn}

The following lemma is probably well-known to experts. We provide a proof 
for readers' convenience.
\begin{lem}\label{lem:pltres}
	With notation $(\star)$, for any Weil divisor $L$ on $Y$, there are canonical isomorphisms 
	\[
	(\cO_Y(L)\otimes\cO_E)^{**}\cong \cO_Y(L)/\cO_Y(L-E)\cong
	\cO_E(\lfloor L|_E \rfloor).
	\]
\end{lem}

\begin{proof}
	We first prove the second isomorphism.
	By \cite[Proposition 3.1]{HW19} and \cite[Theorem 10.4]{Kol13}, for any Weil
	divisor $D$ on $Y$, we have the following short 
	exact sequence
	\[
	0\to\cO_{Y}(K_Y+D)\to\cO_Y(K_Y+E+D)\to\sigma_*
	(\cO_{\oY}(K_{\oY}+\oE+\lceil\sigma^* D\rceil)
	\otimes\cO_{\oE})
	\to 0.
	\]
	Denote the last non-zero term in the above exact
	sequence by $\cF$.
	For a Weil divisor $L$ on $Y$, we take $D:=L-K_Y-E$. 
	Then we have
	a canonical isomorphism between $\cO_Y(L)/\cO_Y(L-E)$ and $\cF$,
	where 
	\[
	\cF=\sigma_*(\cO_{\oY}(K_{\oY}+\oE+\lceil\sigma^* (L-K_Y-E)\rceil)
	\otimes\cO_{\oE})
	\]

	From \eqref{eq:discrep}, we know that 
	\[
	\sigma^*(-K_Y-E)=-K_{\oY}-\oE-
	\sum_{i=1}^r\left(1-\frac{d_{i-1}}{m}\right)C_i.
	\]
	Denote by $\oL:=\sigma_*^{-1}L$. Then we have $\sigma^*L=\oL+\sum_{i=1}^r \frac{l_i}{m} C_i$.
	Hence 
	\[
	K_{\oY}+\oE+\lceil\sigma^* (L-K_Y-E)\rceil=\oL-
	\sum_{i=1}^r\left\lfloor-\frac{l_i}{m}+1-\frac{d_{i-1}}{m}\right\rfloor C_i.
	\]
	Since
	$C_i\cap \oE=\emptyset$ for any $i\geq 2$, we know 
	that 
	\[
	\cF  =\sigma_*(\cO_{\oE}\otimes\cO_{\oY}(\oL+\lfloor l_1/m\rfloor C_1))
	=\sigma_*(\cO_{\oE}\otimes\cO_{\oY}(\lfloor\sigma^*L\rfloor)).
	\]
	Here we use the equality $\lfloor -\frac{l_1}{m}+1-\frac{1}{m}\rfloor =-\lfloor\frac{l_1}{m}\rfloor$.
	Thus by definition  we have $\cF=\cO_E(\lfloor L|_E \rfloor)$.

	For the first isomorphism, consider the following diagram:
	
	\[
	\begin{tikzcd}
	0\arrow{r}& \cI_E\cdot\cO_Y(L)\arrow[hookrightarrow]{d}\arrow{r}&
	\cO_Y(L)\arrow{d}{\id}\arrow{r} & \cO_E\otimes\cO_Y(L)\arrow[twoheadrightarrow]{d}{f}\arrow{r} & 0\\
	0\arrow{r}& \cO_{Y}(L-E)\arrow{r}& \cO_Y(L)\arrow{r} &\cO_Y(L)/\cO_Y(L-E)\arrow{r} & 0
	\end{tikzcd}
	\]
	It is clear that the vertical morphisms are isomorphic over a non-empty Zariski open
	subset. Since $\cO_Y(L)/\cO_Y(L-E)$ is torsion free, 
	the morphism $f$ induces a surjective morphism $(\cO_Y(L)\otimes\cO_E)^{**}\twoheadrightarrow
	\cO_Y(L)/\cO_Y(L-E)$ between invertible sheaves on $E$ which has to be an isomorphism.
\end{proof}

\begin{defn}\label{defn:pltrest}
	Let $(Y,E+\Delta)$ be a plt pair over an algebraically closed field $k$
	of characteristic zero where $E$ is the integral part. Let $L$ be a $\bQ$-Cartier Weil divisor 
	on $Y$. We define $L|_E$ as the $\bQ$-divisor class on $E$ determined by
	Definition \ref{defn:pltsurf} at every codimension $1$ point $\zeta$ of $E$.
\end{defn}

\begin{proof}[Proof of Theorem \ref{thm:pltind}]
	Let $m:=\ind(y\in Y,L)$, then $mL|_E$ is Cartier which implies that
	$m$ is a multiple of $\ind(y\in E,L|_E)$. Let us consider the index $1$ cover
	$\pi:\tilde{Y}\to Y$ of $L$. After shrinking $Y$ if necessary we may assume that  $Y$ is affine and $mL$ is trivial. 
	Then we can  take a generating section $s\in H^0(Y,mL)$ such that
	\[
	\tilde{Y}\cong\Spec_Y\oplus_{i=0}^{m-1}\cO_Y(-iL),
	\]
	where the ring structure is defined using $s$. Let $\tilde{E}:=\pi^*E$ and $\tilde{\Delta}:=\pi^*\Delta$.
	Then \cite[Proposition 5.20]{KM98} implies that $(\tilde{Y},\tilde{E}+\tilde{\Delta})$ is 
	plt, which in particular implies that $\tilde{E}$ is normal.
	Since $\pi^{-1}(y)$ is a single point, we know that $\tilde{E}$ is irreducible.
	Denote by $\hat{E}:=\Spec_E\oplus_{i=0}^{m-1}\cO_E(\lfloor -iL|_E\rfloor)$
	where the ring structure is defined using $s|_E$.
	By Lemma \ref{lem:pltres}, we have 
	\begin{equation}\label{eq:hatE}
	\hat{E}\xrightarrow{~\cong~}\Spec\oplus_{i=0}^{m-1}(\cO_E\otimes\cO_Y(-iL))^{**}
	\rightarrow E\times_Y\tilde{Y}.
	\end{equation}
	Notice that the first morphism is an isomorphism because it is an
	isomorphism over all codimension $1$ points on $E$ by Lemma \ref{lem:pltres},
	and both structure sheaves are reflexive.
	We denote the composition morphism by $g$. Then $g$ is finite and  isomorphic over $E^\circ\times_Y \tilde{Y}$ where
	$E^{\circ}\subset E$ is the Cartier locus of $E$ in $Y$. 
	Since $\tilde{E}$ is the 
	reduced scheme of $E\times_Y \tilde{Y}$, we may lift $g$ to 
	a finite morphism $h:\hat{E}\to \tilde{E}$
	which is isomorphic over $\pi^{-1}(E^\circ)$. By definition we know that
	$\hat{E}$ is equidimensional, hence $h$ is an isomorphism by the Zariski main theorem.
	Clearly, the fiber of $\hat{E}\to E$ over $y$ has cardinality $m/\ind(y\in E,L|_E)$.
	Thus $m=\ind(y\in E,L|_E)$ and the proof is finished. 
\end{proof}


\begin{thebibliography}{99}
	\bibitem[Ale93]{Ale93} V. Alexeev, \textit{Two two--dimensional terminations}. Duke Math. J., 69(3), 1993: 527--545.

	
	\bibitem[Amb06]{Amb06} F. Ambro, \textit{The set of toric minimal log discrepancies}. Cent. Eur. J. Math. 4, no. 3, 358--370, 2006.
		
		
	\bibitem[BCHM10]{BCHM10}
	C. Birkar, P. Cascini, C.D. Hacon and J. M\textsuperscript{c}Kernan,  \textit{Existence of minimal models for varieties of log general type}. J. Amer. Math. Soc. {\bf 23} (2010), no. 2, 405--468.	
	
		\bibitem[Bir16]{Bir16} C. Birkar, \textit{Singularities of linear systems and boundedness of Fano varieties}. arXiv: 1609.05543v1, 2016.
	
	\bibitem[Bir19]{Bir19} C. Birkar, \textit{Anti-pluricanonical systems on Fano varieties}. Ann. of Math. (2), 190(2):345--463, 2019.
	


	\bibitem[Bor97]{Bor97} A.A. Borisov, \textit{Minimal discrepancies of toric singularities}. Manuscripta Mathematica 92(1), 1997: 33--45.
	
		
	\bibitem[BS03]{BS03} A.A. Borisov and V.V. Shokurov, \textit{Directional rational approximations with some applications to algebraic geometry}. Tr. Mat. Inst. Steklova, 240, 73--81, 2003.
	
	\bibitem[BS10]{BS10} C. Birkar and V.V. Shokurov, \textit{Mld's vs thresholds and flips}. J. Reine Angew. Math. \textbf{638} (2010), 209--234.
	
	
	\bibitem[CH20]{CH20} G. Chen and J. Han, \textit{Boundedness of $(\epsilon,n)$-complements for surfaces}. arXiv: 2002.02246v1.
	
	\bibitem[CDHJS18]{CDHJS18} W. Chen, G. Di Cerbo, J. Han, C. Jiang and R. Svaldi, \textit{Birational boundedness of rationally connected
	Calabi-Yau 3-folds}. arXiv: 1804.09127.

	
	\bibitem[CPS10]{CPS10} I. Cheltsov, J. Park and and C. Shramov, \textit{Exceptional del Pezzo hypersurface}. J. Geom. Anal. 20:4 (2010) 787--816.
	
	\bibitem[CS11]{CS11} I. Cheltsov and C. Shramov, \textit{On exceptional quotient singularities}. Geom. Topol. 15, no. 4 (2011): 1843--1882.
	
	\bibitem[CS14]{CS14} I. Cheltsov and C. Shramov, \textit{Weakly--exceptional singularities in higher dimensions}. J. Reine Angew Math., no. 689 (2014): 201--241.
	


\bibitem[dFH11]{dFH11} T. de Fernex and C.D. Hacon, \textit{Deformations of canonical pairs and Fano varieties}. J. Reine Angew Math., 2011(651): 97--126.

\bibitem[dFKX17]{dFKX17} T. de Fernex, J. Koll\'ar, and C. Xu, \textit{The dual complex of singularities}. Higher dimensional algebraic geometry. Adv. Stud. Pure Math. \textbf{74}, 103–130 (2017)
	
\bibitem[DS16]{DS16} G. Di Cerbo, R. Svaldi, \textit{Birational boundedness of low dimensional elliptic Calabi-Yau varieties with a section}. arXiv:1608.02997.
	

	
		\bibitem[Has17]{Has17} K. Hashizume, \textit{Non-vanishing theorem for lc pairs admitting a Calabi--Yau pair}. arXiv: 1708.01851, 2017.
		
		\bibitem[HLM19]{HLM19} J. Han, J. Liu, and J. Moraga, \textit{Boundedness of $(\epsilon,\delta)$-lc singularities}. To appear in J. Math. Sci. Univ. Tokyo. arXiv: 1903.07202, 2019.
		
		
	\bibitem[HLQ17]{HLQ17} J. Han, Z. Li, and L. Qi. \textit{ACC for log canonical threshold polytopes}. arXiv:1706.07628, 2017.
	
	\bibitem[HLQ20]{HLQ20} J. Han, Y. Liu, and L. Qi.
	\textit{On local volumes and boundedness of singularities}. Preprint, 2020.
	
	\bibitem[HMX14]{HMX14} C.D. Hacon, J. M\textsuperscript{c}Kernan, and C. Xu, \textit{ACC for log canonical thresholds}. Ann. of Math. (2) \textbf{180} (2014), no. 2, 523--571.
	
	\bibitem[HW19]{HW19} C.D. Hacon and J. Witaszek:
	\textit{ On the rationality of Kawamata log terminal singularities in positive characteristic}. Algebr. Geom. 2019, \textbf{6}, no.5, 516--529.
	
	\bibitem[HX15]{HX15} C.D. Hacon and C. Xu, \textit{Boundedness of log Calabi--Yau pairs of Fano type}. Math. Res. Lett. {\bf 22} (2015), no. 6, 1699--1716.
	
	
	\bibitem[IP01]{IP01} S. Ishii and Y.G. Prokhorov, \textit{Hypersurface exceptional singularities}. Internat. J. Math. 12(6), 2001: 661--687.
	
	\bibitem[Jia18]{Jia18} C. Jiang, \textit{On birational boundedness of {F}ano fibrations}. Amer. J. Math., 140(5), 2018: 1253--1276.

    \bibitem[Jia19]{Jia19} C. Jiang, \textit{A gap theorem for minimal log discrepancies of non-canonical singularities in dimension three}. To appear in J. Algebraic Geom., arXiv: 1904.09642v1, 2019.

	\bibitem[Kaw07]{Kaw07} M. Kawakita,  \textit{Inversion of adjunction and log canonicity}. Amer. J. Math., 133(5), 2011: 1299--1311.
	
	\bibitem[Kaw11]{Kaw11} M. Kawakita, \textit{Towards boundedness of minimal log discrepancies by the Riemann--Roch theorem}. Amer. J. Math., 133(5), 2011: 1299--1311.
	
	\bibitem[Kaw14]{Kaw14} M. Kawakita, \textit{Discreteness of log discrepancies over log canonical triples on a fixed pair}, J. Algebraic Geom., 23(4), 2014: 765--774.
	
	\bibitem[Kaw18]{Kaw18} M. Kawakita, \textit{On equivalent conjectures for minimal log discrepancies on smooth threefolds}. To appear in J. Algebraic Geom., arXiv: 1803.02539, 2018.
	
		
	\bibitem[KM92]{KM92} J. Koll\'{a}r and Shigefumi Mori, \textit{Classification of three--dimensional flips}. J. Amer. Math. Soc., 1992 5(3):533--703.
	
			\bibitem[KM98]{KM98} J. Koll\'{a}r and Shigefumi Mori, \textit{Birational geometry of algebraic varieties}. Cambridge Tracts in Math. \textbf{134}, Cambridge Univ. Press, 1998.
	
	
	\bibitem[KMM87]{KMM87} Y. Kawamata, K. Matsuda, and K. Matsuki. \textit{Introduction to the minimal model problem}. In Algebraic geometry (Sendai, 1985), Adv. Stud. Pure Math., no 10, North-Holland, Amsterdam (1987) 283--360.
	
	

	\bibitem[Kol$^+$92]{Kol92} J. Koll\'{a}r \'{e}t al., \textit{Flip and abundance for algebraic threefolds}. Ast\'{e}risque No. \textbf{211}, 1992.
	
	\bibitem[Kol13]{Kol13} J. Koll\'ar, {\it
		Singularities of the minimal model program}. With a collaboration of S\'andor Kov\'acs.
	Cambridge Tracts in Math. \textbf{200}. Cambridge Univ. Press, 2013.
	
	\bibitem[Kud01]{Kud01} S.A. Kudryavtsev, \textit{Pure log terminal blow-ups}. Math. Notes, 69(5), 814--819, 2001.
	
	\bibitem[Kud02]{Kud02} S.A. Kudryavtsev, \textit{Classification of three--dimensional exceptional log canonical hypersurface singularities. {I}}. Izv. Ross. Akad. Nauk Ser. Mat., 66(5), 2002, 949--1034.
	
		\bibitem[Li17]{Li17} C. Li, \textit{K-semistability is equivariant volume minimization}. Duke Math. J. \textbf{166}, no. 16, 3147--3218, 2017.
	

	\bibitem[Li18]{Li18} C. Li, \textit{Minimizing normalized volumes of valuations}. Math. Zeit. \textbf{289}, no. 1-2, 491--513, 2018.
	
		\bibitem[Liu18]{Liu18} J. Liu, \textit{Toward the equivalence of the ACC for a-log canonical thresholds and the ACC for minimal log discrepancies}. arXiv: 1809.04839, 2018.
		
	\bibitem[LLX18]{LLX18} C. Li, Y. Liu, and C. Xu, \textit{A guided tour to normalized volume}. arXiv: 1806.07112, 2018.
	

	\bibitem[LX16]{LX16} C. Li and C. Xu, \textit{Stability of Valuations and Koll\'{a}r Components}. To appear in J. Eur. Math. Soc. (JEMS), arXiv: 1604.05398v4, 2016.
	
	
		
	
		
			
	\bibitem[Mor18]{Mor18} J. Moraga, \textit{On minimal log discrepancies and Koll\'{a}r components}. arXiv:1810.10137, 2018.
		
		
	\bibitem[MP99]{MP99} D. Markushevich and Y.G. Prokhorov, \textit{Exceptional quotient singularities}. Amer. J. Math., \textbf{121}, 1179--1189, 1999.
	


	\bibitem[Nak16]{Nak16} Y. Nakamura, \textit{On minimal log discrepancies on varieties with fixed Gorenstein index}. Michigan Math. J., Volume 65, Issue 1, 165--187, 2016.

	
	\bibitem[Pro00]{Pro00} Y.G. Prokhorov, \textit{Blow-ups of canonical singularities}. Algebra (Moscow, 1998) (2000): 301--318.
	
	\bibitem[PS01]{PS01} Y.G. Prokhorov and V.V. Shokurov, \textit{The first main theorem on complements: from global to local}. Izv. Ross. Akad. Nauk Ser. Mat., 65(6), 1169--1196, 2001.
	
	\bibitem[PS09]{PS09} Y.G. Prokhorov and V.V. Shokurov, \textit{Towards the second main theorem on complements}. J. Algebraic Geom., 18(1), 2009: 151--199.
	
	\bibitem[Sak12]{Sak12} D. Sakovics, \textit{Weakly--exceptional quotient singularities}. Cent. Eur. J. Math. 10(3), 2012: 885--902.
	
	\bibitem[Sak14]{Sak14} D. Sakovics, \textit{Five-dimensional weakly exceptional quotient singularities}. Proc. Edinb. Math. Soc. (2), 57(1), 2014: 269--279.
	
	\bibitem[Sho88]{Sho88}
	V.V. Shokurov, {\it Problems about {F}ano varieties}. {Birational Geometry of Algebraic Varieties, Open Problems. The
		XXIIIrd International Symposium, Division of Mathematics, The Taniguchi
		Foundation}, 30--32, August 22--August 27, 1988.
		
	\bibitem[Sho91]{Sho91} V.V. Shokurov, \textit{A.c.c. in codimension 2}. 1991 (preprint).
	
	\bibitem[Sho92]{Sho92} V.V. Shokurov, \textit{3--fold log flips}. Izv. Ross. Akad. Nauk Ser. Mat., 56:105--203, 1992. 
	
	\bibitem[Sho96]{Sho96} V.V. Shokurov, \textit{3--fold log models}. J. Math. Sciences \textbf{81}, 2677--2699, 1996.
	
	\bibitem[Sho00]{Sho00} V.V. Shokurov, \textit{Complements on surfaces}. J. Math. Sci. (New York), 102, no. 2 (2000): 3876--3932.
	
	\bibitem[Sho04]{Sho04} V.V. Shokurov, \textit{Letters of a bi-rationalist, V. Minimal log discrepancies and termination of log flips}. (Russian) Tr. Mat. Inst. Steklova 246, Algebr. Geom. Metody, Svyazi i Prilozh., 328--351, 2004.			
	
	\bibitem[Sho20]{Sho20} V.V. Shokurov, \textit{Existence and boundedness of $n$-complements}. 2019 (preprint to appear).
	
	
	\bibitem[Tia97]{Tia97} G. Tian, \textit{K\"{a}hler--Einstein metrics with positive scalar curvature}. Inv. Math. \textbf{130}, 239--265, 1997.
	
	\bibitem[Xu14]{Xu14} C. Xu, \textit{Finiteness of algebraic fundamental groups}. Compos. Math. \textbf{150}, no. 3, 409--414, 2014.
\end{thebibliography}
\end{document}